\documentclass{amsart}
\usepackage{amsmath}
\usepackage{amsfonts}
\usepackage{stmaryrd}
\usepackage{amssymb}
\usepackage{graphicx}
\usepackage[all]{xy}
\usepackage{pdfsync}
\usepackage{enumerate}
\usepackage{hyperref}
\usepackage{multirow} % Required for multirows
\usepackage{tensor} %required for double fiber products
\usepackage{xcolor}

\usepackage[english]{babel}
\usepackage[utf8]{inputenc}
\usepackage[T1]{fontenc}

\usepackage{tikz-cd}
\usepackage[dvips]{geometry}

\newtheorem{theorem}{Theorem}[section]
\theoremstyle{definition}
\newtheorem{definition}{Definition}[section]

\theoremstyle{corollary}
\newtheorem{corollary}[theorem]{Corollary}

\newtheorem{example}[theorem]{Example}
\newtheorem{examples}[theorem]{Examples}

\newtheorem{lemma}[theorem]{Lemma}
\theoremstyle{proposition}
\newtheorem{proposition}[theorem]{Proposition}

\theoremstyle{remark}
\newtheorem{remark}[theorem]{Remark}

\newenvironment{sketchproof}[1][Sketch of proof]{\noindent\textit{#1.\;}}{\ \hfill$\lozenge$}

\def\e{{\varepsilon}}
\def\l{{\lambda}}
\def\i{{\iota}}

\def\calG{{\mathcal{G}}}
\def\calH{{\mathcal{H}}}
\def\calK{{\mathcal{K}}}
\def\calO{{\mathcal{O}}}
\def\V{{\mathcal{V}}}

\def\frakk{{\mathfrak{k}}}
\def\frakl{{\mathfrak{l}}}

\newcommand{\Rr}{\mathbb R}

\newcommand{\G}{\mathcal{G}}            % Lie groupoid
\renewcommand{\H}{\mathcal{H}}          % Lie subgroupoid
\newcommand{\Ker}{\text{\rm Ker}\,}     % Kernel
\newcommand{\tto}{\rightrightarrows}    % Arrows of a groupoid
\tikzset{commutative diagrams/.cd,
mysymbol/.style={start anchor=center,end anchor=center,draw=none} }

\newcommand{\nocontentsline}[3]{}
\newcommand{\tocless}[2]{\bgroup\let\addcontentsline=\nocontentsline#1{#2}\egroup}
\newcommand{\Addresses}{{% additional braces for segregating \footnotesize
  \bigskip
  \footnotesize

  \textsc{Universidade Federal Fluminense, Instituto de Matemática e Estatística, Rua Prof.
Marcos Waldemar de Freitas Reis, S/n, 24210-201, Niterói, RJ, Brazil}\par\nopagebreak
  \textit{E-mail address}, C.~C.~C\'ardenas: \texttt{ccardenascrist@gmail.com}

}}

\begin{document}

\title{Deformations of Lie groupoid morphisms}
 
\author{Cristian Camilo C\'ardenas}

%-----------------------------------------------------------------------
% If your printer does not reproduce dimensions exactly, it may be
% necessary to remove the % signs and adjust the dimensions in the
% following commands:
%
%     \setlength{\textheight}{24cm}
%     \setlength{\textwidth}{16cm}
%
% Similarly for the following, if you need to adjust the positioning
% on the paper:
%
%     \setlength{\topmargin}{-1cm}
%     \setlength{\oddsidemargin}{0pt}
%     \setlength{\evensidemargin}{0pt}
%------------------------------------------------------------------------ 
\maketitle % this produces the title block
 
\begin{abstract}
We establish the deformation theory of Lie groupoid morphisms, describe the corresponding deformation cohomology of morphisms, and show the properties of the cohomology. We prove its invariance under isomorphisms of morphisms. Additionally, we establish stability properties of the morphisms using Moser-type arguments. Furthermore, we demonstrate the Morita invariance of the deformation cohomology and consider simultaneous deformations of the morphism, its domain, and codomain. These simultaneous deformations are utilized to define cohomology for generalized morphisms and to study deformations of multiplicative forms on Lie groupoids.
\end{abstract}

\tableofcontents    % imprime o sumário

\section{Introduction}

Deformation theory constitutes a distinct research area in mathematics, with numerous noteworthy works emerging in connection with various branches of mathematics and physics such as algebraic geometry, quantum mechanics, complex geometry, and algebra. As defined in Kontsevich's lectures, deformation theory is the infinitesimal study of moduli spaces. Specifically, it involves the infinitesimal study of families of structures, which we call deformations, around a given structure. This infinitesimal study yields tangent vectors to the moduli spaces, measuring the direction of each family. Typically, these vectors are elements of low degree (1 or 2) for a specific cohomology, known as the \emph{deformation cohomology} of that structure. Almost every mathematical object possesses its own deformation theory. Examples include Lie algebras, Lie subalgebras, associative algebras, and algebra homomorphisms, each exhibiting a rich deformation theory \cite{richardson1969deformationssubalg.}, \cite{gerstenhaber1964deformation}, \cite{nijenhuisrichardson1967deformationshomom}. Some works in the realm of differential geometry span topics ranging from complex manifolds, foliations, G-structures, and pseudogroup structures to Lie groups, Lie algebroids, and Lie groupoids, among others \cite{kodairaspencer1958deformations}, %complex manifolds,
\cite{2001deformationtransvhomogfoliations}, %foliations, 
\cite{2022deformationssympfoliationsZambon}, %symplectic foliations,
\cite{1966deformationpseudogstruc}, %pseudogroups,
\cite{griffiths1964deformationsGstructures}, %G-struc,
\cite{2015methoddeformingGstructBunk}, %G-structures,
\cite{blaom2006geometricstructdeformedinfinitesymmeties}.   %G-struc deformed inf symmetries.
Notably, deformations of the latter two Lie objects enable us to address the deformations of many well-known geometric structures, including foliations, Lie groups, Lie group actions, and Poisson structures, as evidenced in  \cite{CM}, \cite{CMS} and \cite{deformationscptfoliat}. This is not coincidental; indeed, Lie algebroids and Lie groupoids have recently attracted significant attention for their role in codifying various geometric structures. Moreover, they boast numerous connections with physics-related topics such as symplectic foliations, Poisson structures, Dirac structures, quantization, non-commutative geometry, and more \cite{M}, \cite{MM}, \cite{BCWZ}, \cite{cabrera2022generating-semiquantizat}, \cite{crainic2011lecturesintegrability}, \cite{crainic2004integrability}.

In general, the philosophy of Deformation Theory asserts that deformation cohomology arises from either a DGLA structure or a $L_{\infty}$-algebra structure, as exemplified in \cite{Zambon-Schatz2013}, \cite{oh-park2005}. In this context, their Maurer-Cartan elements correspond to structures nearby to the initial structure undergoing deformation. The study of deformations, therefore, allows us to understand the behaviour of structures around a given one.

However, obtaining the algebraic structure on the deformation complex is not always straightforward. This is exemplified well in the case of Lie groups or Lie group homomorphisms \cite{CardS1}, where the usual Lie group cohomology is employed to study deformations, but an algebraic structure on this complex remains unknown. The \emph{deformation complex of morphisms of Lie groupoids}, which we will work in this paper, also falls into this category of having (so far) an unknown algebraic structure on the complex. Nevertheless, we can still use the deformation cohomology to approach the study of nearby structures in an alternative manner. Indeed, in Section \ref{Section:Rigidity} we approach the problem of \emph{stability under deformations} of morphisms which give conditions to understand when \emph{any} smooth family of Lie groupoid morphisms represents a constant path in the corresponding moduli space. 
\emph{Stable} morphisms under deformations are closely related to representing the isolated points on the moduli space of morphisms. For example, one can consider the compact-open topology in the space of morphisms between two Lie groupoids. If this space is locally path-connected around a fixed morphism $\Phi$ then the stability of $\Phi$ under deformations amounts to representing an isolated point in the moduli space. 
However, despite the known fact that the space of smooth maps between manifolds is locally path-connected under compactness of the domain (See \cite{guillemin-golubitsky}, Theorem 1.11, p. 76), the question of whether or not the space of morphisms between two Lie groupoids is locally path-connected is rather subtle, and we leave it to be explored elsewhere. The stability under deformations is then equivalent to stating that $\Phi$ can be deformed only in some \emph{trivial} ways determined by the conjugation by either \emph{bisections} or \emph{gauge-maps}; we will detail these \emph{trivial ways} and relations between them in Section \ref{Section:Deform.morph.}.
For instance, we employ the properness of Lie groupoids to verify that the stability under deformations property holds for morphisms as stated in the following result.

\begin{theorem}
Let $\Phi:\calH\rightarrow\calG$ be any morphism of Lie groupoids. If $\H$ is proper and $\calG$ is a transitive Lie groupoid then any deformation of $\Phi$ is trivial.
\end{theorem}

In Section \ref{Section:Deformation complex morph.} we introduce the deformation complex of morphisms which, together with deformations of morphisms, are the main notions we will work in this paper. As in the statement below, we show that \emph{naturally isomorphic Lie groupoid morphisms have isomorphic deformation cohomologies}. This result establishes the deformation cohomology of morphisms as a key concept within the category of Lie groupoids and equivalence classes of morphisms.

\begin{theorem}
If $\Phi$ and $\Psi$ are equivalent morphisms then their deformation cohomologies $H^{*}_{def}(\Phi)$ and $H^{*}_{def}(\Psi)$ are isomorphic.
\end{theorem}
We also provide an alternative description of the complex using the theory of $VB$-groupoids and their cohomology \cite{GrMe-grpds}.
In Section \ref{Section:Examples} we illustrate the deformation complex with examples and explore its relation to other complexes. 
In Section \ref{Section:Triviality} we use the exactness of 1-cocycles associated to a deformation of morphisms to characterize the \emph{trivial deformations} of morphisms:

\begin{theorem}
A deformation by morphisms $\Phi_\e$ is trivial if and only if the family $X_\e\in C^{1}_{def}(\Phi_\e)$ of associated cocycles is smoothly exact.
\end{theorem}

Such \emph{deformation 1-cocycles} should be conceived as the velocity vectors of the deformation and their exactness as the vanishing of their velocities. Still in Section \ref{Section:Triviality}, we conduct a similar analysis by considering other types of relations between the morphisms that determine the respective coarse moduli spaces of morphisms. We use the associated 1-cocycles to characterize the trivial deformations in each of these moduli spaces and that arise naturally when observing the deformation complex. For instance, we consider the relation induced by the composition action of the full group of automorphisms of the codomain Lie groupoid. A morphism $\Phi$ can then be deformed by composing it with family of such automorphisms, and the family that it produces is said to be a \textit{weakly trivial} deformation.

In Section \ref{Section:Thom-Levine} we extend results from Section \ref{Section:Triviality} to $k$-deformations of morphisms. As an application, we obtain a more geometric proof of the Thom-Levine Theorem, which characterizes the triviality of $k$-deformations of smooth maps between manifolds (\cite{guillemin-golubitsky}, Theorem 3.3, p. 124). 

The Thom-Levine Theorem is crucial in the study of stability of smooth maps and provides an intermediate step in establishing the equivalence between stability and infinitesimal stability of smooth maps. It plays a key role in verifying that infinitesimal stability implies stability under deformations. This, in turn, is used to prove stability under the compactness hypothesis on the domain, ensuring that any map close to a fixed one can be reached by a deformation (path) which, at each time $\e$, is locally trivial.
%**something to guarantee to say that...**any map close enough to a fixed one can be reached by a trivial deformation of maps.?

In Sections \ref{section:lowdegrees} and \ref{Section:Regularsetting}, using the $VB$-cohomology description of the deformation complex, we establish additional results to enhance our understanding of the behaviour of the deformation cohomology. These results are aplicable more generally to any $VB$-groupoid $\Gamma$ over a base $\G\tto M$. For instance, if $\partial:C\rightarrow E$ is the complex of vector bundles over $M$ induced by $\Gamma$ (where $C$ is the core and $E$ the side bundles of $\Gamma$) with $\frakk=\mathrm{Ker}\: \partial$ and $\frakl=\mathrm{Coker}\: \partial$ then we show that there exists an exact sequence for the low-degree cohomology as follows:

\begin{proposition}
\begin{equation*}
0\rightarrow H^{1}(\calG, \mathfrak{k})\stackrel{r}{\rightarrow} H^{1}_{VB}(\calG,\Gamma)\stackrel{\pi}{\rightarrow}\Gamma(\mathfrak{l})^{inv}\stackrel{K}{\rightarrow} H^{2}(\calG, \mathfrak{k})\rightarrow H^{2}_{VB}(\calG,\Gamma).
\end{equation*}
\end{proposition}

Additionally, we verify that the previous sequence can be extended in the particular case of regular groupoids, indeed if $\Gamma$ is a regular $VB$-groupoid

\begin{theorem}
There exists a map $K:H^{\bullet}(\calG, \frakl)\longrightarrow H^{\bullet+2}(\calG, \frakk)$ such that the cohomology $H^{*}_{VB}(\calG,\Gamma)$ associated to $\Gamma$ fits into a long exact sequence
\begin{equation*}
\cdots\longrightarrow H^{k}(\calG, \mathfrak{k})\stackrel{r}{\longrightarrow} H^{k}_{VB}(\calG,\Gamma)\stackrel{\pi}{\longrightarrow} H^{k-1}(\calG, \mathfrak{l})\stackrel{K}{\longrightarrow} H^{k+1}(\calG, \mathfrak{k})\longrightarrow\cdots
\end{equation*}
\end{theorem}

This $VB$-cohomology approach gives us a way to make explicit the deformation cohomology groups of a morphism $\Phi:\calH\to\G$ under the properness condition of its domain.

\begin{proposition}
If $\calH$ is proper, then $H^{0}_{def}(\Phi)\cong\Gamma(\phi^{*}i_\calG)^{inv}$, $H^{1}_{def}(\Phi)\cong\Gamma(\phi^{*}\nu_{\mathcal{G}})^{inv}$ and $H^{k}_{def}(\Phi)=0$ for every $k>1$, where $\mathfrak{i}_\calG$ and $\nu_\calG$ are the isotropy bundle of $\calG$ and the normal bundle to the orbits of $\calG$, respectively.
\end{proposition}

We consider more properties of the complex in Section \ref{Sec:Moritainv} where we establish a Morita invariance of the complex. We use this invariance, combined with \emph{simultaneous deformations} of morphisms, their domains, and codomains (detailed in Section \ref{Section:Simultaneous}) to derive as an application a \emph{deformation cohomology for generalized morphisms}. In essence, we introduce a deformation complex $C^{*}_{def}(\Psi/\Phi)$ for \emph{fractions} $\Psi/\Phi:\H\to\G$ between morphisms of groupoids, and verify that such a cohomology is invariant under equivalence of fractions:

\begin{theorem}
If $\Psi/\Phi$ and $\Psi'/\Phi'$ are equivalent fractions, then their deformation cohomologies $H^{*}_{def}(\Psi/\Phi)$ and $H^{*}_{def}(\Psi'/\Phi')$ are isomorphic.
\end{theorem}

Thus, this result yields a well-defined algebraic object associated to generalized morphisms or maps between differentiable stacks. The use of this complex in the study of deformations of fractions will be explored in future work. %...in what sense CHECK in that Section?.
The study of deformations of Lie subgroupoids, which is explored in detail in \cite{CardS-Subgroupoids}, is also motivated by the work on simultaneous deformations, where the deformation complexes are obtained through the cohomological construction of the mapping-cone complex. As a final application, in Section \ref{Section:multiplicativeforms} we introduce the study of deformation of multiplicative forms on Lie groupoids. There, we define the \emph{deformation complex for multiplicative forms} and characterize the trivial deformations in terms of the exactness of the associated cocycles in the complex, which can be thought of as a Moser's type Theorem for multiplicative forms. This topic is deeply explored in \cite{CardMS} for studying deformations of symplectic groupoids.

\section*{Acknowledgements}

I would like to thank to Ivan Struchiner for many valuable discussions, important advice as well as for suggestions, and comments on initial drafts of this paper. Special thanks to Joao Nuno Mestre for valuable suggestions on the first versions of this paper. This work was also benefited by conversations with Cristian Ortiz and Matias del Hoyo, which improved the final version of this paper. This research received support from a PNPD postdoctoral fellowship at UFF.

\section{Background: Deformation Theory of Lie Groupoids and VB-Groupoids}\label{Background}
In this section we will recall some preliminary content that will be used throughout the paper. Its purpose is mostly to establish notations and to maintain this paper as self-contained as possible. For further details on Lie groupoids, Lie algebroids, deformation theory of Lie groupoids and VB-groupoids, we refer the reader to \cite{M}, \cite{CMS} and \cite{GrMe-grpds}.

Let $\calG\rightrightarrows M$ be a Lie groupoid, denote by $s$, $t$, $m$, $i$ and $u$ the source, target, multiplication, inversion and unit of $\calG$, respectively. We will write $m(g,h)=gh$ and $i(g)=g^{-1}$ when the context is clear, and identify $x\in M$ with the corresponding unity $u(x)\in\calG$. For $g\in\calG$, we write $g:x\rightarrow y$ meaning $s(g)=x$ and $t(g)=y$.\\

If $g:x\rightarrow y$, there is a right translation
$$R_{g}:s^{-1}(y)\rightarrow s^{-1}(x),\ R_g(h)=hg;$$
and analogously a left-translation $L_{g}:t^{-1}(x)\rightarrow t^{-1}(y),$ $L_g(h)=gh$, between the $t$-fibers. We denote their differentials, respectively, by $r_g$ and $l_g$. With this, $X\in\mathfrak{X}(\calG)$ is a \textit{right-invariant vector field} if $ds(X_g)=0$ and $X_{gh}=r_h(X_g)$ for all $(g,h)\in\calG^{(2)}$. Observe that the subset of right-invariant vector fields $\mathfrak{X}^{r}(\calG)\subset\mathfrak{X}(\calG)$ has a $C^{\infty}(M)$-module structure given by $f\cdot X:=(t^{*}f)X$. Further, $\mathfrak{X}^{r}(\calG)\subset\mathfrak{X}(\calG)$ is a Lie subalgebra with the usual Lie bracket on vector fields.

Just as a Lie group $G$ has an associated Lie algebra $\mathfrak{g}$, Lie groupoids also can be studied infinitesimally giving rise to the notion of Lie algebroid. The Lie algebroid $A_\calG$ of $\calG$ is determined (like $\mathfrak{g}$) by the Lie algebra of right-invariant vector fields on $\calG$. More precisely, $A_\calG$ is the vector bundle $\left.(T^{s}\calG)\right|_{M}$ over $M$, where $T^{s}\calG:=Ker(ds:T\calG\longrightarrow s^{*}TM)$ and $M\subset\calG$ is viewed inside $\calG$ as the units. In this way, there is a map $\Gamma(A_\calG)\longrightarrow\mathfrak{X}^{r}(\calG),\ \alpha\longmapsto\vec{\alpha}:g\mapsto r_g(\alpha_{t(g)}),$ which is easily seen to be an isomorphism of $C^{\infty}(M)$-modules inducing then a Lie bracket on $\Gamma(A_\calG)$. The vector field $\vec{\alpha}$ is called the right-invariant vector field associated to $\alpha$. The vector bundle $A_{\calG}$ is also equipped with a vector bundle map $\rho:A_\calG\longrightarrow TM$ given by the restriction of $dt:T\calG\longrightarrow TM$ to $A_\calG\subset T\calG$. The map $\rho$ is called the \textit{anchor map of} $A_\calG$.

In other words, the Lie algebroid associated to $\calG$ consists of the pair $(A_\calG,\rho)$ together with the Lie bracket on sections of $A_\calG$, induced from that of $\mathfrak{X}^{r}(\calG)$. With this point of view, we can abstract such a notion of Lie algebroid and to say that a vector bundle $A$ over $M$ is a Lie algebroid if there exist a vector-bundle map $\rho:A\longrightarrow TM$ together with a Lie bracket on the sections of $A$ in such a way that a Leibniz rule is satisfied:
$$[\alpha,f\beta]_{\Gamma(A)}=f[\alpha,\beta]+L_{\rho(\alpha)}(f)\cdot\beta,$$
for every $\alpha, \beta\in \Gamma(A)$ and $f\in C^{\infty}(M)$.
As the reader may expect, $A_\calG$ defined as above is an example of a Lie algebroid in this more abstract context. More examples can be found in \cite{MM} and \cite{M}.

\subsection{Deformation theory of Lie Groupoids}\label{DeformationLieGrpds}

%\section{On groupoids, algebroids and deformation of Lie groupoids}\label{deform.complexgrpds}

The deformation theory of Lie groupoids was recently introduced in \cite{CMS}. In there, the authors developed the main aspects of the theory; among other things, they exhibite the corresponding cohomology attached to deformations of Lie groupoids and use it to prove the stability of compact Lie groupoids. We recall here some key facts of the constructions in \cite{CMS}.\\

\subsubsection*{\textbf{Deformations of Lie groupoids}}
A deformation of a manifold is roughtly understood in terms of a smooth family of manifolds. A smooth family of manifolds $\{\left.M_\e\right|\e\in I\}$ is viewed as a manifold $\tilde{M}$ together with a submersion $\tilde{\pi}:\tilde{M}\longrightarrow I$, such that every $M_\e$ is the fiber $\pi^{-1}(\e)$ over $\e$. One also says that the family $\{\left.M_\e\right|\e\in I\}$ is \textit{smoothly parametrized by} $I$. This notion is the central idea to define deformations of Lie groupoids. Explicitly,

\begin{definition}(Smooth family of Lie groupoids)\\
A \emph{smooth family of Lie groupoids} parametrized by a manifold $B$ is given by a Lie groupoid $\tilde{\calG}\rightrightarrows\tilde{M}$ and a surjective submersion $\pi$ such that $\pi\circ\tilde{s}=\pi\circ\tilde{t}$
$$\tilde{\calG}\rightrightarrows\tilde{M}\stackrel{\pi}{\rightarrow}B.$$

In this way, $\pi$ determines the family of Lie groupoids $\{\left.\calG_b\right|b\in B\}$, where $\calG_b$ denotes the restricted groupoid over $M_b=\pi^{-1}(b)$. One says that the family is \textit{proper} if $\tilde{\calG}$ is proper, i.e., if $\tilde{s}\times\tilde{t}:\tilde{\calG}\longrightarrow\tilde{M}\times\tilde{M}$ is a proper map.\\
Two familes $\tilde{\calG}\rightrightarrows\tilde{M}\stackrel{\pi}{\rightarrow} B$ and $\tilde{\calG}'\rightrightarrows\tilde{M}'\stackrel{\pi'}{\rightarrow} B$ are \textit{isomorphic} if there exists an isomorphism of groupoids $(F,f):\tilde{\calG}\rightarrow\tilde{\calG}'$ compatible with the submersions $\pi$ and $\pi'$ in the sense that $\pi'\circ f=\pi$. This isomorphism $F$ can be thought of as a smooth family of isomorphisms $F_b:\calG_b\longrightarrow\calG'_b$ parametrized by $B$.
\end{definition}

\begin{definition}\label{Deform.grpds}(Deformation of Lie groupoids)\\
Let $\calG\rightrightarrows M$ be a Lie groupoid with structural maps $s,\ t,\ m,\ i,\ u.$ A \textit{deformation} of $\calG$ is a smooth family of Lie groupoids $\tilde{\calG}$ parametrized by an open interval $I$ containing zero,
$$\tilde{\calG}=\{\calG_\e\rightrightarrows M_\e:\e\in I\}$$ 
such that $\calG_0=\calG$. We denote the structural maps of $\calG_\e$ by $s_\e,\ t_\e,\ m_\e,\ i_\e,\ u_\e.$\\

The deformation $\tilde{\calG}$ of $\calG$ is called \textbf{strict} if the all fibers $\calG_\e$ and $M_\e$ are diffeomorphic to $\calG$ and $M$ in a smooth way, i.e., if there exist two diffeomorphisms $F:\tilde{\calG}\rightarrow\calG\times I$ and $f:\tilde{M}\rightarrow M\times I$ such that $pr_I\circ F=\pi\circ\tilde{s}$ and $pr_I\circ f=\pi$. In other words, essentially we only deform the structural maps of $\calG$: $(\tilde{\calG}\rightrightarrows\tilde{M}, \tilde{s}, \tilde{t})\cong(\calG\times I\rightrightarrows M\times I, s_\e\times\e, t_\e\times\e)$. In such a case, we can assume $\tilde{\calG}=\calG\times I$ and the deformation is said to be \textbf{$s$-constant} if $s_\e$ does not depend on $\e$. The (strict) deformation such that $\calG_\e=\calG$ as groupoids is called the \textbf{constant deformation} of $\calG$.\\
Two deformations $\tilde{\calG}=\{\calG_\e\rightrightarrows M_\e:\e\in I\}$ and $\tilde{\calG}'=\{\calG'_\e\rightrightarrows M'_\e:\e\in I'\}$ are \textit{locally equivalent} if there exist a family of isomorphisms of groupoids $F_\e:\calG_\e\longrightarrow\calG'_\e$, smoothly parametrized by $\e$ in a open interval containing zero (contained in $I\cap I'$), such that $F_0=Id_\calG$. 
\end{definition}

\begin{remark}
Consider two locally equivalent deformations $\tilde{\calG}$ and $\tilde{\calG}'$. For simplicity and because around $\calG_0$ the families $\tilde{\calG}$ and $\tilde{\calG}'$ are isomorphic, we will just say that $\tilde{\calG}$ and $\tilde{\calG}'$ are \textbf{equivalent deformations} of $\calG_0$ (even if $I\neq I'$).
\end{remark}
With the convention of the last remark, the deformation $\tilde{\calG}$ is called \textbf{trivial} if it is equivalent to the constant deformation.

\begin{remark}[Fibrations] \label{FibrationsGrpds}
Recall that a fibration between two Lie groupoids is a Lie groupoid morphism $\mathfrak{F}:\calG\to\calH$ such that the map $\mathfrak{F}^{!}:\calG\to \calH\times_{N}M$, $g\mapsto (\mathcal{F}(g),s_{\calG}(g))$, is a surjective submersion. As pointed out in \cite{MdHRF}, a deformation of $\calG$ also can be regarded in terms of fibrations of Lie groupoids.
The data $\tilde{\calG}\rightrightarrows\tilde{M}\stackrel{\pi}{\rightarrow}I$ involved in the definition of a deformation of $\calG$ can be expressed in the form
$$\mathfrak{F}:(\tilde{\calG}\rightrightarrows\tilde{M})\longrightarrow(I\rightrightarrows I),$$
where $\mathfrak{F}$ is a fibration of Lie groupoids. The family $\calG_\e$ of Lie groupoids corresponds to the fibers of the fibration. In this sense, a strict deformation can be thought of as a fibration $\mathfrak{F}$ where the maps between the arrows and the objects are locally trivial. In fact, two trivializations $F_1:\tilde{\calG}\longrightarrow\calG\times I$ and $F_0:\tilde{M}\longrightarrow M\times I$ induce a family of Lie groupoid structures $\calG_\e$ on the manifold $\calG=\tilde{\calG}_0$. For instance, the deformation of the source map is determined by $(s_\e(g),\e)=F_0\circ\tilde{s}\circ F_1^{-1}(g,\e)$ and so on.
\end{remark}

Examples of deformations of Lie groupoids are considered in (\cite{MdHRF}, p. 16). As a manner of illustration we sketch here some of them.

\begin{examples}%$\left.\right.$
\begin{enumerate}
\item Let $G=\mathbb{R}^{2}$. Consider the family of Lie groups $G\times\mathbb{R}\longrightarrow\mathbb{R},\ (g,\e)\mapsto\e$ given by
$$(x_1,y_1)\cdot_\e(x_2,y_2):=(x_1+x_2,y_1+e^{x_{1}\e}y_2).$$
Due to the fact that for $\e\neq0$ the multiplication $\cdot_\e$ is non-abelian, this is a non-trivial deformation of $G$.

\item Consider the family of Lie group actions of $\mathbb{R}$ on $\mathbf{T}^{2}:=\mathbb{R}^{2}/\mathbb{Z}^{2}$, given by:
$$r\cdot_\e(x_1,x_2):=(x_1+r,x_2+\e r).$$
Thus, if $G:=\mathbb{R}\times\mathbf{T}^{2}$ then such a family of actions can be seen as a family of action groupoids $G\times\mathbb{R}\longrightarrow\mathbb{R}$. This is of course a non-trivial deformation of $G$ since the topology of the orbits varies with $\e$.
\end{enumerate}
\end{examples}

\subsubsection*{\textbf{Deformation cohomology of Lie groupoids}}

The fundamental fact of the deformation complex $(C^{*}_{def}(\calG),\delta_{\calG})$ of a Lie groupoid $\calG$ is that it governs deformations of $\calG$. Concretely, to every deformation of $\G$ one associates a cohomology class in $H^{2}_{def}(\calG)$, and this correspondence also shows a relation between the equivalence classes of deformations of $\calG$ and the classes of $H^{2}_{def}(\calG)$. The \textbf{deformation complex of $\calG$} is defined as follows.

For any $k\in\mathbb{N}$, consider $\calG^{(k)}=\left\{(g_{1},...,g_{k}): s(g_{i})=t(g_{i+1})\right\}$ the manifold of $k$-strings of composable arrows, and define $\calG^{(0)}=M$. The space of $k$-\textit{cochains} $C^{k}_{def}(\calG)$ is given by
$$C^{k}_{def}(\calG)=\left\{\left.c:\calG^{(k)}\rightarrow T\calG\right|\ c(g_{1},...,g_{k})\in T_{g_{1}}\calG \text{ and } c \text{ is } s\text{-projectable}\right\},$$
where $s$-projectable means that $ds\circ c(g_{1},...,g_{k})=:s_{c}(g_{2},...,g_{k})$ does not depend on $g_1$. The differential of $c$ is defined by
\begin{align*}
(\delta c)(g_1,...,g_{k+1}):&=-d\bar{m}(c(g_1g_2,g_3,...,g_{k+1}),c(g_2,...,g_{k+1}))+\\
&+\sum_{i=2}^{k}(-1)^i c(g_1,...g_ig_{i+1},...,g_{k+1})+(-1)^{k+1}c(g_1,...,g_k),
\end{align*}
where $\bar{m}:\calG_s\times_s\calG\longrightarrow\calG$, $\bar{m}(g,h)=gh^{-1}$ is the division map of $\calG$.\\
For $k=0$, $C^{0}_{def}(\calG):=\Gamma(A)$ with differential defined by
$$\delta\alpha=\overrightarrow{\alpha}+\overleftarrow{\alpha}\in C^{1}_{def}(\calG),$$
where $\overleftarrow{\alpha}$ is the \textit{left-invariant} vector field on $\calG$ associated to $\alpha$ defined by $\overleftarrow{\alpha}(g):=l_g(di(\alpha_{s(g)}))$. Note that a section of $A$ can be viewed as a map $c:\calG^{(0)}\longrightarrow T\calG$, with $c(1_x)\in T_{1_x}\calG$ such that $ds\circ c=0$.\\
This data in fact defines a cohomology ($\delta^{2}=0$) and $H^{*}_{def}(\calG)$ denotes the \textbf{deformation cohomology of} $\calG$.

In this way, one can describe explicitly the cohomology class $[\xi_0]\in H^{2}_{def}(\calG)$ associated to an $s$-constant deformation of $\calG$ by
$$\xi_0(g,h):=\left.\frac{d}{d\e}\right|_{\e=0}\bar{m}_{\e}(m_0(g,h),h),\ \ \xi_0\in Z^{2}_{def}(\calG),$$
where $m_0$ denotes the multiplication of $\calG=\calG_0$. The fact that $\xi_0$ is a cocycle is implied from applying $\left.\frac{d}{d\e}\right|_{\e=0}$ to the associativity property of $\bar{m}_\e$. The element $\xi_0$ is called the \textbf{deformation cocycle} of the deformation of $\calG$. For deformations which are not necessarily $s$-constant, a slightly different approach needs to be used yielding a non canonical 2-cocycle, however one does gets a canonical a 2-cohomology class for any deformation in the same equivalence class (see Section 5.4 in \cite{CMS}).

Of remarkable importance is the transgression of the 2-cocycle $\xi_0$; when it exists, it plays a key role in the stability under deformations problem of Lie groupoids, as we explain below.

\subsubsection*{\textbf{Moser's argument (towards stability under deformations)}}

%With the help of the deformation complex one can satisfactorily solves the rigidity question for Lie groupoids.
One fundamental step to study the stability question for Lie groupoids is given by the following proposition, which uses the deformation complex to state a result in the same spirit as that of the classical Moser's theorem of symplectic geometry (see e.g. \cite{McS} p. 93).

\begin{proposition}\label{Moser-grpds}\cite{CMS}
Let $\tilde{\calG}=\left\{\calG_{\e}:\e\in I\right\}$ be an $s$-constant deformation of $\calG$. Consider the induced cocycles $\xi_{\e}\in C^{2}_{def}(\calG_{\e})$, at each time $\e$, defined in analogous way to $\xi_0$ above $(\xi_\e=\left.\frac{d}{d\l}\right|_{\l=0}\bar{m}_{\e+\l}(m_\e(g,h),h))$. Assume that for every $\e$ small enough there exists $X_{\e}\in C^{1}_{def}(\calG_{\e})$ such that
\begin{equation}\label{cohom eq}
\delta_\e(X_{\e})=\xi_{\e},
\end{equation}
and that the resulting time-dependent vector field $X:=\left\{X_{\e}\right\}$ on $\calG$ is smooth. Then, for $\e_1$ and $\e_2$ close to 0, the time-dependent flow $\psi^{\e_2,\e_1}_{X}$ is a locally defined morphism from $\calG_{\e_1}$ to $\calG_{\e_2}$ covering the time-dependent flow of $V:=\{V_{\e}:=ds(X_{\e})\}$ on $M$.\\
\noindent Additionally, if $\calG$ is proper, $\psi^{\e_2,\e_1}_{X}(g)$ is defined if and only if $\psi^{\e_2,\e_1}_{V}(s(g))$ and $\psi^{\e_2,\e_1}_{V}(t(g))$ are defined.
\end{proposition}

This proposition tells us the conditions under which one finds a flow compatible with the variations of the structural maps of $\calG$. However, by considering the structural maps of the total groupoid $\tilde{\calG}$, one can express the following equivalent version of the proposition.

\begin{proposition}\cite{CMS}
Consider an $s$-constant deformation as above. A one-parameter family $X_{\e}$ of vector fields on $\calG$ satisfies the cocycle equations (\ref{cohom eq}) if and only if the induced vector field on $\tilde{\calG}=\calG\times I$,
$$\tilde{X}(g,\e)=(X_{\e}(g),0)+\frac{\partial}{\partial\e}\in\mathfrak{X}(\calG\times I),$$
is multiplicative.
\end{proposition}

In this way, one knows that the flow of $\tilde{X}$ (when uniformly defined) is given by automorphisms of $\tilde{\calG}$ (\cite{MX}, Prop. 3.5). And the stability under deformations question of Lie groupoids is essentially solved by finding a \textit{complete} vector field like $X$ (or $\tilde{X}$) above for every deformation of $\calG$ (\cite{CMS}, Theorem. 7.1).\\
Analogous concepts to those described in the three steps above (deformations, cohomology and Moser's trick) will be developed when working with the deformation theory of Lie groupoid morphisms in this paper.

We introduce now some notions of the theory of VB-groupoids which will serve us to obtain alternative descriptions and give a treatment of the deformation complex of morphisms which we will work with.

\newpage

\subsection{Interlude on VB-groupoids}\label{VB-grpds}
\subsubsection*{\textbf{VB-groupoids}}\label{InterludeVB-groupoids}
A VB-groupoid can be thought of as a groupoid object in the category of vector bundles. They provide alternative ways to look at the representation theory and the deformation theory of Lie groupoids. For instance, the deformation complex of Lie groupoids (Subsection \ref{DeformationLieGrpds}) can be seen as the VB-complex which is naturally associated to the cotangent groupoid when regarded as a VB-groupoid (See \cite{CMS}, \cite{dHO} and Remark \ref{deformation and vb-complexes} below). This point of view will be useful in the study of the deformation complexes defined in this paper. A more detailed description of the theory of VB-groupoids and VB-complexes can be found in \cite{M}, \cite{GrMe-grpds} and \cite{BCdH}.

\begin{definition}
 A $VB$-groupoid $(\Gamma, E, \calG, M)$ is a structure of two Lie groupoids and two vector bundles as in the diagram below
\begin{equation}\label{VB-grpd}
%\[
  \begin{tikzcd}[column sep=4em, row sep=10ex]
    \Gamma \arrow[yshift=2pt]{r}{\tilde{s}} \arrow[yshift=-2pt]{r}[swap]{\tilde{t}} \arrow{d}{\tilde{q}} & E \arrow{d}{q}\\
    \calG \arrow[yshift=2pt]{r}{s} \arrow[yshift=-2pt]{r}[swap]{t}    & M,
  \end{tikzcd}
%\]
\end{equation}
where the vertical directions are vector bundle structures and the horizontal ones are Lie groupoids, such that the structure maps of the groupoid $\Gamma$ (source, target, identity, multiplication, inversion) are vector bundle morphisms over the corresponding structure maps of the groupoid $\calG$. 
\end{definition}

\begin{remark}
Note that the multiplication $m_\Gamma:\Gamma^{(2)}\longrightarrow\Gamma$ makes sense as a vector bundle morphism when one considers the induced vector bundle structure of $\Gamma^{(2)}$ over $\calG^{(2)}$ (guaranteed from the fact that the \textit{`double source map'} $(\tilde{q},\tilde{s}):\Gamma\longrightarrow\calG\; _{s}\,\times E$ is a surjective submersion (appendix A in \cite{Li-Bland-Severa}).
\end{remark}

In this setting a \textbf{morphism of VB-groupoids} $(\Phi_\Gamma,\Phi_E,\Phi_\calG,\Phi_M):(\Gamma, E, \calG, M)\longrightarrow (\Gamma', E', \calG', M')$ is a morphism $(\Phi_\Gamma,\Phi_E)$ between the Lie groupoids $\Gamma\rightrightarrows E$ and $\Gamma'\rightrightarrows E'$ preserving the vector bundle structures, i.e, such that $\Phi_\Gamma$ and $\Phi_E$ are vector bundle morphisms covering the maps $\Phi_\calG:\calG\longrightarrow \calG'$ and $\Phi_M:M\longrightarrow M'$, respectively. Observe that, by restricting $\Phi_\Gamma$ to the zero section, $\Phi_\calG$ turns out to be a Lie groupoid morphism.

\begin{example}(Tangent VB-groupoid)
Given a Lie groupoid $\calG\rightrightarrows M$ with source, target and multiplication maps $s$, $t$ and $m$, by applying the tangent functor one gets the tangent groupoid $T\calG\rightrightarrows TM$ with structure maps $Ts$, $Tt$, $Tm$ and so on. This tangent groupoid is further a VB-groupoid over $\calG\rightrightarrows M$ (with respect to the tangent projections).
\end{example}

\begin{remark}
Note that in the previous example one has the following short exact sequences of vector bundles over $\calG$,
\begin{equation}\label{left}
s^{*}(A_\calG)\stackrel{-l\circ Ti}{\longrightarrow} T\calG\stackrel{(Tt)^{!}}{\longrightarrow}t^{*}(TM)
\end{equation} and
\begin{equation}\label{right}
t^{*}(A_\calG)\stackrel{r}{\longrightarrow} T\calG\stackrel{(Ts)^{!}}{\longrightarrow}s^{*}(TM)
\end{equation}
where $r$ and $l$ are the right and left multiplication on vectors tangent to the $s$-fibers and $t$-fibers of $\calG$, respectively; and $(Ts)^{!}$ and $(Tt)^{!}$ are the maps induced by $Ts$ and $Tt$ with image on the corresponding pullback bundles. 
\end{remark}

\begin{example}(Cotangent groupoid)
As noticed in \cite{CosteDazordWeins.}, given a Lie groupoid $\calG$ its cotangent bundle inherits a groupoid structure over the dual of the Lie algebroid of $\calG$,
$$T^{*}\calG\rightrightarrows A_\calG^{*},$$ 
with source and target maps induced, respectively, from the dual of the exact sequences (\ref{left}) and (\ref{right}). Explicitly, for $\alpha_g\in T^{*}_{g}\calG$ and $a\in\Gamma(A_\calG)$,
$$\left\langle \tilde{s}(\alpha_g),a_{s(g)}\right\rangle=-\left\langle \alpha_g, l_g\circ Ti(a)\right\rangle$$
and
$$\left\langle \tilde{t}(\alpha_g), a_{t(g)}\right\rangle=\left\langle \alpha_g, r_g(a)\right\rangle.$$
With multiplication determined by $$\left\langle \tilde{m}(\alpha_g,\beta_h),Tm(v_g,w_h)\right\rangle=\left\langle \alpha_g,v_g\right\rangle+\left\langle \beta_h,w_h\right\rangle,$$
for $(v_g.w_h)\in (T\calG)^{(2)}$.
\end{example}

There are two canonical bundles over $M$ associated to a VB-groupoid $\Gamma$, they are called the side and core bundles of $\Gamma$. The \emph{side bundle} is just the vector bundle $E$ over $M$. The \emph{core bundle} $C$ is determined by the restriction to $M$ of the kernel $K$ of the surjective map
$$\Gamma\stackrel{\tilde{s}^{!}}{\longrightarrow} s^{*}E;$$
explicitly, one defines $C$ as the restriction $\left.K\right|_{M}$ of $K$ to the units of $\calG$. The core bundle then can be thought of as the bundle in the complementar direction to that of the side bundle. These two bundles gives rise to the short exact sequence
\begin{equation}\label{coresequence}
0\rightarrow t^{*}C\stackrel{r}{\rightarrow}\Gamma\stackrel{\tilde{s}^{!}}{\rightarrow}s^{*}E\rightarrow0
\end{equation}
which is called the \textbf{core sequence} of $\Gamma$, where $r$ is the right multiplication by zero elements of $\Gamma$ on the vectors of $C$. Observe thus that the sequence \eqref{right} above is a particular case of this sequence.

The target map of $\Gamma$ determines a vector bundle map $\partial:C\rightarrow E$ between the core and side bundles of $\Gamma$. This map, which is called the \emph{core-anchor map}, coincides with the known anchor $\rho_\calG$ of $A_\calG$ when $\Gamma=T\calG$.

An splitting $\sigma:s^{*}E\to\Gamma$ of the exact sequence \eqref{coresequence} is called \textbf{unitary} if over the units $M$ of $\G$ it coincides with the unit map $u_\Gamma:E\to\Gamma$ of $\Gamma$. In the particular case of the tangent groupoid (i.e., $\Gamma=T\G$), such an unitary splitting is called an \textbf{Ehresmann connection} of $\G$ for the source $s$.

These horizontal lifts of the core sequence have a relevant role in the representation theory of Lie groupoids. They allow to define \emph{quasi-actions} $\Delta^{E}_g:E_{s(g)}\to E_{t(g)}$ and $\Delta^{C}_g:C_{s(g)}\to C_{t(g)}$ of $\G$ on the side and core bundles of $\Gamma$ which are key elements in the notion of 2-terms \emph{representations up to homotopy of Lie groupoids} \cite{GrMe-grpds}. Explicitly, they are given by

$$\Delta^{E}_g(e_{s(g)}):=t_{\Gamma}(\sigma_g(e_{s(g)})), \text{ and } \Delta^{C}_g(c_{s(g)}):=\sigma_{g}(\rho(c_{s(g)}))\cdot c_{s(g)}\cdot 0_{g}^{-1}.$$

\begin{remark}\label{rmk:adjoint action}
Notice that an interesting fact of these quasi-actions $\Delta^{C}$ and $\Delta^{E}$ is that they restrict to canonical actions of $\calG$ on the spaces $Ker\;\partial$ and $Coker\;\partial$. Explicitly, the action on $Ker\;\partial$ is given by
$$\Delta^{Ker\;\partial}_g(c_{s(g)}):=0_g\cdot c_{s(g)}\cdot 0_g^{-1}.$$
In particular, in the case of the tangent VB-groupoid $T\G$ of a groupoid $\G$, the action $\Delta^{\Ker(\rho)}$ is called the \emph{adjoint action} on the isotropy (possibly singular) bundle $\mathfrak{i}_\G$ of $\G$; and the action $\Delta^{Coker(\rho)}$ turns out to be the so-called action of $\G$ on the normal (possibly singular) bundle $\nu_\G$ whose fibers are the normal spaces to the orbits of $\G$.
\end{remark}

\subsubsection*{\textbf{VB-groupoid cohomology}}
VB-groupoids have a special cohomology induced from their own groupoid structure which, additionally, takes into account the linear structure of the vector bundle and aims to give a geometric interpretation of the 2-terms representations up to homotopy of Lie groupoids. Such a complex, which is called the \emph{VB-groupoid complex}, was defined by Gracia-Saz and Mehta in \cite{GrMe-grpds} and it turns out to be (canonically) isomorphic to the deformation complex of Lie groupoids when considering the cotangent groupoid (see Remark \ref{deformation and vb-complexes}), providing another interpretation of the deformation complex of a Lie groupoid. Its definition is as follows.

Let $\Gamma$ be a VB-groupoid. The differentiable complex of $\Gamma$ (as Lie groupoid) has a natural subcomplex $C^{*}_{lin}(\Gamma)$ given by the fiberwise linear cochains of $\Gamma$. The VB-groupoid complex $C^{*}_{VB}(\Gamma)$ of $\Gamma$ is the subcomplex of $C^{*}_{lin}(\Gamma)$ determined by the \textbf{left-projectable} elements of $C^{*}_{lin}(\Gamma)$, that is, the elements satisfying the following two conditions
\begin{enumerate}
	\item $c(0_{g_1},\gamma_{g_2},...,\gamma_{g_k})=0$,
	\item $c(0_g\cdot\gamma_{g_1},\gamma_{g_2},...,\gamma_{g_k})=c(\gamma_{g_1},\gamma_{g_2},...,\gamma_{g_k})$.
\end{enumerate}
This VB-complex turns out to be isomorphic to the one of 2-terms representation up to homotopy of $\G$ over the side and core bundles of the dual VB-groupoid, and thus allows us to think about the 2-terms representation theory in the geometric terms of VB-groupoids (\cite{GrMe-grpds}). In particular, it yields an interpretation of the adjoint representation of a Lie groupoid in terms of the VB-complex of the cotangent groupoid. In that way, having in mind the relation of the deformation complex of groupoids with the adjoint representation, then an expected but relevant fact of the VB-complex concerns its relation with the deformation complex of Lie groupoids:

\begin{remark}\label{deformation and vb-complexes}%($C^{*}_{VB}(T^{*}\calG)$ vs $C^{*}_{def}(\calG)$)

A straightforward computation shows that the deformation complex of a Lie groupoid $\calG$ is isomorphic to the VB-groupoid complex $C^{*}_{VB}(T^{*}\calG)$ of its cotangent VB-groupoid (\cite{dHO}, Prop. 4.5). The isomorphism is given by $C^{*}_{def}(\calG)\longrightarrow C^{*}_{VB}(T^{*}\calG)$, $c\mapsto c'$ with
$$c'(\eta_{g_1},...,\eta_{g_k})=\left\langle\eta_{g_1}, c(g_1,...,g_k)\right\rangle.$$ 
\end{remark}

\subsubsection*{\textbf{VB-Morita maps}}

A VB-Morita map takes the so important notion of Morita maps of Lie groupoids to the level of VB-groupoids. A morphism of VB-groupoids $(\Phi,\phi):\Gamma\to\Gamma'$ is a \textbf{VB-Morita map} if $\Phi$ is a Morita morphism \cite{dHO}. In that sense, VB-Morita maps are supposed to play the same role as Morita maps for Lie groupoids. For instance, in \cite{dHO} the authors prove the VB-Morita invariance of the VB-cohomology. That is, if $\Phi:\Gamma\to\Gamma'$ is a VB-Morita map then the VB-cohomologies $H_{VB}^{\bullet}(\Gamma)$ and $H_{VB}^{\bullet}(\Gamma')$ are isomorphic.
In particular, since the tangent lift $T\Phi:T\G\to T\G'$ of a Morita morphism $\Phi:\G\to\G'$ is a VB-Morita map then, by using Remark \ref{deformation and vb-complexes} and the fact that VB-Morita maps are preserved by dualization, the authors give an alternative proof of the Morita invariance of the deformation cohomology of Lie groupoids, first proven in \cite{CMS}. We will also use the notion of VB-Morita morphisms to obtain a Morita invariance of the deformation cohomology of Lie groupoid morphisms. Such a result will make possible to connect the deformation cohomology of morphisms with maps of differentiable stacks (see Section \ref{Sec:Moritainv}).

%!!!!!!!!!!!!!!!!!!!!!!!!!!!!!!!!!!!!!!!!!!!!!!!!!!!!!!!!!!!!!!!!!!!!!!!!!!!!!!!!!!!!!!!!!!!!!!!!!!!!!!!!!!!!!!!!!!!!!!!!!!!!!!!!!!!!!!!!!!!!!!!!!!!!!!!!!!!!!!!!!!!!!!!!!!!!!!!!

%!!!!!!!!!!!!!!!!!!!!!!!!!!!!!!!!!!!!!!!!!!!!!!!!!!!!!!!!!!!!!!!!!!!!!!!!!!!!!!!!!!!!!!!!!!!!!!!!!!!!!!!!!!!!!!!!!!!!!!!!!!!!!!!!!!!!!!!!!!!!!!!!!!!!!!!!!!!!!!!!!!!!!!!!!!!!!!!

\section{Gauge maps and Deformation of morphisms}\label{Section:Deform.morph.}

%Roughtly speaking, a deformation of a structure is thought of as a smoothly parametrized family of such structures. For Lie groupoids, as seen in the Section \ref{deform.complexgrpds}, we can formalize the smoothness of such a family looking at it as the fibers of a submersion. However, for Lie groupoid morphisms can be regarded simply as functions, the notion of deformations has a simpler description. 

In this section we will define deformations of Lie groupoid morphisms.  We also explain two ways in which two such deformations can be considered as equivalent deformations: by using either \emph{gauge maps} or bisections. Each of these two ways turns out to be convenient depending on the examples and/or applications one has in mind, as we will see in the examples below. Let us first introduce the notion of \emph{gauge map}. % in order to obtain two different manners to relate Lie groupoid morphisms.

A \emph{gauge map} arises naturally from the concept of \emph{natural transformations} between functors and allows us to relate two Lie groupoid morphisms when we regard them as functors between the groupoids. In that case, a \emph{natural transformation} yields an obvious relation between two functors, and a gauge map is just an abstraction of that idea in which we define it without making reference to the functors (as it is the case of natural transformations). Formally: let $\G\tto M$ be a Lie groupoid, $N$ a manifold and $\tau:N\to\calG$ be a smooth map. We will call $\tau$ a \textbf{gauge map covering} $f:=s_\calG\circ\tau:N\to M$. Thus indeed, given a gauge map $\tau$ covering $f$ and a Lie groupoid morphism $\Phi:\calH\to\calG$ with base map $f:N\to M$, then $\tau$ relates the morphism $\Phi$ to the morphism $\Phi'$ defined by
$$\Phi'(h)=\tau(t(h))\Phi(h)\tau(s(h))^{-1}.$$

Given a bisection $\sigma:M\to\calG$ of $\G$ and $f:N\to M$, one obtains a gauge map over $f$ by considering $\sigma\circ f$. But, there are in general more gauge maps than those obtained from bisections of $\G$. Hence, a bisection $\sigma$ yields another relation between two morphisms: it relates the morphism $\Phi$ to the morphism $\Phi''$ defined by
$$\Phi''(h)=\sigma(f(t(h)))\Phi(h)\sigma(f(s(h)))^{-1}.$$
Thus, the set of morphisms related by bisections to $\Phi$ is smaller than the set of morphisms related to $\Phi$ by gauge maps. These two manners of acting on a Lie groupoid morphism (either by gauge maps or bisections) are natural ways of generalizing the adjoint action of a Lie group on a Lie group homomorphism considered in \cite{CardS1} when working with deformations of Lie group homomorphisms.

We will consider below these two relations between morphisms at the level of \emph{deformations of morphisms}: Definitions \ref{Def:equivalent deformations} and \ref{Def:strongly equivalent}. However, as we will show in Proposition \ref{Prop:gaugefrombisections}, under some conditions, a gauge map can be obtained by considering \emph{local} bisections, and that fact can be used to prove that, at the level of deformations, such two relations can eventually be the same (see Proposition \ref{Prop:familybisectionsfromgauge} and Theorem \ref{Thm:equivalenttostronglyequivalent}).

\subsubsection*{\textbf{Deformations of Morphisms}}

Let $\calH \tto N$ and $\calG\rightrightarrows M$ be two Lie groupoids and let
\[\xymatrix{\calH \ar[r]^\Phi \ar@<0.25pc>[d] \ar@<-0.25pc>[d]& \calG \ar@<0.25pc>[d] \ar@<-0.25pc>[d]\\
N \ar[r]_{\phi} & M}\] 
be a Lie groupoid morphism. Let $I$ be an open interval containing $0$.

\begin{definition}
A \textbf{deformation of $\Phi$} is a pair of smooth maps $\tilde{\Phi}: \H \times I \to \G$, and $\tilde{\phi}: N \times I \to M$ such that $\tilde{\Phi}(\cdot , 0) = \Phi$, $\tilde{\phi}(\cdot , 0) = \phi$, and for each $\e \in I$ the map 
\[\xymatrix{\calH \ar[r]^{\Phi_\e} \ar@<0.25pc>[d] \ar@<-0.25pc>[d] &\calG \ar@<0.25pc>[d] \ar@<-0.25pc>[d]\\
N \ar[r]_{\phi_\e} & M}\] 
is a Lie groupoid morphism, where $\Phi_\e = \tilde{\Phi}(\cdot , \e)$ and similarly $\phi_\e = \tilde{\phi}(\cdot , \e)$. 
\end{definition}

\begin{remark}[Fibrations]\label{deform.morph.2}
%\textcolor{brown}{%Considering the description of a Lie groupoid deformation in terms of fibrations (remark \ref{FibrationsGrpds}),
A deformation of $\Phi_0:\calH\longrightarrow\calG$ can be equivalently described by a morphism $\tilde{\Phi}$ between the trivial fibrations $(\calH\times I\rightrightarrows N\times I\longrightarrow I\rightrightarrows I)$ and $(\calG\times I\rightrightarrows M\times I\longrightarrow I\rightrightarrows I)$ covering the identity such that restricted to the fiber over 0 is $\Phi_0$. The corresponding base-map of $\tilde{\Phi}$ between the units $N\times I$ and $M\times I$ will be $\tilde{\phi}$. %To $\tilde{\Phi}$ corresponds the induced base-map $\tilde{\phi}$ between the units $N\times I$ and $M\times I$. 
%}
\end{remark}

In what follows we will denote a deformation of $\Phi: \H \to \G$ by $\Phi_\e$.
 
We will consider smooth families of gauge maps covering smooth families of maps $f_\e: N \to M$. By a smooth family of maps from $N$ to $M$ we mean a smooth map $\tilde{f}: N \times I \to M$. Similarly, a smooth family of gauge maps over $f_\e$ is a smooth map $\tilde{\tau}: N \times I \to \G$ such that $\tau_\e = \tilde{\tau}(\cdot, \e)$ is a gauge map over $f_\e = \tilde{f}(\cdot, \e)$.

\begin{definition}\label{Def:equivalent deformations}
Two deformations $\Phi_\e$ and $\Phi'_\e$ of $\Phi$ are \textbf{equivalent} if there exists a smooth family $\tau_\e: N \to \G$ of gauge maps with $\tau_0=u_\G \circ \phi_0$, where $u_\G$ denotes the identity bisection of $\G$, and such that
\begin{equation}\label{eq: gauge}
\Phi'_\e(h)=\tau_\e(t(h))\Phi_\e(h)\tau_\e(s(h))^{-1},
\end{equation}
for all $\e$ in some open interval $I$ containing $0$, and all $h \in \H$.

A deformation will be called \textbf{trivial} if it is gauge equivalent to the constant deformation.
\end{definition}

\begin{remark}
The set $\G au(N,\G)$ of gauge maps $\tau:N\to\G$ is naturally a groupoid over $C^{\infty}(N,M)$ with structure determined pointwise by the groupoid structure of $\calG \tto M$. Moreover there is a natural (left) action of this groupoid on the set $Mor(\H,\G)$ of Lie groupoid morphisms from $\H$ to $\G$, where $\tau$ acts on $\Phi:\H\to\G$ if the base map $\phi$ of $\Phi$ is equal to $s\circ\tau$, i.e., the moment map of the action is the map which associates to $\Phi: \H \to \G$, its base map $\phi: N \to M$. The action is given by
\[(\tau \cdot \Phi) (h) = \tau(t(h))\Phi(h)\tau(s(h))^{-1}.\]
With this notation, expression \eqref{eq: gauge} becomes $\Phi'_\e=\tau_\e\cdot\Phi_\e$, for all $\e$ in some open interval $I$ containing $0$.
\end{remark}

\begin{remark}
We remark that expression \eqref{eq: gauge} only makes sense if $\tau_\e$ is a family of gauge maps over $\phi_\e$. Observe also that if we regard the maps $\Phi_\e$ and $\Phi'_\e$ as smooth families of functors, then the gauge maps $\tau_\e$ are simply a smooth family of natural isomorphisms between these functors.
\end{remark}

\begin{example}[Non-trivial Deformation]
Let $\Phi_\e:\mathbb{R}\longrightarrow S^{1}\times S^{1}\subset\mathbb{C}^{2}$ be the family of morphisms given by
$$\Phi_\e(r):=(e^{2\pi ir}, e^{2\pi ir(1+\e)}),$$ where we view the Lie groups $\Rr$ and $S^1\times S^1$ as Lie groupoids over a point. This is a nontrivial deformation because $\Phi = \Phi_0$ is not injective, but $\Phi_\e$ is injective for any $\e\in\Rr \setminus \mathbb{Q}$.
\end{example}

\begin{example}\label{Ex: flat principal connections}
Let $P = N \times G \to N$ be the trivial principal $G$-bundle over $N$. There is a one-to-one correspondence between flat principal connections $\omega$ on $P$ and Lie groupoid morphisms
\[\xymatrix{
  \Pi_1(N) \ar@<0.25pc>[d] \ar@<-0.25pc>[d]  \ar[rr]^{\Phi^\omega} &  & 
 G
\ar@<0.25pc>[d] \ar@<-0.25pc>[d]\\
 N \ar[rr] & &  \{*\}.}\]
 The correspondence is obtained by using parallel transport along curves on $N$ and the canonical identification of the fibers of $P$ with $G$, i.e., if $\gamma: [0,1]  \to N$ is a path, and $\tilde{\gamma}: [0,1] \to G$ is such that $\tilde{\gamma}(0) = e$, and $(\gamma, \tilde{\gamma})$ is horizontal with respect to $\omega$, then $\Phi^{\omega}([\gamma]) = \tilde{\gamma}(1)$.
 Under this correspondence, gauge equivalence of morphisms translates to gauge equivalence of the connections. In particular, a deformation of a morphism $\Phi^\omega: \Pi_1(N) \to G$ is trivial if and only if the corresponding deformation of $\omega$ is gauge trivial.
 
In example \ref{example:vbconnections} below, we deal also with flat vector bundle connections in the context of Lie groupoid morphisms.
\end{example}

In some situations the notion of equivalence between morphisms we have defined may be too wide. The following is an example of such situation.

\begin{example}
Let $\H = N \times N \tto N$ and $\G = M \times M \tto M$ be pair groupoids. Any morphism $(\Phi, \phi)$ from $\H$ to $\G$ is of the form
\[\Phi(x,y) = (\phi(x),\phi(y)),\]
and this determines a one to one correspondence between morphism for pair groupoids and smooth maps from $N$ to $M$. Note that any two morphisms are equivalent by a gauge transformation. In fact, if $\Phi = (\Phi,\phi)$ and $\Psi = (\Psi,\psi)$ are morphisms from $N \times N$ to  $M \times M$, then $\tau = (\phi, \psi)$ is a gauge transformation such that $\tau \cdot \Psi = \Phi$. 
\end{example}

The example above suggests looking at deformations of morphisms up to stronger equivalences. An instance of that is obtained by taking bisections of groupoids (instead of gauge maps) which, in the previous example, corresponds to diffeomorphisms of the base manifold.

A smooth family of bisections of $\G$ is a smooth map $\tilde{\sigma}: M \times I \to \G$ such that $\sigma_\e = \tilde{\sigma}(\cdot, \e): M \to \G$ is a bisection for all $\e \in I$. Such a family of bisections will be denoted be $\sigma_\e$.

\begin{definition}\label{Def:strongly equivalent}
Let $\Phi: \H \to \G$ be a morphism of Lie groupoids. Two deformations $\Phi_\e$ and $\Phi'_\e$ of $\Phi$ are \textbf{strongly equivalent} if there exist an open interval $I$ containing $0$, and a smooth family of bisections $\sigma_\e$ of $\G$ such that $\sigma_0 = u$ is the identity bisection, and
\[\Phi_\e (h) = \sigma_\e(\Phi'_\e(t(h)))\Phi'_ \e(h) \sigma_\e(\Phi'_\e(s(h)))^{-1},\]
for all $\e \in I$ and $h \in \H$.
\end{definition}

In many instances, we will denote the conjugation by $\sigma_\e$ by $I_{\sigma_\e}$, i.e.,
\[I_{\sigma_\e}(g) = \sigma_\e(t(g))g\sigma_e(s(g))^{-1},\]
so that $\Phi_\e$ is strongly equivalent to $\Phi'_\e$ if $\Phi_\e = I_{\sigma_\e} \circ \Phi'_\e$. A deformation $\Phi_\e$ will be called \textbf{strongly trivial} if $\Phi_\e$ is strongly equivalent to the constant deformation $\Phi'_\e = \Phi$ for all $\e$, or in other words, if there exists a smooth family of bisections $\sigma_\e$ such that $\sigma_0 = u$, and $\Phi_\e = I_{\sigma_\e}\circ \Phi$.

\begin{remark}
Note that if $\Phi_\e$ is a strongly trivial deformation of $\Phi = \Phi_0$, then $\Phi_\e$ also can be regarded as a strongly trivial deformation of $\Phi_\l$ for any $\l\in I$. Indeed, $\Phi_{\l+\e}=I_{\sigma_{\l+\e}\star\sigma_\l^{-1}}\circ\Phi_\l$, where $\star$ denotes the product in the group of bisections of $\calG$, which is given by $(\sigma\star\tau)(x)=\sigma(t(\tau(x)))\tau(x)$; for $\sigma$ and $\tau$ bisections of $\calG$ and $x\in M$. Notice that an analogous observation also holds for trivial deformations.
\end{remark}

\begin{example}\label{example:vbconnections}
Let $E\to M$ be a vector bundle over $M$. There is a one-to-one correspondence between flat connections $\nabla$ on $E$ and Lie groupoid morphisms
\[\xymatrix{
  \Pi_1(M) \ar@<0.25pc>[d] \ar@<-0.25pc>[d]  \ar[rr]^{\Phi^\nabla} &  & 
 \mathcal{G}L(E)
\ar@<0.25pc>[d] \ar@<-0.25pc>[d]\\
 M \ar[rr]^{=} & &  M.}\]
This correspondence is obtained by using the parallel transport along curves on $M$. Under this correspondence, two connections $\nabla_1$ and $\nabla_2$ related by a gauge transformation $F\in \mathrm{Aut}(E)$ induce strongly equivalent Lie groupoid morphisms $\Phi^{\nabla_1}$ and $\Phi^{\nabla_2}$. Explicitly, if $\nabla_2=F\circ\nabla_1\circ F^{-1}$ then $\Phi^{\nabla_2}=I_{F}\circ \Phi^{\nabla_1}$, where we see $F\in \mathrm{Aut}(E)$ as a bisection of the groupoid $\G L(E)$. We just remark that for a general bisection $b\in \mathrm{Bis}(\calG L(E))$, however, the morphism $I_b\circ\Phi^{\nabla}$ does not correspond to some connection of $E$ due to the fact that its base map $f:=t\circ b\in \mathrm{Diff}(M)$ can be different from the identity $id_M$, but it does corresponds to a connection of the pullback bundle $f^{*}E$.

\end{example}

\subsubsection*{\textbf{Equivalent vs Strongly Equivalent deformations}}

For each $\phi: N \to M$ a bisection $b$ of $\G$ induces  a gauge transformation $\tau_b = b \circ \phi: N \to \G$ covering $\phi$. It then follows that strongly equivalent deformations are equivalent. Theorem \ref{Thm:equivalenttostronglyequivalent} below will deal with the converse and more subtle statement due to the fact that not all the gauge maps are obtained from bisections of $\G$. Still, Proposition \ref{Prop:gaugefrombisections} below shows that there are some cases where the gauge map can be induced by a \emph{local bisection}. Recall that a \textbf{local bisection of $\G$} is a local section $\sigma_U:U\subset M\to\G$ of the source map such that $t\circ\sigma_U:U\to M$ is a diffeomorphism onto its image.

 \begin{proposition}\label{Prop:gaugefrombisections}%CHECK FOR THE HAUSSDORF CONDITION!! HERE AND IN RESULTS BELOW.....
Let $\G$ be a Hausdorff Lie groupoid and $\tau : N \to \G$ be a gauge map covering an embedding, i.e., such that $\phi = s \circ \tau: N \to M$ is an embedding. Then there exists a local bisection $b: U \to \G$ defined on a neighbourhood $U$ of $\phi(N)$ such that $\tau = b \circ \phi$ if, and only if, $t \circ \tau$ is also an embedding.
%Any gauge map which covers and $t$-projects to embedding maps is induced by a local bisection..... iff!
\end{proposition}

To prove the proposition we will use an Ehresmann connection $H$ on $\G$ which is (also) transversal to the $t$-fibers, that is the reason of assuming $\G$ to be Hausdorff since, under non-Hausdorffness, an Ehresmann connection may not exist (see Example 13.93 in \cite{crainic-fernandes-marcut2021lectures}). The existence of such a $s$ and $t$-transversal connection is the content of Lemma \ref{lemma:biconnection} which we state next.%then we will prove the Proposition \ref{Prop:gaugefrombisections}.

\begin{lemma}\label{lemma:biconnection}
Let $s,t:\G\to M$ be the source and target maps of a Lie groupoid $\G$. If $\G$ admits an Ehresmann connection, then there exists an Ehresmann connection for the source $s$ which is also an Ehresmann connection for the target $t$.
\end{lemma}
 
\begin{proof}
The strategy to obtain this connection will be to vary locally an Ehresmann connection $H$ of $s$ in such a way that the intersections with the tangent spaces to the $t$-fibers are $\{0\}$, and then glue these local connections together by using a partition of the unity.
 
Let denote by $I_g=H_g\cap\mathrm{Ker}(d_{g}t)$ the intersection of $H$ with the tangent space at $g$ to the $t$-fibers. Then, for every $g\in\G$, $T_g\G=Ker(d_gs)\oplus I_g\oplus I_g^{\perp}$, where $I_g^{\perp}$ means the orthogonal to $I_g$ inside $H_g$ for some Riemannian metric on $\G$ such that $H$ is the orthogonal space to the $s$-fibers.

Recall that an Ehresmann connection can be equivalently described by a morphism $\omega:T\G\to t^{*}A_\G$ of vector bundles which is a left-inverse of the right-multiplication $r:t^{*}A_\G\to T\G$. This map $\omega$ is regarded as the projection of the tangent vectors of $\G$ to the $s$-vertical bundle $V^{s}\cong t^{*}A_\G$. The horizontal distribution $H$ corresponds with the kernel of the morphism $\omega$.
 
Thus, for any $g\in\G$ let $U_g$ be an open coordinate neighbourhood around $g$ such that the $t$-fiber through $g$ corresponds to a slice on $U_g$. We will define a new (local) connection $\omega_{U_g}$ on $U_g$ by modifying the values of $\omega_g$ at the part $I_g$ to be non-zero and extending, by translation, to the other points of $U_g$. In that way, by shrinking $U_g$ if needed, $\omega_{U_g}$ will be the local connection on $U_g$ which is complementar to the $s$-fibers and $t$-fibers.
 
Explicitly, if $I_g=0$ then we choose $U_g$ small enough such that, for all points in $U_g$, $I_g=0$. Otherwise, if $I_g\neq0$ for some $g$, then, in order to modify $\omega_g$, we just need to do it in a linear way. For that it suffices to take basis for $I_g$ and $A_{t(g)}$ and make an injective correspondence of the basic elements of $I_g$ with the elements of the basis of $A_{t(g)}$. We then define $\omega_{U_g}$ by extending to $U_g$ the modified $\omega_g$ by translation, say, as a constant-coefficient linear map on the coordinate neighborhood $U_g$. Shrinking $U_g$ if necessary it follows that the intersection $Ker_{h}(\omega_{U_g})\cap Ker(d_{h}t)$ is trivial for any $h\in U_g$.

Thus, consider the open cover $(U_\alpha)$ of $\G$ given by the sets $U_g$, and let $(\rho_\alpha)$ be a partition of 1 subordinated to $(U_\alpha)$, therefore
$$\omega':=\Sigma_\alpha \rho_\alpha \omega_\alpha$$
determines an Ehresmann connection $H'$ on $\G$ for the source map which is also complementar to the $t$-vertical bundle. That is, $H'$ is an Ehresmann connection for the source and the target maps of $\G$.

\end{proof}

\begin{proof}[Proof of Proposition \ref{Prop:gaugefrombisections}] 
Let $\tau$ be a gauge transformation such that $\phi = s \circ \tau$ is an embedding. It is clear that if $\tau=\sigma_U\circ\phi$ for some local bisection $\sigma_U$ of $\G$, then $t\circ \tau = (t\circ\sigma_U)\circ \phi$ is an embedding.

%********************************************

%By using tubular neighborhoods constructed with convenient metrics on the target groupoid...

For the converse statement, we first use the horizontal curves of an Ehresmann connection as in the previous Lemma to construct a local section $\sigma:U\to\G$ of the source map which contains the image $\tau(N)$ of the gauge map. Then we check that it is also a local bisection of $\G$.

Use the Ehresmann connection to make the source map a Riemannian submersion where the horizontal spaces correspond to the orthogonal spaces to the $s$-fibers. Then, by using the horizontal geodesics starting at $\tau(N)$, we obtain a submanifold $S\subset\G$ such that its projection by the source map  yields a tubular neighborhood $U$ of $\phi(N)\subset M$. This fact defines uniquely the local section $\sigma:U\to \G$ by making the correspondence of curves under the $s$-projection.

Moreover, the local section $\sigma$ is in fact a local bisection of $\G$: since the horizontal spaces are also complementar to the $t$-fibers then, by shrinking the radius of $S$ if necessary, the submanifold $S$ $t$-projects injectively to a tubular neighborhood of $t\circ\tau(N)$ in $M$. Therefore, $t\circ\sigma$ is a local diffeormorphism of $M$.
%Take a Riemannian metric on $M$ and use the normal geodesics to $\phi(N)$ to build a tubular neighborhood $U$ of $\phi(N)\subset M$. 
\end{proof}

Hence, in some cases the gauge maps are induced by (local) bisections. We can then wonder if that also happens when considering a smooth family of gauge maps as in the case of trivial deformations. It turns out that, under a compactness condition we can avoid the Hausdorffness of $\G$ and check that there exists a smoooth family of (global) bisections inducing the gauge maps. That is the content of the Proposition below.

\begin{proposition}\label{Prop:familybisectionsfromgauge}
Let $\H$ and $\G$ be Lie groupoids over $N$ and $M$, respectively, and $\Phi_0:\H\to\G$ be a morphism covering an injective immersion $\phi_0$. Assume that $N$ is compact. Then, any trivial deformation $\Phi_\e$ of $\Phi_0$ %whose family $\phi_\e$ of base-maps is given by embeddings
is indeed strongly trivial.
\end{proposition}

\begin{proof}

Let $\tau_\e$ be the smooth family of gauge maps associated to the trivial deformation $\Phi_\e$.
We will prove that there exists a smooth family $\sigma_\e$ of bisections of $\G$ which induces the gauge maps in the sense that $\tau_\e=\sigma_\e\circ\phi_0$.

%-----------------------------------------------

We will do that by considering the smooth family of local sections $\bar{\alpha}_\e$ of $A_\G$ induced by the family $\tau_\e$, and then extending them to global sections $\alpha_\e$ to obtain the global bisections we want by using the exponential map.

On the one hand, let $\Phi:\H\times I\to\G\times I$ denote the deformation $\Phi_\e$, as in Remark \eqref{deform.morph.2}, covering the map $\phi:N\times I\to M\times I$. Then, after shrinking $I$ if necessary, the map $\phi:N\times I\longrightarrow M\times I$ is an injective imersion, and also an embedding due to the fact that each $\phi_\e$ is an embedding. Moreover, the compactness of $N$ also implies that $\phi(N\times I)$ is closed inside $M\times I$.

%***************For gauge maps**********************\\

On the other hand, define the local sections $\bar{\alpha}_\e\in \Gamma(\phi_\e^{*}A_{\G})$, induced by the family of gauge maps $\tau_\e$, by
\begin{equation}\label{eq:integralcurve}
\bar{\alpha}_\e(x):=r_{\tau_\e(x)}^{-1}\left(\frac{d}{d\l}|_{\l=\e}\tau_\l(x)\right).
\end{equation}

Notice that this family of sections can be regarded as the smooth section $\bar{\alpha}\in\Gamma(\phi^{*}A_{\G\times I})$ of the pullback by $\phi^{*}$ of the Lie algebroid $A_{\G\times I}$ given by $\bar{\alpha}(x,\e):=\bar{\alpha}_\e(x)$. Thus, since the map $\phi$ is injective, the section $\bar{\alpha}\in\Gamma(\phi^{*}A_{\calG\times I})$ can be seen as a section of the restriction Lie algebroid $\left.A_{\calG\times I}\right|_{\phi(N\times I)}$. Therefore, we can extend the section $\bar{\alpha}$ to a section $\alpha$ of all $A_{\calG\times I}$ which we extend in such a way that its support is contained in an open subset $U\times I\subset M\times I$, where $U\subset M$ is an open subset containing $\phi_0(N)$ with compact closure $\bar{U}$.
%******************************************

This extended section has the form $\alpha(x,\e)=\alpha_{\e}(x)$, with $\alpha_{\e}\in\Gamma(A_\calG)$ extending $\bar{\alpha}_{\e}\in\Gamma(\phi_\e^{*}A_\calG)$. Let $\sigma$ denote the bisection of $\G\times I$ induced by the exponential flow of the section $\alpha$. Such a bisection amounts to having a smooth family $\sigma_\e:M\to\G$ of bisections of $\G$ which can be alternatively obtained by using the flow of the time-dependent vector field $\overrightarrow{\alpha}_{\e}$ on $\G$ as follows. Let $\psi^{t_1,t_0}$ denote the time-dependent flow of  $\overrightarrow{\alpha}_{\e}$. Due to the compactness of $\overline{U}$ and the vanishing of the sections $\alpha_{\e}$ outside $U$, the flow $\psi^{\e,0}$ is defined on all $M$ for $\e$ small enough.
%Lemma \ref{lemma:complete-tdflow} shows that the flow $\psi^{\e,0}$ is defined on $\phi_0(N)$ for $\e\in I$ NOOO.
Thus, the family of bisections $\sigma_\e$ of $\calG$ is given by $\sigma_{\e}(p):=\psi^{\e,0}(p)$, for $p\in M$.

In that way, the curve $\e\mapsto\sigma_\e(\phi_0(x))$ is an integral curve of the time-dependent vector field $\overrightarrow{\alpha}_\e$ starting at $\phi_0(x)$. Also, by equation \eqref{eq:integralcurve} the curve $\e\mapsto\sigma_\e(\phi_0(x))$ is an integral curve starting at the same point. Hence the family $\sigma_\e$ induces the family $\tau_\e$ of gauge maps; that is,
$$\tau_\e(x)=\psi^{\e,0}(\phi_0(x))=\sigma_\e\circ\phi_0(x).$$
Therefore, the strong triviality of the deformation follows.
\end{proof}

%%%%{\color{red}
\begin{remark}

In the previous proof, notices that the family of bisections $\sigma_\e$ can be different depending on the choice of the open set $U$ and on the extension $\alpha$ of $\bar{\alpha}$. However, if $\alpha'$ is another such an extension with induced family of bisections $\sigma'_\e$, then these two families of bisections agree when restricted to $\phi_0(N)$, which is the key fact for the proof. That follows from the last equation in the proof where one gets
$$\sigma'_\e\circ\phi_0(x)=\tau_\e(x)=\sigma_\e\circ\phi_0(x).$$

\end{remark}
%%%%}

The next theorem tells us that under certain conditions, the relations of equivalence and strongly equivalence of deformations are the same.

\begin{theorem}\label{Thm:equivalenttostronglyequivalent} 
Let $\calH$ and $\calG$ be Lie groupoids over $N$ and $M$, respectively, and $\Phi_0:\calH\to\calG$ be a Lie groupoid morphism covering an injective immersion $\phi_0$. Assume that $N$ is compact and $\G$ is Hausdorff. Then, any two deformations of $\Phi_0$ are equivalent if, and only if, are strongly equivalent.
\end{theorem}

\begin{proof}
Since a bisection induce a gauge map by composition with the base map, then two strongly equivalent deformations are clearly equivalent. For the converse question, assume that $\Phi_\e$ and $\Psi_\e$ are two equivalent deformations of $\Phi_0$. Then,
\begin{equation}\label{eq:equiv.deformat}
\Psi_\e=\tau_\e\cdot\Phi_\e;
\end{equation}
where $\tau_\e$ is a smooth family of gauge maps covering the family $\phi_\e$.

Let denote by $\Psi,\; \Phi: \calH\times I\to \calG\times I$ the morphisms covering $\psi$ and $\phi$ which encode the deformations $\Psi_\e$ and $\Phi_\e$, respectively. That is, one has $\Psi(h,\e)=(\Psi_\e(h),\e)$ and $\Phi(h,\e)=(\Phi_\e(h),\e)$. Then, the equation \eqref{eq:equiv.deformat} can also be expressed by

\begin{equation}
\Psi=\tau\cdot\Phi;
\end{equation}
where $\tau:N\times I\to\G\times I$ is the gauge map $\tau(x,\e)=(\tau_\e(x),\e)$ covering the map $\phi$. Observe then that, shrinking $I$ if necessary, since $\phi$ and $\psi$ are embedding maps, it follows from Proposition \ref{Prop:gaugefrombisections} that there exists a local bisection $\sigma$ of the groupoid $\G\times I$ such that $\tau=\sigma\circ\phi$.

Next step now consists in obtaining a global bisection inducing the gauge map $\tau$. Let $\sigma_\e$ be the smooth family of local bisections of $\G$ induced by $\sigma$. We will obtain the global bisection by considering the infinitesimal side of the local bisections $\sigma_\e$ and then using the exponential map. Indeed, let $x\in N$ and consider the vectors tangent to $\G$

$$\frac{d}{d\e}\sigma_\e\circ\phi_\e(x).$$
These vectors are not necessarily tangent to the $s$-fibers of $\calG$, for that we use an Ehresmann connection on $\G$ to project the vectors to the $s$-fibers. Let define by $\bar{\alpha}_\e\in\Gamma(\Psi_\e^{*}A_\G)$ the family of sections obtained by right translation of the projection to the $s$-fibers of the vectors $\frac{d}{d\e}\sigma_\e\circ\phi_\e(x)$, for all $x\in N$. Thus, since $\psi$ is an embedding map (shrinking $I$ if necessary), it follows that the family $\bar{\alpha}_\e$ can be regarded as a section $\bar{\alpha}\in\Gamma(A_{\G\times I}|_{\psi(N\times I)})$ over the closed submanifold $\psi(N\times I)\subset M \times I$ of the Lie algebroid $A_{\G\times I}$ of $\G\times I$. 

Therefore, we can extend the section $\bar{\alpha}$ to a global section $\alpha$ of $A_{\G\times I}$. Then finally, by using the exponential flow of $\alpha$, as in the proof of Proposition \ref{Prop:familybisectionsfromgauge}, we obtain a smooth family of bisections $\sigma_\e$ of $\G$ which induces the family of gauge maps $\tau_\e$.
\end{proof}

\section{Deformation complex $C^{\bullet}_{def}(\Phi)$ of a morphism and properties}\label{Section:Deformation complex morph.}

In this section we introduce the deformation complex of morphisms and study some of its properties.

\subsubsection*{\textbf{Deformation complex of morphisms}}
The deformation complex of a morphism was briefly discussed in \cite{CMS}, it turns out to be the appropriate complex to deal with trivial deformations of morphisms (Section \ref{Section:Triviality}). We will first recall its definition here. Let $(\mathcal H\rightrightarrows N, s_\calH,t_\calH)$ and $(\calG\rightrightarrows M, s_\calG, t_\G)$ be two Lie groupoids and $(\Phi,\phi):\calH\longrightarrow\calG$ be a morphism between them.

For any $k\in\mathbb{N}$, consider $\calH^{(k)}=\left\{(h_{1},...,h_{k}): s_\H(h_{i})=t_\H(h_{i+1})\right\}$ the manifold of $k$-strings of composable arrows of $\calH$ with $\calH^{(0)}=N$. The space of $k$-\textit{cochains} $C^{k}_{def}(\Phi)$ is given by
$$C^{k}_{def}(\Phi)=\left\{\left.c:\calH^{(k)}\rightarrow T\calG\right|\ c(h_{1},...,h_{k})\in T_{\Phi(h_{1})}\calG \text{ and } c \text{ is } s_\G\text{-projectable}\right\},$$
where $s_\G$-projectable means that the $s_\G$-projection of $c$, $ds_\G\circ c(h_{1},...,h_{k})=:s_{c}(h_{2},...,h_{k})\in T_{t(\Phi(h_2))}M$, does not depend on $h_1$. The differential of $c$ is defined by
\begin{align*}
(\delta_\Phi c)(h_1,...,h_{k+1}):&=-d\bar{m}_\calG(c(h_1h_2,h_3,...,h_{k+1}),c(h_2,...,h_{k+1}))+\\
&+\sum_{i=2}^{k}(-1)^i c(h_1,...h_ih_{i+1},...h_{k+1})+(-1)^{k+1}c(h_1,...,h_k),
\end{align*}
where $\overline{m}_\calG:\calG_s\times_s\calG\longrightarrow\calG$, $\overline{m}_\calG(g,h)=gh^{-1}$ is the division map of $\calG$.\\
For $k=0$, $C^{0}_{def}(\Phi):=\Gamma(\phi^{*}A_\calG)$ and the differential is
$$\delta_{\Phi}\alpha=\overrightarrow{\alpha}+\overleftarrow{\alpha}\in C^{1}_{def}(\Phi),$$
where $\overrightarrow{\alpha}(h)=r_{\Phi(h)}(\alpha_{t(h)})$ and $\overleftarrow{\alpha}(h)=l_{\Phi(h)}(di(\alpha_{s(h)}))$.\\

The fact that $\delta_{\Phi}^{2}=0$ follows in a similar way to the proof that $\delta^{2}=0$ in the deformation complex of Lie groupoids; so $\delta_\Phi$ in fact defines a cohomology and $H^{*}_{def}(\Phi)$ denotes the \textit{deformation cohomology of} $\Phi$. Observe that $H^{*}_{def}(Id_\calG)=H^{*}_{def}(\calG)$.

\begin{remark}\label{chain maps}
Note that, given a morphism $\Phi:\calH\longrightarrow\calG$, there are two natural cochain maps between the deformation complexes (of Lie groupoids and morphisms):

\begin{equation*}
C_{def}^{k}(\calH)\stackrel{\Phi_{*}}{\longrightarrow}C_{def}^{k}(\Phi)\stackrel{\Phi^{*}}{\longleftarrow}C_{def}^{k}(\calG)
\end{equation*}
defined by
$$(\Phi_{*}c_\H)(h_1,...,h_k):=(d\Phi)_{h_1}(c_\H(h_1,...h_k)),\ \ \ c_\H\in C^{k}_{def}(\calH);\ \text{and}$$
$$(\Phi^{*}c_\G)(h_1,...,h_k):=c_\G(\Phi(h_1),...,\Phi(h_k)),\ \ c_\G\in C^{k}_{def}(\calG).$$
Observe that in the case $k=0$ the cochain map $\Phi^{*}$ is just the pullback $\phi^{*}$ of sections $\Gamma(A_\G)\to\Gamma(\phi^{*}A_\G)$ (recall that $\phi$ is the induced map on the units). %to make sense with the definition of $C^{0}_{def}(\Phi)=\Gamma(\phi^{*}A_{\calG})$.
Similarly, if $\mathcal{K}\stackrel{\Psi}{\rightarrow}\calH$ is another Lie groupoid morphism, one can define a cochain-map $\Psi^{*}:C^{*}_{def}(\Phi)\longrightarrow C^{*}_{def}(\Phi\circ\Psi)$.
\end{remark}

\begin{remark}\label{chain maps2}
As a special case, if $\Phi$ above is bijective, there is an inverse $\Phi_{\#}$ of $\Phi^{*}$,
$$\Phi_{\#}:C^{*}_{def}(\Phi)\longrightarrow C^{*}_{def}(\calG),\ \Phi_{\#}(\hat{T})_{(g_1,...,g_k)}:=\hat{T}_{(\Phi^{-1}(g_1),...,\Phi^{-1}(g_k))}.$$
Similarly, if $\Phi^{*}$ is the map between deformation complexes of morphisms (see Remark \ref{chain maps}), then $\Phi^{*}$ also admits an inverse map $\Phi_{\#}:C^{*}_{def}(\Phi\circ\Psi)\longrightarrow C^{*}_{def}(\Phi)$ analogously defined.
\end{remark}

The deformation complex $(C^{*}_{def}(\Phi),\delta)$ has a canonical subcomplex $C^{*}_{def}(\Phi)_{s,t}$ which controls deformations of $\Phi$ that \textbf{keeps the base map $\phi$ fixed} (see remark \ref{remark on deform. with fixed base map}). It is defined by,

\begin{equation}\label{complexbasemapfixed}
C^{k}_{def}(\Phi)_{s,t}:=\{c\in C^{k}_{def}(\Phi)| ds\circ c\equiv 0\equiv dt\circ c\}.
\end{equation}
It is straightforward to check that $\delta$ preserves $C^{*}_{def}(\Phi)_{s,t}$, so $(C^{*}_{def}(\Phi)_{s,t},\delta)$ is indeed a complex.

\begin{remark}\label{remark on deform. with fixed base map}
Any deformation $\Phi_\e$ which keeps the base map fixed determines a 1-cocycle
$$\frac{d}{d\e}|_{\e=0}\Phi_\e\in C^{1}_{def}(\Phi_0)_{s,t},$$
which explains the choice of the name for this complex. This fact will be fully detailed for the more general case of arbitrary deformations of morphisms in Section \ref{Section:Triviality}.
\end{remark}

\begin{example}
Deformations of vector bundle connections as in example \ref{example:vbconnections} above and deformations of representations, as in Section \ref{Section:Examples}, give instances of deformations of morphisms where the base map is kept fixed.
\end{example}

By using right multiplication to translate the tangent vectors $c(h_1,...,h_k)\in T_{\Phi(h_1)}\G$ over the unit elements of $\G$, the complex $C_{s,t}^{*}(\Phi)$ takes a simpler description as the differentiable complex $C^{*}_{\mathrm{diff}}(\calH,\Phi^{*}\mathfrak{i}_{\calG})$ with values in the pullback by $\Phi$ of the (maybe singular) isotropy bundle $\mathfrak{i}_\calG$ of $\calG$. %and where the differential depends on the morphism $\Phi$.
The action of $\H$ on $\Phi^{*}\mathfrak{i}_\G$ is given by the pullback by $\Phi$ of the adjoint action of $\G$ on $\mathfrak{i}_\G$, as defined in remark \ref{rmk:adjoint action}.
We observe that even in the case of a singular isotropy bundle one can still make sense of the differentable complex by defining smooth functions with values in $\mathfrak{i}_\G$ as those smooth functions with image in $A_\G$ which take values in $\mathfrak{i}_\G$.

\subsection{Alternative description of the deformation complex}\label{alternativedescription}

The deformation complex of morphisms can be expressed in a different way by using the theory of VB-groupoids. This interpretation will be useful to get a better understanding of the properties of this cohomology; in particular, it will allow us to obtain the vanishing results in a simple way. Note first that, %\begin{remark}\label{VBcomplex2}
if $\mathcal{V}\rightrightarrows E$ is a VB-groupoid over $\calG\rightrightarrows M$ with source and target $\tilde{s}$ and $\tilde{t}$ and with core-bundle $C$ then, as pointed out in \cite{GrMe-grpds}, the VB-complex $C^{\bullet}_{\mathrm{VB}}(\V^{*})\subset C^{\bullet}_{\mathrm{lin}}(\V^{*})$ of $\V^{*}$ can be identified with a new complex which we denote by $(C^{\bullet}_{VB}(\calG,\V),\delta)$, where $C^{0}_{VB}(\calG,\V):=\Gamma_{M}(C)$ and
\begin{equation}\label{VBcomplex2}
C^{k}_{VB}(\calG,\V):=\{c\in\Gamma_{\calG^{(k)}}(pr_{k}^{*}\V); c(g_1,...,g_k)\in\V_{g_1} \text{ and } c \text{ is } \tilde{s}\mathrm{-projectable}\}
\end{equation}
for $k>0$; and whose differential is the one induced by the differential of $C^{\bullet}_{\mathrm{VB}}(\V^{*})$ under the identification. Here, one says that $c$ is $\tilde{s}$-\textit{projectable} if $\tilde{s}(c(g_1,...,g_k))$ does not depend on $g_1$.
%\end{remark}

In fact, the identification is given by
\begin{align}
C^{k}_{VB}&(\calG,\V) & \longrightarrow & & C^{k}_{VB}&(\V^{*}) \label{isoMehta}\\
&c & \longmapsto & & &\tilde{c}:\ \ \ (\eta_{g_1},...,\eta_{g_k})\mapsto\left\langle \eta_{g_1},c(g_1,...,g_k)\right\rangle \nonumber,
\end{align}
where $\eta_{g_i}$ is in the fiber $\V^{*}_{g_i}$ over $g_i$.

With this at hand, it is straightforward to check that the deformation complex $C^{*}_{def}(\Phi)$ of a morphism $(\Phi,\phi):\calH\rightarrow\calG$ is the complex $C^{*}_{VB}(\calH,\Phi^{*}T\calG)$, where $\Phi^{*}T\calG$ is the pullback by $\Phi$ of the tangent VB-groupoid $T\calG$. Thus we have,

\begin{proposition}
The map \eqref{isoMehta} above induces an isomorphism between the complexes $C^{*}_{def}(\Phi)$ and $C^{*}_{VB}(\phi^{*}T^{*}\calG)$. Moreover, after the choice of a connection $\sigma$ on $\calG$, $C^{*}_{def}(\Phi)$ is isomorphic to $C(\calH,\Phi^{*}Ad^{\sigma}_{\calG})^{*}:=C^{*}(\calH,\phi^{*}A_\calG)\oplus C^{*-1}(\calH,\phi^{*}TM)$, which is the complex with the structure of pullback adjoint representation (up to homotopy) of $\calG$ induced by $\sigma$. So, in particular, $H_{def}^{*}(\Phi)\cong H^{*}(\calH,\Phi^{*}Ad_{\calG})$.
\end{proposition}
\begin{sketchproof}
%\begin{proof}
The proof of the second claim follows from the interpretation of VB-complexes as 2-terms representations up to homotopy of groupoids after the choice of a unitary splitting of the core sequence \eqref{coresequence} (see \cite{GrMe-grpds}). In particular, under this philosophy, the VB-complex $C^{\bullet}_{VB}(T^{*}\G)$ of $T^{*}\calG$ induces on $C(\calG,A_\calG\oplus TM)^{\bullet}:=C^{\bullet}(\calH,A_\calG)\oplus C^{\bullet-1}(\calH,TM)$ the structure of the adjoint representation of $\calG$; in this case an unitary splitting of the core sequence of $\calG$ amounts to a connection on $\calG$. Similarly, by using the connection $\sigma$, one can see the VB-complex $C^{*}_{VB}(\Phi^{*}T^{*}\calG)$ as the pullback $C(\calH,\phi^{*}A_\calG\oplus\phi^{*}TM)^{*}$ by $\Phi$ of the adjoint representation $Ad^{\sigma}_\calG$ of $\calG$. Thus, by the isomorphism \eqref{isoMehta} above, the deformation complex of morphisms is isomorphic to the complex representing the pullback of the adjoint representation of $\calG$.
\end{sketchproof}

\begin{remark}
One can check that a proof of the second claim in the previous Proposition can also be made by some explicit computations after the choice of an Ehresmann connection of $\calG$, in the same way as the proof of Lemma 9.1 in \cite{CMS}.
\end{remark}

\begin{remark}
We notice that combining the identification \eqref{isoMehta} $C^{*}_{def}(\Phi)\to C^{*}_{VB}(\Phi^{*}T^{*}\G)$ and the quasi-isomorphism \cite{CD} $C^{*}_{VB}(\Phi^{*}T^{*}\G)\hookrightarrow C^{*}_{lin}(\Phi^{*}(T^{*}\G))$ towards the linear complex $C^{*}_{lin}(\Phi^{*}(T^{*}\G))$, associated to the VB-groupoid $\Phi^{*}T^{*}\G$ , one gets the identification of the cohomologies
$$H^{*}_{def}(\Phi)\longrightarrow H^{*}_{lin}(\Phi^{*}T^{*}\G).$$
Explicitly, the composition of the maps
$$C^{*}_{def}(\Phi)\longrightarrow C^{*}_{VB}(\Phi^{*}T^{*}\G)\hookrightarrow C^{*}_{lin}(\Phi^{*}T^{*}\G)$$
induces a quasi-isomorphism of complexes.
\end{remark}

\subsection{Properties and Variation of the complex $C^{\bullet}_{def}(\Phi)$}\label{Subsection:Variation complex}

According to the terminology of equivalence of deformations of morphisms we can also say that two morphisms of Lie groupoids are \textbf{equivalent} if they are isomorphic when viewed as functors. The following result then tell us that the deformation cohomology of morphisms is an object associated to the equivalence classes of morphisms.

\begin{theorem}\label{Theor:IsomorpMorphisms-Cohomology}
Let $\Phi$, $\Psi:\calH\to\calG$ be two Lie groupoid morphisms. If $\Phi$ and $\Psi$ are equivalent then their deformation cohomologies $H^{*}_{def}(\Phi)$ and $H^{*}_{def}(\Psi)$ are isomorphic.
\end{theorem}

For the proof of this Theorem we first recall the \textbf{groupoid of arrows} $\G^{I}=\G\times_M \calG\times_M \G$ of $\G$, where the objects of $\G^{I}$ are the arrows of $\G$ and the arrows of $\G^{I}$ are given by commutative squares of arrows in $\G$ or, in other words, by three arrows $(g,h,k)\in\G^{I}$ with common sources.
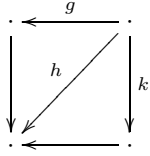
\[
\xymatrixrowsep{0.5in}
\xymatrixcolsep{0.5in}
\xymatrix{
  \cdot \ar[d] & \cdot \ar[l]_{g} \ar[ld]_{h} \ar[d]^{k}\\
  \cdot & \cdot \ar[l]}
   \]
  where the source, target and multiplication maps are respectively given by the right and left arrows of the commutative square, and by the concatenation of the diagramas. That is, $s(g,h,k)=k$, $t(g,h,k)=gh^{-1}$ and $(g',h',k')\cdot(g,h,k)=(g'g,h'g,k)$. With these notations, it is straightforward to see that the maps which take the upper and lower arrows of the square are Morita morphisms. Explicitly,

\begin{lemma}
Let $u$ and $l$ denote the maps $p_1,\; \overline{m}\circ p_{2,3}:\calG^{I}\to\calG$ given by
$$p_1:(g,h,k)\mapsto g,\ \ \text{and}$$
$$\overline{m}\circ p_{2,3}:(g,h,k)\mapsto hk^{-1}$$
which take the upper and lower arrows of a commutative square in $\calG^{I}$. Then, the maps $u$ and $l$ are Morita morphisms covering the source $s:\G\to M$ and target maps $t:\G\to M$ of $\G$, respectively. 
\end{lemma}

\begin{proof}[Proof of Theorem \ref{Theor:IsomorpMorphisms-Cohomology}]

Let $\tau:\Phi\to\Psi$ be the gauge map which relates the morphism $\Phi$ to $\Psi$.

Notice that the gauge map $\tau$ can be seen as the Lie groupoid morphism
$$\tilde{\tau}:\calH\longrightarrow \calG^{I}$$
$$\tilde{\tau}(h):=(\Phi(h),\tau(t(h))\Phi(h),\tau(s(h))).$$
Such a morphism encodes the isomorphic morphisms $\Phi$ and $\Psi$ by composing with the \emph{upper} and \emph{lower morphisms} $u$ and $l$ of $\G^{I}$, indeed, $\Phi=u\circ\tilde{\tau}$ and $\Psi=l\circ\tilde{\tau}$.

We will now prove that the cohomologies $H^{*}_{def}(\Phi)$ and $H^{*}_{def}(\Psi)$ are isomorphic to $H^{*}_{def}(\tilde{\tau})$. Indeed, these isomorphisms follow from observing that if $f:\H\to\G$ is any morphism, then a Morita map $F:\calG\to\calK$ induces the quasi-isomorphism $F_{*}:C_{def}^{*}(f)\longrightarrow C_{def}^{*}(F\circ f)$. This is the content of Proposition \ref{Mor2} in Section \ref{Sec:Moritainv} where we prove the isomorphism between deformation cohomologies by using the notion of VB-Morita maps. Therefore, this fact along with the previous Lemma tell us that the upper and lower morphisms induce isomorphisms in the cohomologies, as we wanted.
\end{proof}

%*********************************

We describe now a variation of the deformation complex of morphisms which should be though of as the tangent space to the Moduli space of morphisms with the relation induced by bisections. Indeed, such a variation will be relevant for us in order to deal with the characterization of strongly trivial deformations. It just consists of changing the space $C^{0}_{def}(\Phi)=\Gamma(\phi^{*}A_\G)$ of sections of the pullback algebroid by the space of pullback sections $\phi^{*}(\Gamma(A_\G))$. We denote this complex by $\tilde{C}^{\bullet}_{def}(\Phi)$. Explicitly,
$$\tilde{C}^{0}_{def}(\Phi)=\phi^{*}(\Gamma(A_\G))$$ and
$$\tilde{C}^{k}_{def}(\Phi)=C^{k}_{def}(\Phi),\ \text{for } k>0.$$
Observe that the cohomology $\tilde{H}^{\bullet}_{def}(\Phi)$ is larger than the usual deformation cohomology $H^{\bullet}_{def}(\Phi)$. This fact is not a surprise in view that the space of trivial deformations is larger than the space of strongly trivial deformations.

%%%%%%%%%%%%%%%%%%%%%%%%%% Section %%%%%%%%%%%%%%%%%%%%%%%%%%%%%%

\section{Examples}\label{Section:Examples}

Representations of Lie groupoids and flat connections can be viewed as morphisms of Lie groupoids. We study below their deformation cohomologies when regarded as morphisms.

Recall that a representation of a Lie groupoid $\G\tto M$ on a vector bundle $E\to M$ is a Lie groupoid morphism $\Phi:\calG\to\calG L(E)$, covering the identity, from $\G$ to the General Linear groupoid. One can naturally consider deformations of representations of Lie groupoids as a special instance of deformations of morphisms of Lie groupoids. We define a \textbf{deformation of a representation of a Lie groupoid} $\calG$ \textbf{on a vector bundle} $E\to M$ as a deformation $\Phi_\e$ of the morphism $\Phi:\calG\to\calG L(E)$ which keeps the base map fixed (which is the identity). Therefore, the complex $C^{\bullet}_{s,t}(\Phi)$ (see expression \eqref{complexbasemapfixed}) should control the deformations of the representation $\Phi$. 

The representation $\Phi$ induces a canonical action of $\calG$ on the vector bundle $End(E)$ of endomorphisms of $E$ as follows:  
\begin{equation}\label{Eq:Representations}
g\cdot F_{s(g)}:= \Phi(g)\circ F_{s(g)}\circ \Phi(g)^{-1},
\end{equation}
where $g:x\to y$ is an element of $\calG$ and $F_{x}$ denotes an element in the fiber $End(E)_x$ of $End(E)$ over $x$.
Thus, since the complex $C^{\bullet}_{s,t}(\Phi)$ can be described in terms of the isotropy bundle of $\G L(E)$ and $\mathfrak{i}_{\calG L(E)}=End(E)$, it follows that the complex controlling the deformations of the representation $\Phi$ of $\calG$ is $C^{\bullet}_{\mathrm{diff}}(\calG, End(E))$, where the differential is induced from the canonical action \eqref{Eq:Representations} defined above. By other side, one can check that the deformation cohomology $H^{\bullet}_{\mathrm{def}}(\Phi)$ agrees with such a $H^{\bullet}_{s,t}(\Phi)$.

\begin{proposition}[Representations of Lie groupoids]\label{Prop:Representations}
Let $\Phi:\calG\to\calG L(E)$ be a representation of $\calG\tto M$ on the vector bundle $E\to M$. Then $H_{\mathrm{def}}^{*}(\Phi)\cong H_{\mathrm{diff}}^{*}(\calG;\ End(E)).$
\end{proposition}

\begin{proof}
This isomorphism in cohomology can be checked directly as the induced by the right translation $r: C_{\mathrm{diff}}^{\bullet}(\calG;\ End(E))\longrightarrow C_{def}^{\bullet}(\Phi)$ after notice two things: \emph{(i)} that $End(E)$ is the isotropy bundle of Lie algebras of the Lie algebroid $\mathfrak{gl}(E):=Lie(\calG L(E))$ and \emph{(ii)} that the action of $\calG$ on $End(E)$ defined above agrees with the canonical adjoint action described in Remark \ref{rmk:adjoint action}. In Section \ref{Section:Regularsetting} we give an alternative way to check this isomorphism (see example \ref{Regular setting, morphisms}).

%We remark that this isomorphism turns out to be just a particular case of Example \ref{Regular setting, morphisms}.
\end{proof}

The previous Proposition tell us then that the usual deformation cohomology $H^{\bullet}_{\mathrm{def}}(\Phi)$ of $\Phi:\G\to \G L(E)$ is the one that controls deformations of the representation $E$ of $\G$.

\begin{remark}\label{rmk:transitivetargetgrpd}
Similarly one can see that if the target groupoid, for any morphism $\Phi:\calH\to\calG$, is transitive then all the information of a deformation of $\Phi$ is concentrated just on the cohomology of $\H$ with values in the isotropy bundle of $\calG$. In some sense that means that the non-trivial deformations of $\G$ will be determined by the isotropy directions of $\G$.
\end{remark}

\begin{remark}
In the special case of a representation of $\calG\tto M$ on $TM$, we can get an alternative view-point in terms of 1-jets: the groupoid $\calG L(TM)$ can be though as the Lie groupoid $J^{1}(M,M)$ of 1-jets of local diffeomorphisms of $M$, and its Lie algebroid corresponds to the set $J^{1}(TM)$ of 1-jets of vector fields on $M$. The vector bundle $End(TM,TM)$ translates to the set $J^{1}_{0}(TM):=\{j^{1}_{x}X; X \in \mathfrak{X}(M) \text{ and } X(x)=0\}$ of 1-jets of vector fields on its singular points. The action of $\calG$ on $J^{1}_{0}(TM)$ will be obtained by pulling back the canonical action of $J^{1}(M,M)$ on $J^{1}_{0}(TM)$ given by 
$$j^{1}_{x}\Phi\ast j^{1}_{x}X:=j^{1}_{\Phi(x)}(\Phi_{*}X).$$
\end{remark}

%\subsection{Connections on vector bundles***}

\begin{example}[Connections on vector bundles]
Example \ref{example:vbconnections} allows us to see any flat connection $\nabla$ on a vector bundle $E\to M$ as a Lie groupoid morphism $\Phi^{\nabla}:\Pi_1(M)\to \G L(E)$ covering the identity map $Id_M$. Equivalently, $\nabla$ is regarded as a representation of $\Pi_1(M)$.

Thus, we define a \textbf{deformation of $\nabla$ by flat connections} as a deformation of the associated morphism $\Phi^{\nabla}$ which keeps the base map fixed (which is the identity).

Hence, by Remark \ref{rmk:transitivetargetgrpd}, the deformation cohomology controlling deformations of $\nabla$ by flat connections should be
$$H^{\bullet}_{s,t}(\Phi^{\nabla})\cong H^{\bullet}_{\mathrm{diff}}(\Pi_1(M);End(E)).$$
In particular, when $M$ is a simply connected manifold, since $\Pi_1(M)$ will be proper, it follows that
$$H^{k}_{s,t}(\Phi^{\nabla})=0,\ \text{for all } k>0.$$
In that case, the Theorem \ref{gauge-triviality} tells us that the flat connections can be deformed just in a \textit{trivial} manner.
\end{example}

\begin{example}[Flat principal connections]
Example \ref{Ex: flat principal connections} regards any flat principal connection $\omega$ on the trivial $G$-principal bundle over $N$ as a Lie groupoid morphism $\Phi^{\omega}:\Pi_1(N)\to G$ uniquely. 

Therefore, by Remark \ref{rmk:transitivetargetgrpd}, the cohomology controlling deformations of $\omega$ by flat principal connections is 
$$H^{\bullet}_{s,t}(\Phi^{\omega})\cong H^{\bullet}_{\mathrm{diff}}(\Pi_1(N),\mathfrak{g}_N),$$
where $\mathfrak{g}_N$ is the trivial vector bundle with fiber $\mathfrak{g}=Lie(G)$ over $N$.
\end{example}

\begin{example}[Morse Lie groupoid morphisms]
The notion of Morse-Lie groupoid morphism has been defined recently in \cite{ortiz2022morse} as a morphism $(F,f):\calG\tto M\to\mathbb{R}\tto\mathbb{R}$ towards the unital groupoid over $\mathbb{R}$ such that every \emph{critical orbit} of $f$ is non-degenerate.
Indeed, the critical points of a Lie groupoid morphism towards the unital groupoid come in saturated subspaces, in that sense the non-degeneracy of a critical orbit is determined in terms of the non-degeneracy of the normal Hessian of $f$. Even though, a Morse-Lie groupoid morphism is in particular a morphism towards the unital groupoid over $\mathbb{R}$, which is a regular groupoid with no isotropy . Therefore its deformation cohomology groups can be computed according the following sequence induced from  \eqref{RegularVBmorphisms} for every $k>0$
$$0\to H^{k}_{def}(F)\to H^{k-1}(\calG, f^{*}\nu_{\mathbb{R}})\to 0.$$
The degree zero cohomology vanishes and is computed according Example \ref{cohomologydegreezero}.
More explicitly, the cohomology groups are given in terms of the differentiable cohomology
$$H^{k}_{def}(F)\cong H^{k-1}(\calG,\mathbb{R}).$$
Therefore, it turns out that the cohomology does not depend on the morphism $F$ but only on the groupoid $\calG$.

\end{example}

%%%%%%%%%%%%%%%%%%%%%%%%%% Section %%%%%%%%%%%%%%%%%%%%%%%%%%%%%%

\section{Low degrees and vanishing}\label{section:lowdegrees}

In this section we will describe the deformation cohomology groups in low degrees. A central point here will be the description of the deformation complex in terms of VB-groupoids. The content of this section will be key to show the stability properties of morphisms.

Let $\V\rightrightarrows E$ be a VB-groupoid over $\calG\rightrightarrows M$, with core $C$ and core-anchor map $\partial: C\rightarrow E$. We consider the following (possibly singular) vector bundles
$$\mathfrak{k}=Ker\: \partial\ \ \ \text{and}\ \ \ \mathfrak{l}=Coker\: \partial$$
over $M$. It is a known fact that, even though $\frakk$ and $\frakl$ may be singular, $\V$ induces canonical actions of $\calG$ on them (see Section \ref{VB-grpds} or \cite{GrMe-grpds}). In particular they turn out to be actual representations if $\V$ is a regular VB-groupoid (see Section \ref{Section:Regularsetting}). Nevertheless, we can make sense of the cohomologies with values in $\frakk$ and $\frakl$ in the singular case, as we show below. We are mainly interested in the low degree cohomologies. Define the `smooth sections' of $\frakk$ and $\frakl$ by

$$\Gamma(\frakk):=C^{0}(\calG,\frakk)=\left\{\sigma\in\Gamma(C)| \sigma(x)\in\frakk_{x}\subset C_{x}\right\}$$
and
$$\Gamma(\frakl):=C^{0}(\calG,\frakl)=\frac{\Gamma(E)}{Im (\partial)},$$
where we are looking at $\partial$ as the induced map on sections $\Gamma(C)\rightarrow\Gamma(E)$.

\begin{remark}
Note that the definition of $C^{0}(\calG,\frakk)$ as the subspace of sections in $\Gamma_{M}(C)$ with values in $\frakk\subset C$ has the direct generalization to define $C^{k}(\calG,\frakk)$ as the subspace of sections in $\Gamma_{\calG^{(k)}}(t_k^{*}C)$ which take values in $\frakk$; where $t_k:\calG^{(k)}\rightarrow M$ is given by $t_k(g_1,...,g_k)=t(g_1)$. Also, the canonical action of $\calG$ on $\frakk$ allows us define a differential $\delta:C^{k}(\calG,\frakk)\rightarrow C^{k+1}(\calG,\frakk)$ with the same formula as that of the differentiable cohomology of $\calG$ with values in a representation. Such a differential makes $(C^{\bullet}(\calG,\frakk),\delta)$ a cochain complex, whose cohomology we denote by $H^{\bullet}(\calG,\frakk)$. Moreover, this complex with values in $\mathfrak{k}$ can be viewed as a subcomplex of the VB-complex $C^{\bullet}_{VB}(\calG,\Gamma)$ through the right multiplication by zero elements:
\begin{equation}\label{eq:r}
r(c)(g_1,...,g_k):=c(g_1,...,g_k)\cdot0_{g_1}\ \ \in\ \V_{g_1},\ \ c\in C^{k}(\calG,\frakk).
\end{equation}
\end{remark}

Following the idea of Lemma 4.5 and Definition 4.6 of \cite{CMS}, we can make sense of $H^{0}(\calG,\frakl)$ as being the invariant sections of $\frakl$.

\begin{definition}
Let $V\in \Gamma(E)$. We say that $[V]\in\Gamma(\frakl)$ is \textit{invariant} if there exists a section $X\in\Gamma_{\calG}(\V)$ which is both $\tilde{s}$-projectable and $\tilde{t}$-projectable to $V$. In other words, $X$ is an $(\tilde{s},\tilde{t})$-lift of $V$. We denote the space of invariant elements by
$$H^{0}(\calG,\frakl)=\Gamma_{M}(\frakl)^{inv}.$$
\end{definition}

Observe that when $\frakk$ and $\frakl$ are vector bundles, the previous definitions agree with the usual ones of cochains with values in a representation of $\calG$. With the general setting above, we obtain the following two propositions related to the low degree cohomologies.

\begin{proposition}
If $\V$ is a VB-groupoid over $\calG$ one has $H^{0}(\calG,\V)\cong H^{0}(\calG,\frakk)=\Gamma(\frakk)^{inv}$.
\end{proposition}

\begin{proof}
It is clear that the differential of $C^{\bullet}(\calG,\V)$ on a 0-cochain $\alpha\in\Gamma(C)$ is given by $\delta(\alpha)(g)=\alpha_{t(g)}\cdot 0_{g}+0_{g}\cdot i_{\V}(\alpha_{s(g)})$. Thus, since $\tilde{t}(\delta(\alpha)(g))=\partial(\alpha(t(g)))$ then $\alpha$ is a 0-cocycle if and only if $\alpha\in\Gamma(\frakk)$ and $\alpha_{t(g)}=-0_{g}\cdot i_{\V}(\alpha_{s(g)})\cdot 0_{g}^{-1}=0_{g}\cdot\alpha_{s(g)}\cdot 0_{g^{-1}}$. That is, if and only if $\alpha\in\Gamma(\frakk)^{inv}$.
\end{proof}

Thus, in particular we obtain,

\begin{example}\label{cohomologydegreezero} 

For a morphism of Lie groupoids $(\Phi,\phi):\calG\rightarrow\calH$, $H^{0}_{def}(\Phi)\cong\Gamma(\phi^{*}\mathfrak{i}_{\calG})^{inv}$.

\end{example}

\begin{proposition} Let $\V$ be a VB-groupoid as above. Then we have the following exact sequence
\begin{equation}\label{singularsequence}
0\longrightarrow H^{1}(\calG, \mathfrak{k})\stackrel{r}{\longrightarrow} H^{1}_{VB}(\calG,\V)\stackrel{\pi}{\longrightarrow}\Gamma(\mathfrak{l})^{inv}\stackrel{K}{\longrightarrow} H^{2}(\calG, \mathfrak{k})\stackrel{r}{\longrightarrow} H^{2}_{VB}(\calG,\V)
\end{equation}
where the maps $r,\ \pi$ and $K$ are determined as follows:

the map $r$ is induced by the cochain map (denoted again by) $r$ defined in \eqref{eq:r}; $\pi$ is induced by the $\tilde{s}$-projection of the elements of $C^{1}_{VB}(\calG,\V)$ to the sections $\Gamma_{M}(E)$ of the side bundle; and $K$ takes an invariant element $[V]$ to $\delta(X)\in C^{2}(\calG,\frakk)\stackrel{r}{\hookrightarrow}C^{2}_{VB}(\calG,\V)$, where $X$ is any $(\tilde{s},\tilde{t})$-lift of $V$.
\end{proposition}

\begin{proof}
The proof of this sequence is entirely analogous to that of the sequence in Proposition 4.11 of \cite{CMS}. We just remark that $\delta(X)$ above in fact has image in $C^{2}(\calG,\frakk)$: it suffices to observe that $\tilde{s}(\delta(X))=0=\tilde{t}(\delta(X))$. The fact that $K$ is well-defined was already shown in Lemma 4.9 in \cite{CMS}.
\end{proof}

As particular cases of the previous sequence we obtain some key sequences in the context of deformation cohomologies. 

\begin{examples}\label{particularcases}\ 
\begin{enumerate}

	\item It is straightforward to see that if $\V=T\calG$ then the sequence above reproduces the sequence in Proposition 4.11 of \cite{CMS} for the deformation complex of Lie groupoids.
	
	\item Let $(\Phi,\phi):\calH\rightarrow\calG$ be a morphism of Lie groupoids. By taking the VB-groupoid $\V=\phi^{*}T\calG$ we obtain a sequence for the deformation complex of the morphism $\Phi$:
\begin{equation}\label{lowdegreemorphism}
0\rightarrow H^{1}(\calH, \phi^{*}\mathfrak{i}_\calG)\stackrel{r}{\rightarrow} H^{1}_{def}(\Phi)\stackrel{\tilde{s}}{\rightarrow}\Gamma(\phi^{*}\nu_\calG)^{inv}\stackrel{K}{\rightarrow} H^{2}(\calH, \phi^{*}\mathfrak{i}_\calG)\rightarrow H^{2}_{def}(\Phi),
\end{equation}
where $\i_\calG$ is the isotropy bundle of $\calG$ and $\nu_\calG$ is the normal bundle to the orbits of $\calG$.

\end{enumerate}
\end{examples}

%\textcolor{red}{remark*** computation of the cohomology for subgrpds in terms of the mapping-cone.}

\subsubsection*{\textbf{Vanishing of cohomologies}}
Here we state the vanishing results for the deformation cohomology. The proofs are straightforward applications of the VB-groupoid interpretation of the deformation complex, the vanishing result of the VB-cohomology \cite{CD} and of the sequence \eqref{lowdegreemorphism}.

\begin{proposition}
Let $\Phi:\calH\rightarrow\calG$ be a morphism of Lie groupoids. If $\calH$ is proper, then $H^{0}_{def}(\Phi)\cong\Gamma(\phi^{*}i_\calG)^{inv}$, $H^{1}_{def}(\Phi)\cong\Gamma(\phi^{*}\nu_\calG)^{inv}$ and $H^{k}_{def}(\Phi)=0$ for every $k\geq2$, where $\nu_\calG$ is the normal bundle to the orbits of $\calG$.
\end{proposition}

\begin{remark}
Alternatively, a direct proof of the vanishing of cohomology for morphisms also can be made in an analogous way to that of the vanishing of the deformation cohomology of proper Lie groupoids in \cite{CMS}.
\end{remark}

\section{Regular setting}\label{Section:Regularsetting}

In this section we show that the sequence \eqref{singularsequence} is just part of a long exact sequence when we impose some regularity conditions on the groupoids. Later we illustrate the long exact sequences which can be deduced from it. We say that a VB-groupoid $\V\rightrightarrows E$ is \textbf{regular} if its core-anchor map $\partial:C\to E$ has constant rank.
 %In particular, the sequences in the examples below are of interest for us: they concern the deformation cohomologies.

\begin{theorem}
Let $\calG\rightrightarrows M$ be a Lie groupoid and $\V\rightrightarrows E$ be a regular VB-groupoid over $\calG$. Then there exists a map $K:H^{\bullet}(\calG, \frakl)\longrightarrow H^{\bullet+2}(\calG, \frakk)$ such that the cohomology $H^{*}_{VB}(\calG,\V)$ associated to $\V$ fits into the long exact sequence
\begin{equation}
\cdots\longrightarrow H^{k}(\calG, \mathfrak{k})\stackrel{r}{\longrightarrow} H^{k}(\calG,\V)\stackrel{\pi}{\longrightarrow} H^{k-1}(\calG, \mathfrak{l})\stackrel{K}{\longrightarrow} H^{k+1}(\calG, \mathfrak{k})\longrightarrow\cdots,
\end{equation}
where $r$ and $\pi$ are the maps induced by the right multiplication by zero elements of $\V$ and the $\tilde{s}$-projection of elements in $\V$.
\end{theorem}

\begin{proof}
The proof of this theorem is an adaptation of the one of Proposition 8.1 in \cite{CMS} for the deformation cohomology of regular Lie groupoids. Note that the regularity condition on the VB-groupoid $\Gamma$ tells us that the associated cohomology induced by the complex $C\stackrel{\partial}{\longrightarrow}E$ associated to $\Gamma$ is given by the cohomology bundles $\mathfrak{k}=\mathrm{Ker(\partial)}$ and $\mathfrak{l}=\mathrm{Coker}(\partial)$ over $M$, where $C$ and $E$ are the core and side bundles of $\Gamma$. With this setting we construct the complexes $\mathcal{C}^{\bullet}$ and $\mathcal{A}^{\bullet}$ which fit into the following exact sequences

\begin{equation}\label{first}
0\longrightarrow C^{\bullet}(\calG,\mathfrak{k})\stackrel{r}{\longrightarrow} C^{\bullet}(\calG,\Gamma)\stackrel{R}{\longrightarrow}\mathcal{C}^{\bullet}\longrightarrow 0
\end{equation}

\begin{equation}\label{second}
0\longrightarrow \mathcal{C}^{\bullet}\longrightarrow \mathcal{A}^{\bullet}\stackrel{S}{\longrightarrow} C^{\bullet}(\calG,\mathfrak{l})\longrightarrow 0,
\end{equation}
where $\mathcal{A}^{\bullet}$ is acyclic. Namely, $\mathcal{A}^{\bullet}$ is defined by
$$\mathcal{A}^{k}=C^{k}(\calG, E)\oplus C^{k-1}(\calG, E)$$
with differential $\delta(\omega,\eta)=(\delta'\omega, \omega-\delta'\eta)$, where $\delta':C^{\bullet}(\calG,E)\longrightarrow C^{\bullet+1}(\calG,E)$ is expressed by
$$\delta'(\omega)(g_1,...,g_{k+1})=\sum^{k}_{i=1}(-1)^{k+1}\omega(g_1,...,g_ig_{i+1},...,g_{k+1})+ (-1)^{k+1}\omega(g_1,...,g_k).$$
It is straightforward to check that $\mathcal{A}^{\bullet}$ is acyclic: it is the mapping cone of the identity $Id_{C^{*}(\calG,E)}$.

The map $S:\mathcal{A}^{\bullet}\longrightarrow C^{\bullet}(\calG,\mathfrak{l})$ is given by
$$S(\omega,\eta)(g_1,...,g_k)=[\omega(g_1,...,g_k)]-g_1\cdot[\eta(g_2,...,g_k)],$$
where $[V]\in\mathfrak{l}$ denotes the class of $V\in E$ in $\mathfrak{l}$ and `$g_1\cdot$' the action of $\calG$ induced by $\Gamma$ on $\mathfrak{l}$. This map $S$ is compatible with the differentials. The complex $\mathcal{C}^{\bullet}$ is taken as being the kernel of $S$, and $R$ takes a cochain $c\in C^{k}(\calG,\Gamma)$ to the pair
$$R(c)=(\omega_c,\eta_c)\in C^{k}(\calG, E)\oplus C^{k-1}(\calG, E),$$
where $\omega_c$ and $\eta_c$ are the $\tilde{t}$ and $\tilde{s}$-projection of $c$, respectively. Observe that the definition of the quasi-action of $\calG$ on $E$, induced by $\Gamma$, implies that $R$ in fact takes values in $\mathcal{C}^{\bullet}$ and hence $\mathrm{Im} R\subset\mathrm{Ker} S$. Again, $R$ is compatible with the differentials. By choosing a splitting $\sigma$ of the core-sequence
$$t^{*}C\stackrel{r}{\longrightarrow}\Gamma\stackrel{\tilde{s}}{\longrightarrow} s^{*}E$$
it is possible to show, analogous to the proof of the Proposition 8.1 in \cite{CMS}, that $\mathrm{Ker} S\subset\mathrm{Im} R$. It is also clear that $\mathrm{Im}\: r=\mathrm{Ker} R$. The surjectivity of $S$ follows from the surjectivity of the projection $E\stackrel{p}{\longrightarrow}\mathfrak{l}$; in particular, if $[\omega]\in C^{k}(\calG,\mathfrak{l})$ with $\omega\in C^{k}(\calG,E)$ then $S(\omega,0)=[\omega]$.

In that way, the long exact sequence induced by the sequence \eqref{first} is exactly the long sequence to be proved up to an isomorphism
$$\theta:H^{\bullet-1}(\calG,\mathfrak{l})\longrightarrow H^{\bullet}(\mathcal{C})$$
induced by the sequence \eqref{second}. Hence to complete the proof it suffices to show that $R=\theta\circ\pi$ in cohomology. Indeed, take $\gamma$ a cocycle in $C^{k}_{VB}(\calG,\Gamma)$, thus $\pi(\gamma)=[\tilde{s}(\gamma)]$. Therefore, $\theta\circ(\bar{\gamma})=\theta(\overline{[\tilde{s}(\gamma)]})=\overline{\delta(\tilde{s}(\gamma),0)}$; that is, $(\delta'(\tilde{s}(\gamma)),\tilde{s}(\gamma))$ represents de cohomology class $\theta(\overline{[\tilde{s}(\gamma)]})\in H^{k}(\mathcal{C}^{\bullet})$. On the other hand, since $R(\gamma)=(\tilde{t}(\gamma),\tilde{s}(\gamma))$ is a cocycle, then $(0,0)=(\delta'(\tilde{t}(\gamma)),\tilde{t}(\gamma)-\delta'(\tilde{s}(\gamma)))$, which implies $\delta'(\tilde{s}(\gamma))=\tilde{t}(\gamma)$. Hence, $(\tilde{t}(\gamma),\tilde{s}(\gamma))=R(\gamma)$ represents the cohomology class $\theta(\overline{[\tilde{s}(\gamma)]})$, which completes the proof.
\end{proof}

\begin{remark} The a priori arbitrary complex $\mathcal{A}$ defined in the previous proof is a key element of the argument. Nevertheless one can give a more geometric meaning of it in terms of VB-groupoids as follows. Consider the anchor morphism
\[\xymatrix{
  \V \ar[d]  \ar[rr]^{an_{\V}} &  & 
 Pair(E)
\ar[d]\\
 \G \ar[rr]^{an_{\G}} & &  Pair(M)}\]
 between VB-groupoids. Then, the complex $\mathcal{A}$ is just the complex $C^{\bullet}(\G, an^{*}_{\G}(Pair(E)))$ associated to the pullback VB-groupoid $an^{*}_{\G}(Pair(E))$ over $\G$. And thus the acyclicity of $\mathcal{A}$ follows directly from the vanishing of the VB-cohomology of the proper VB-groupoid $\mathrm{Pair}(E)$ and from the acyclicity of its associated 2-term complex $\partial=Id:E\to E$. 
\end{remark}

We can thus use the previous theorem and obtain a long exact sequence for the deformation cohomology.

\begin{example}\label{Regular setting, morphisms}
Let $\Phi:\calH\rightarrow\calG$ be a morphism of Lie groupoids and assume that $\calG$ is regular. Then there exists a map $K:H^{*}(\calH,\phi^{*}\nu_\calG)\rightarrow H^{*+2}(\calH,\phi^{*}\mathfrak{i}_\calG)$ such that $H^{*}_{def}(\Phi)$ fits into the long exact sequence:
\begin{equation}\label{RegularVBmorphisms}
\cdots\longrightarrow H^{k}(\calH, \phi^{*}\mathfrak{i}_\calG)\stackrel{r}{\longrightarrow} H^{k}_{def}(\Phi)\stackrel{\pi}{\longrightarrow} H^{k-1}(\calH, \phi^{*}\nu_{\calG})\stackrel{K}{\longrightarrow} H^{k+1}(\calH, \phi^{*}\mathfrak{i}_\calG)\longrightarrow\cdots,
\end{equation}
where $\nu_\calG$ is the normal bundle to the orbits of $\calG$.
\end{example}

\section{Triviality of deformations of morphisms}\label{Section:Triviality}
%\section{Moser type argument}\label{Morphisms}

%In this section we describe the results directly related to the rigidity questions for morphisms and Lie subgroupoids.\\
We discuss here characterizations of the several types of triviality of deformations of morphisms in terms of the deformation cohomology. The main results in this section can be regarded as a Moser type theorem in the context of morphisms of Lie groupoids.

\begin{proposition}\label{cociclomorph.}
Let $\Phi_{\e}$ be a deformation of $(\Phi_{0},\phi_0):(\calH\tto N)\longrightarrow(\calG\tto M)$. Then, for each $\lambda$ we obtain a 1-cocycle
\[X_\lambda(h) = \left.\frac{d}{d\e}\right|_{\e = \lambda}\Phi_\e(h)\]
in the deformation complex of $\Phi_\l$. Moreover, the corresponding cohomology class at time $\e=0$ in $H^{1}_{def}(\Phi_0)$ depends only on the equivalence class of the deformation.
\end{proposition}

\begin{proof}
%CORRECT THIS PROOF ACCORDING PAG 12 IN IPAD EN BLANCO DOCUMENT!!*****\\ DONE!

The first part follows from taking derivative at $\e=\l$ of the morphism condition $\Phi_\e(\bar{m}_{\calH}(gh,h))=\bar{m}_{\calG}(\Phi_\e(gh),\Phi_\e(h))$ satisfied by every $\Phi_\e$. In fact, we get
\begin{equation*}
X_{\l}(g)-d\bar{m}_{\calG}(X_{\l}(m_{\calH}(g,h)),X_{\l}(h))=0,
\end{equation*} which says that $X_\l$ is a 1-cocycle in $C^{\bullet}_{def}(\Phi_\l)$.

Now suppose that $\Psi_\e$ is a deformation of $\Phi_0$ which is equivalent to $\Phi_\e$. Then $\Psi_\e=\tau_\e\cdot\Phi_\e$ for a smooth family $\tau_\e$ of gauge maps over $\phi_0$ with $\tau_0=1_\calG\circ\phi_0$. Denote by $X'_0$ the associated 1-cocycle at time zero. Heuristically, the exactness of $X'_0-X_0$ comes from taking derivatives at $\e=0$ of the equivalence condition $\Psi_\e=\tau_\e\cdot\Phi_\e$. However, notice that since $\tau_\e$ acts merely on the elements of $Im(\Phi_\e)$ then we can not use directly the chain rule to differentiate the expression with respect to $\e$. Therefore, consider the maps $\tau:M\times I\to\calG$, $\tau(x,\e):=\tau_\e(x)$, and $\Phi:\calH\times I\to\calG$, $\Phi(h,\e)=\Phi_\e(h)$, which codify the families of gauge maps and morphisms $\Phi_\e$. Thus the map $\tau\cdot\Phi(\cdot,\e)=(\tau_\e\cdot\Phi_\e(\cdot),\e)$ contains all the information of the expression we want to differentiate. In order to differentiate this map, we will write it now in an equivalent way. Indeed, let $\bar{\calH}\tto\bar{N}$ be the Lie groupoid which is the cartesian product of $\calH$ with $I$, $\calH\times I\tto N\times I$, and $\theta_1, \theta_2:\bar{\calH}\to\calG$ denote the maps $$\theta_1(p):=m_\G(\tau(t(p)),\Phi(p))\ \ \text{ and }\ \  \theta_2(p):=\tau(s(p))^{-1},$$ for $p\in\bar{\calH}$. Thus, we get that
$$\tau\cdot\Phi=m_\G\circ(\theta_1\times\theta_2).$$ Therefore, we get the derivative with respect to $\e$ by applying the differential of this map to the vector field $\partial/\partial\e\in\mathfrak{X}(\bar{\calH})$. That is,
$$d(\tau\cdot\Phi)\partial/\partial\e=dm\left(dm(d\tau(dt(\partial/\partial\e)), d\Phi(dF(\partial/\partial\e))),di(d\tau(ds(\partial/\partial\e)))\right),$$
which, after a straightforward computation (see Theorem 1.4.14 in \cite{M}), can be written as
\begin{equation}\label{eq:infinitesimally equiv.deform}
\begin{split}
%d(\tau\cdot\Phi)\partial/\partial\e(p)
\frac{d}{d\e}\Psi_\e(p)=&\; l_{\sigma_1\star\sigma_2}\left(di\left(\frac{d}{d\e}\tau_\e(s(p))\right)\right)+ r_{\sigma_3}l_{\sigma_1}\left(\frac{d}{d\e}\Phi_\e(p)\right)- l_{\sigma_1\star\sigma_2}r_{\sigma_3}(ds(\partial_\e(\Phi_\e(p))))+\\
&+r_{\sigma_2\star\sigma_3}\left(\frac{d}{d\e}\tau_\e(t(p))\right)-l_{\sigma_1}r_{\sigma_2\star\sigma_3}(dt(\partial_\e(\Phi_\e(p)))),
\end{split}
\end{equation}
where $p\in\bar{\calH}$ and $\sigma_1$, $\sigma_2$ and $\sigma_3$ are local bisections of $\G$ through $\tau(t(p))$, $\Phi(p)$ and $\tau(s(p))^{-1}$.

Therefore, for $p=(h,0)\in\calH\times\{0\}$ we can choose $\sigma_1=u_\G=\sigma_3$, and the previous equation becomes
\begin{equation}
\begin{split}
X'_0(h)-X_0(h)&=r_{\Phi_0(h)}\left(\left.\frac{d}{d\e}\right|_{\e=0}\tau_\e(t(h))-dt(\partial_\e\Phi_\e(h))\right)+\\
&\ \ \ \ +l_{\Phi_0(h)}\left(di\left(\left.\frac{d}{d\e}\right|_{\e=0}\tau_\e(s(h)\right)-ds(\partial_\e\Phi_\e(h)))\right)\\
&=\delta_{\Phi_0}(\bar{\alpha}_0)(h),
\end{split}
\end{equation}
where $\bar{\alpha}_0\in\Gamma(\phi_0^{*}A_\G)$ is given by
$$\bar{\alpha}_0(x)=\left.\frac{d}{d\e}\right|_{\e=0}\tau_\e(x)-\left.\frac{d}{d\e}\right|_{\e=0}\phi_\e(x),\ \ \text{for } x\in N.$$
That is, $X'_0$ and $X_0$ are in the same cohomology class, as we wanted to prove.
\end{proof}

The 1-cocycle $X_0$ is also called the \textbf{infinitesimal deformation} associated to the deformation $\Phi_\e$, and $X_\e$ will be called the \textbf{deformation cocycle}.

\begin{remark}\label{exactcocyclesmorph.}
Notice that if $(\Psi_\e,\psi_\e)$ is a trivial deformation of $ \Phi_0$, then equation \eqref{eq:infinitesimally equiv.deform} shows that every 1-cocycle $X'_\e$ is exact. Alternatively, by using the rule chain, this can be checked by a direct computation as below.

\begin{align*}
\left.\frac{d}{d\e}\right|_{\e=\l}(\tau_\e\cdot\Phi_0)(h)&=\left.\frac{d}{d\e}\right|_{\e=\l}m\left(\tau_{\e}(t(h))\!\cdot\!\Phi_0(h),\tau_{\l}(s(h))^{-1}\right)\\
&+\left.\frac{d}{d\e}\right|_{\e=\l}m\left(\tau_{\l}(t(h))\!\cdot\!\Phi_0(h),\tau_{\e}(s(h))^{-1}\right)\\
&=\left.\frac{d}{d\e}\right|_{\e=\l}R_{\Phi_0(h)\tau_{\l}(s(h))^{-1}}(\tau_{\e}(t(h)))\\
&+\left.\frac{d}{d\e}\right|_{\e=\l}L_{\tau_\l(t(h))\Phi_0(h)}(\tau_{\e}(s(h))^{-1})\\
&=r_{\tau_{\l}\cdot\Phi_0(h)}\left(\left.\frac{d}{d\e}\right|_{\e=\l}(\tau_\e(t(h))\cdot\tau_\l(t(h))^{-1})\right)\\
&+ l_{\tau_{\l}\cdot\Phi_0(h)}\left[di\left(\left.\frac{d}{d\e}\right|_{\e=\l}(\tau_\e(s(h))\cdot\tau_\l(s(h))^{-1})\right)\right].
\end{align*}

Define the family of sections $\bar{\alpha}_\e\in\Gamma(\psi_\e^{*}A_{\calG})$ by
\begin{equation}\label{sections-gauge}
\bar{\alpha}_{\l}(x):=\left.\frac{d}{d\e}\right|_{\e=\l}\left(\tau_{\e}(x)\tau_{\l}(x)^{-1}\right),\ \ \text{for each } \l.
\end{equation}

Therefore we get the exactness of the cocycles,
$$\frac{d}{d\e}\Psi_\e(h)=\delta_{\tau_{\e}\cdot\Phi_0}(\bar{\alpha}_{\e})(h).$$
%So, by taking $\l=0$ the claim is proved.
\end{remark}

\begin{remark}\label{rmk:sections-bisections}
Since two deformations which are strongly equivalent are, in particular, equivalent then they determine the same cohomology class in $H^{1}_{def}(\Phi_0)$. Moreover, in a totally analogous way, one can prove that they determine the same cohomology class in the variation cohomology $\tilde{H}^{1}_{def}(\Phi_0)$. In fact, it suffices considering a smooth family $\sigma_\e$ of bisections of $\G$ instead of the family $\tau_\e$ of gauge maps of the previous proof and make $\tau_\e=\sigma_\e\circ\phi_\e$. Furthermore, for a smooth family $\sigma_\e$ of bisections one can define the family $\alpha_\e\in\Gamma(A_\G)$ of sections of $A_\G$
$$\alpha_\l(\varphi_\l(x))=\left.\frac{d}{d\e}\right|_{\e=\l}\left(\sigma_{\e}(x)\sigma_{\l}(x)^{-1}\right),$$
for each $\l$, where $\varphi_\l=t\circ\sigma_\l$, and obtain, for a strongly trivial deformation, that
$$\left.\frac{d}{d\e}\right|_{\e=\l}(I_{\sigma_\e}\circ\Phi_0)(h)=\delta_{I_{\sigma_{\l}}\circ\Phi_0}((\varphi_\l\circ\phi_{0})^{*}\alpha_{\l})(h).$$
We remark this result in the following proposition.
\end{remark}

\begin{proposition}\label{prop:exactcocyles variatedcomplex} Let $\Phi_\e$ be a deformation of $\Phi_0:\H\to\G$. Then, the corresponding cohomology class $[X_0]$ at time $\e=0$ in $\tilde{H}^{1}_{def}(\Phi_0)$ depends only on the strong equivalence class of the deformation. Moreover, a strongly trivial deformation has exact cocycles $X_\e$ in the variation complex $\tilde{C}_{def}(\Phi_e)$.
\end{proposition}

In order to prove the triviality of a deformation $\Phi_\e$ of $\Phi_0$ by means of a Moser type argument, we will need not only that the cohomology class of the deformation cocycle vanishes, but that it vanishes in a smooth manner as we now explain.

Recall that a deformation $\Phi_\e$ of $\Phi_0$ is a smooth map
\[\xymatrix{\calH\times I \ar[r]^\Phi \ar@<0.25pc>[d] \ar@<-0.25pc>[d]& \calG\times I \ar@<0.25pc>[d] \ar@<-0.25pc>[d]\\
N\times I \ar[r]_{\phi} & M\times I}\]

\begin{definition}\label{def:smoothfamilycochains}
A family of cocycles $X_\e\in C^{1}_{def}(\Phi_\e)$ is \textbf{smoothly exact} if there exists a smooth section $\bar{\alpha}\in\Gamma_{N\times I}(\phi^{*}A_{\G\times I})$ such that for each $\e\in I$, $\bar{\alpha}_\e=\bar{\alpha}(\cdot,\e)\in\Gamma_N(\phi^{*}_\e A_\G)=C^{0}_{def}(\Phi_\e)$ and
$$\delta_{\Phi_\e}(\bar{\alpha}_\e)=X_\e.$$
A family $\bar{\alpha}_\e\in\Gamma_N(\phi^{*}_\e A_\G)$ will be \textbf{smooth} if it can be encoded in a smooth section $\bar{\alpha}\in\Gamma_{N\times I}(\phi^{*}A_{\G\times I})$ as above.

\end{definition}

Equivalently, defining the morphism $(\tilde{\Phi},\tilde{\phi}):\calH\times I\to\calG$, which is the projection to $\calG$ of $\Phi$, the family $\bar{\alpha}_{\e}$ is smooth if the section $\tilde{\alpha}\in\Gamma(\tilde{\phi}^{*}A_\calG)$ given by $\tilde{\alpha}(x,\e):=\bar{\alpha}_{\e}(x)$ is smooth. %Observe that since $\phi^{*}A_{\calG\times I}=\tilde{\phi}^{*}A_\calG$, the vector bundle sections $\bar{\alpha}$ and $\tilde{\alpha}$ are the same.

\begin{theorem}\label{gauge-triviality}
Let $(\Phi_\e,\phi_\e)$ be a deformation of the morphism $(\Phi_0,\phi_0):(\calH\rightrightarrows N)\longrightarrow (\calG\rightrightarrows M)$. %Assume that $N$ is compact and $\phi_0$ is transversal to the orbits of $\calG$.
Then, $\Phi_\e$ is trivial if and only if the family $X_\e$ of cocycles is smoothly exact in $C^{\bullet}_{def}(\Phi_0)$.
\end{theorem}

\begin{remark}
The smooth exactness condition of the family $X_\e$ is just another way to say that each 1-cocycle $X_\e=\frac{d}{d\e}\Phi_{\e} \in C_{def}^{1}(\Phi_\e)$ is equal to $\delta_{\Phi_\e}(\bar{\alpha}^{\e})$, where $\bar{\alpha}^{\e}$ is a \textit{smooth} family of 0-cochains in the sense of Definition \ref{def:smoothfamilycochains}.
\end{remark}

\begin{proof}
The smooth exactness of the cocycles $X_\e$ was already verified in Remark \ref{exactcocyclesmorph.}. We prove now the converse statement, where the goal is finding a smooth family of gauge-maps which verifies the triviality of the deformation $\Phi_\e$. Assume that, for every $\e\in I$, $X_\e=\delta_{\Phi_\e}(\bar{\alpha}_\e)$, for $\bar{\alpha}_\e\in\Gamma(\phi_\e^{*}A_\calG)$ such that $\tilde{\alpha}(x,\e):=\bar{\alpha}_\e(x)$ is a smooth section in $\Gamma_{N\times I}(\tilde{\phi}^{*}A_\calG)$ (see Definition \ref{def:smoothfamilycochains} for the notations). We will define the family of gauge maps in terms of the flow of an appropriate vector field determined by the sections $\bar{\alpha}_\e$.

Consider the vector field
\begin{equation}\label{eq: vectorfield V}
V(g,y,\e):=(r_g(\tilde{\alpha}(y,\e)),0_y,\partial/\partial\e),
\end{equation}
defined on the fibered-product $\tilde{\phi}^{*}\calG:=\calG\;_t\!\times_{\tilde{\phi}}(N\times I)$ of the target map $t$ with the base map $\tilde{\phi}$, as in the diagram below,

%$$DIAGRAM-FIBERED-PRODUCT$$

\begin{equation}\label{fiberedproduct}
%\[
  \begin{tikzcd}[column sep=4em, row sep=10ex]
    \calG\;_t\!\times_{\tilde{\phi}}(N\times I)\arrow{r}{\mathrm{pr}_\calG} \arrow{d}{} & \calG \arrow{d}{t}\\
    N\times I \arrow{r}{\tilde{\phi}}   & M.
  \end{tikzcd}
%\]
\end{equation}

This vector field is indeed well-defined since $\frac{d}{d\e}\Phi_\e=\delta_{\Phi_\e}(\bar{\alpha}_\e)$, and therefore $\frac{d}{d\e}\phi_\e=\rho(\bar{\alpha}_\e)$.

 %The compactness of $N$ guarantees that the flow $\psi_\e$ is defined for all $\e$ small enough over all $\phi_0(N)\times N\times\{0\}$, thus we define the family of gauge-maps
Lemma \ref{lemma:completeflow} below guarantees that the flow $\psi_\e$ of $V$ is defined for all $\e\in I$ over the points of $M\;_t\!\times_{\tilde{\phi}}N\times\{0\}\subset\tilde{\phi}^{*}{\calG}.$
%$\phi_0(N)\;_t\!\times_{\tilde{\phi}}N\times\{0\}\subset\tilde{\phi}^{*}{\calG}$.
Thus $\psi_\e$ defines the family of gauge maps (which projects to $\phi_0$)
$$\tau_\e:N\longrightarrow\calG;\ \ x\longmapsto\mathrm{pr}_\calG(\psi_\e(\phi_0(x),x,0)),$$
where $\mathrm{pr}_\calG:\tilde{\phi}^{*}\calG\longrightarrow\calG$ is the natural projection to $\calG$ as in diagram \eqref{fiberedproduct} above. That is, the flow of $V$ over $\phi_0(N)\times N\times\{0\}$ can be written as $\psi_\e(\phi_0(x),x,0)=(\tau_\e(x),x,\e)$.\\

\textbf{Claim:} The family $\tau_\e$ proves the triviality of $\Phi_\e$, i.e., it holds $\Phi_\e=\tau_\e\cdot\Phi_0$.

We will prove this claim by showing that $\Phi_\e$ and $\tau_\e\cdot\Phi_0$ determine the same integral curves of a vector field $Z$ defined below. For that, notice first that since $\frac{d}{d\e}\phi_\e=\rho(\bar{\alpha}_\e)$, then the curve $\phi_\e(x)$ belongs to a unique orbit of $\calG$. Thus, for $h\in\calH_\calO$, the curve $\e\to\Phi_\e(h)$ lies inside the restriction groupoid $\calG_{\calO_{\phi_0(s(h))}}$. That is, for any $h\in\calH_\calO$, both curves $\e\mapsto\Phi_\e(h)$ and $\e\mapsto I_{\tau_\e}\circ\Phi_0(h)$ are inside the restriction $\calG_{\calO'}$; where $\calO':=\calO_{\phi_0(s(h))}$. Therefore, we can consider the following fibered product

%$$DIAGRAM 2$$
\begin{equation}\label{fiberedproduct2}
%\[
  \begin{tikzcd}[column sep=4em, row sep=10ex]
    B \arrow{r}{} \arrow{d}{} & \calG\ _t\!\times_{\tilde{\phi}}(\calO\times I) \arrow{d}{s_\calG\times pr_I}\\
    \calO\times I \arrow{r}{\phi}   & \calO'\times I,
  \end{tikzcd}
%\]
\end{equation}
which, by transversality of the maps involved in the diagram, will be well-defined. Let $Z$ denote the vector field on $B$ defined by
$$Z(g,y,x,\e):=(r_g(\tilde{\alpha}(y,\e))+l_g(di(\tilde{\alpha}(x,\e))),0_y,0_x,\partial/\partial\e).$$
Then, for every $h\in\calH_\calO$, $\frac{d}{d\e}\Phi_\e(h)=\delta_{\Phi_\e}(\tilde{\alpha}(-,\e))(h)$ and $\frac{d}{d\e}\tau_\e\cdot\Phi_0(h)=\delta_{\tau_\e\cdot\Phi_0}(\tilde{\alpha}(-,\e))(h)$. It follows that $(\Phi_\e(h), t(h), s(h), \e)$ and $(\tau_\e\cdot\Phi_0(h), t(h), s(h), \e)$ are integral curves of $Z$ starting at the same point. Therefore, $\Phi_\e(h)=\tau_\e\cdot\Phi_0(h)$, for every $h\in\calH_\calO$. Since $h\in\calH$ is arbitrary, it follows that $\Phi_\e=\tau_\e\cdot\Phi_0$.

\end{proof}

\begin{lemma}\label{lemma:completeflow}
Let $V$ be the vector field defined by equation \eqref{eq: vectorfield V} above. The flow $\psi^{\e}$ of $V$ is defined for all $\e\in I$ over $M\;_t\!\times_{\tilde{\phi}}N\times\{0\}\subset\tilde{\phi}^{*}{\calG}.$
%$\phi_0(N)\times N\times \{0\}$.
\end{lemma}

\begin{proof}
Note that the vector field $V$ projects by the target map to a vector field
$\bar{a}\in\mathfrak{X}(M\;_t\times_{\tilde{\phi}}(N\times I))$ $\cong\mathfrak{X}(N\times I),$ given by
$$\bar{a}(y,\e)=(\rho(\bar{\alpha}(y,\e)),0_y,\partial/\partial\e)\in TM\;_{Id_{TM}}\!\times_{d\tilde{\phi}}T(N\times I).$$
It follows from the fact that $\frac{d}{d\e}\phi_\e=\rho(\bar{\alpha}_\e)$, that the integral curves of $\bar{a}$ are determined by the smooth family of base maps $\phi_\e$. %whose integral curves are determined by the family of base maps $\phi_\e$, due to the fact that $\frac{d}{d\e}\phi_\e=\rho(\bar{\alpha}_\e)$.
Therefore they are defined for all $\e\in I$ when starting at points of $M\;_t\!\times_{\tilde{\phi}}N\times\{0\}\subset\tilde{\phi}^{*}{\calG}$. Thus, the proof now follows by an argument completely analogous to that of Theorem 3.6.4 in \cite{M}, which allows us to check that the flow $\psi^\e$ over $M\;_t\!\times_{\tilde{\phi}}N\times\{0\}\subset\tilde{\phi}^{*}{\calG}$ is defined for the same time as the flow of $\bar{a}$; i.e., for all $\e\in I$.\\

\end{proof}

\begin{remark}
It is straightforward to see that the previous result generalizes that of \cite{CardS1} concerning the triviality of deformations of Lie group homomorphisms.
\end{remark}

Notice that one can use Theorem \ref{gauge-triviality} and Proposition \ref{Prop:familybisectionsfromgauge} to obtain as a direct consequence the following kind of characterization of strongly trivial deformations.

\begin{theorem}\label{thm:homostronglytrivial}%\label{homo.trivial}%\label{Rigiditymorph.}

Let $(\Phi_0,\phi_0):(\calH\rightrightarrows N)\longrightarrow (\calG\rightrightarrows M)$ be a Lie groupoid morphism and $(\Phi_\e,\phi_\e)$ be a deformation of $\Phi_0$. Assume that $\phi_0$ is an injective immersion and $N$ is compact. Then, the deformation $\Phi_\e$ is strongly trivial if and only if the family of 1-cocycles $X_\e=\frac{d}{d\e}\Phi_{\e}$ is smoothly exact in $C^{\bullet}_{def}(\Phi_0)$.
\end{theorem}

The following theorem shows that with the help of the subcomplex $\tilde{C}^{\bullet}_{def}(\Phi_0)$ is posible to obtain a cleaner characterization of the strongly trivial deformations.

\begin{theorem}\label{thm:strongtriviality}
Let $(\Phi_0,\phi_0):(\calH\rightrightarrows N)\longrightarrow (\calG\rightrightarrows M)$ be a Lie groupoid morphism and $(\Phi_\e,\phi_\e)$ be a deformation of $\Phi_0$. Assume that $N$ is compact.
Then, the deformation $\Phi_\e$ is strongly trivial if and only if the family of 1-cocycles $X_\e=\frac{d}{d\e}\Phi_{\e}$ is smoothly exact in the subcomplex $\tilde{C}^{\bullet}_{def}(\Phi)$.
\end{theorem}

\begin{proof}
If $\Phi_\e$ is strongly trivial, the smooth exactness of the family of 1-cocycles $X_\e$ was already proved in Remark \ref{rmk:sections-bisections}. Conversely, let $X_\e=\delta_{\Phi_\e}(\phi_\e^{*}\alpha'_\e)$, where $\alpha'_\e$ is a family of sections of $A_\G$, such that $\phi_\e^{*}\alpha'_\e$ is a smooth family of sections in the sense of Definition \ref{def:smoothfamilycochains}.

Since $N$ is compact, $\phi(N\times I)$ is a closed subset inside $M\times I$. Shrinking the interval $I$ if necessary, let $V\subset U\subset M$ be open subsets with compact closure such that $\phi(N\times I)\subset V\times I$. Then we can extend the smooth family $\alpha'_\e|_U$ of restriction sections to a smooth family $\alpha_\e$ of sections supported on $U$ such that $\alpha_\e|_V=\alpha'_\e|_V$.

Then, on the one hand, since $\alpha_\e=\alpha'_\e$ over $V$ it follows that every cocycle $X_\e$ is the pullback by $\Phi_\e$ of the vector field $\delta_\e(\alpha_\e)\in\mathfrak{X}(\G)$. Indeed, 
\begin{equation}\label{eq:1}
\begin{split}
X_{\e}(h)&=\delta_{\calG}(\alpha'_{\e})(\Phi_{\e}(h))\\
&=r_{\Phi_\e(h)}\left(\alpha'_{\e}(t\circ\Phi_{\e}(h))\right)+l_{\Phi_\e(h)}\left(di(\alpha'_{\e}(s\circ\Phi_{\e}(h)))\right)\\
&=\delta_{\calG}(\alpha_{\e})(\Phi_{\e}(h)).
\end{split}
\end{equation}

On the other hand, let $\sigma_\e$ be the smooth family of bisections of $\G$ induced, by the exponential flow, by the family of sections $\alpha_\e$ as in the proof of Proposition \ref{Prop:familybisectionsfromgauge}. Note that such a family $\sigma_\e$ is defined for all $\e$ small enough. Thus, by Remark \ref{rmk:sections-bisections},
\begin{equation}\label{eq:2}
\frac{d}{d\l}|_{\l=\e} I_{\sigma_\l}\circ\Phi_0(h)=\delta_\G(\alpha_\e)(I_{\sigma_\e}\circ\Phi_0)(h).
\end{equation}

Therefore, in other words, by equations \eqref{eq:1} and \eqref{eq:2} one has that $\e\mapsto\Phi_\e(h)$ and $\e\mapsto I_{\sigma_\e}\circ\Phi_0(h)$ are integral curves of the time-dependent vector field $\delta_\G(\alpha_\e)$ passing through $\Phi_0(h)\in\G$ at time $\e=0$. That is, $\Phi_\e(h)=I_{\sigma_\e}\circ\Phi_0(h)$ for all $\e$ small enough.

%The family is determined for all $\e\in I$, because the target projection has the integral curves defined for $\e\in I$.
\end{proof}

%%%%%%%%%%%%%%%%%%%%%%%%%%%%%%%%%%%%%%%%%%%%%%%%%%%%

The triviality of the following special type of deformations is not hard to prove directly without using the cohomological tools of this section, nevertheless we will verify it as a consequence of the Theorem \ref{gauge-triviality} above.

\begin{example}
Let $\pi:M\longrightarrow M'$ be a surjective submersion. Assume that $(F,f)$ is a Lie groupoid morphism between $\calH$ and the submersion groupoid $\calG=M\times_{M'}M$ and that $(F_\e,f_\e)$ is a deformation of $F$. Then, $F_\e$ is trivial if and only if $\pi\circ f_\e=\pi\circ f_0$.

In fact, on the one hand, notice first that, since the Lie algebroid $A_{M\times_{M'}M}$ of the submersion groupoid consists of the vertical vectors $T^{\pi}M$ of $TM$, then the family of elements $\bar{\alpha}_\e$ given by $\bar{\alpha}_\e(n):=(\frac{d}{d\e}f_\e(n),0_{f_\e(n)})$ is a family of 0-cochains in $C^{0}_{def}(F_\e)$ if and only if $\pi\circ f_\e=\pi\circ f_0$ for all $\e$.

On the other hand, it is straightforward to check that every morphism $(F_\e,f_\e)$ is of the form $F_\e=(f_\e\circ t_\calH,f_\e\circ s_\calH)$. Therefore, if for every $h\in\calH$ and all $\e$, $f_\e(h)$ and $f_0(h)$ are in the same $\pi$-fiber, then the family of 1-cocycles $X_\e=\frac{d}{d\e}F_\e$ is smoothly transgressed by the family of 0-cochains $\bar{\alpha}_\e$. And conversely, if $X_\e=\frac{d}{d\e}F_\e$ is smoothly transgressed, then $\frac{d}{d\e}f_\e(n)$ must be a vector tangent to the $\pi$-fibers, for all $n\in N$.

\end{example}

The previous example tells us that the family $F_\e$ is trivial if and only if it preserves the correspondence, determined by $F_0$, between the leaves $\H_\calO$ of the foliation by orbit-groupoids of $\H$ and the leaves $\G_\calO$ of $\G$. That is, if for every $h\in\H$, the curve $\e\mapsto F_\e(h)$ lies inside the restriction groupoid $\G_{\calO_{s(F_0(h))}}$. The following proposition explores this idea for any Lie groupoid $\G$ (not only the submersion groupoid).

\begin{proposition}\label{prop:exactcocyles}
Let $\Phi_\e:\H\to\G$ be a deformation of $\Phi_0$ and assume that $\H$ is proper. The deformation $\Phi_\e$ is trivial if, and only if, the curves $\e\mapsto\Phi_\e(h)$, $h\in\H$, determined by $\Phi_\e$ lie inside the leaves of the foliation $\{\G_{\calO_x}\}_{x\in M}$ by orbit-groupoids of $\G$.
\end{proposition}

\begin{proof}
If $\Phi_\e$ is a trivial deformation, then obviously the curve $\e\mapsto\Phi_\e(h)$ lies inside a unique leaf because $\Phi_\e(h)=\tau_\e(t(h))\Phi_0(h)\tau_\e(s(h))^{-1}$, for a smooth family $\tau_\e$ of gauge maps covering $\phi_0$. Conversely, let $\tilde{\Phi}:\H\times I\to\G$ be the morphism such that restricted to $\H\times\{\e\}$ is $\Phi_\e$, as in Definition \ref{def:smoothfamilycochains}. Then, the family of 1-cocycles $X_\e=\frac{d}{d\e}\Phi_\e\in C^{1}_{def}(\Phi_\e)$ can be encoded in a unique 1-cocycle $\tilde{X}\in C^{1}_{def}(\tilde{\Phi})$ given by
\begin{align*}
\tilde{X}(h,\e)&:=\tilde{\Phi}_{*}(\partial/\partial\e|_{(h,\e)})\\
&=X_\e(h).
\end{align*}
And, $\tilde{X}$ turns out to be exact if and only if the family $X_\e$ is smoothly exact.

Now, since $\H$ is a proper groupoid, it follows that $H^{1}(\H,\tilde{\phi}^{*}\i_\G)$ is trivial and the sequence \eqref{lowdegreemorphism} above for the morphism $\tilde{\Phi}$ becomes

\begin{equation}
0\rightarrow 0\stackrel{r}{\longrightarrow} H^{1}_{def}(\tilde{\Phi})\stackrel{\tilde{s}}{\longrightarrow}\Gamma(\tilde{\phi}^{*}\nu_\calG)^{inv}\stackrel{K}{\longrightarrow}\cdots.
\end{equation}
Moreover, due to the fact that $\tilde{X}$ is always tangent to the orbit groupoids of $\G$ then $\tilde{X}$ lies in the kernel of $\tilde{s}$. Therefore $\tilde{X}$ is an exact cocycle which, by Theorem \ref{gauge-triviality}, says that $\Phi_\e$ is a trivial deformation.
\end{proof}

We can also apply our methods to study deformations which are \emph{trivial up to automorphisms of $\G$}, that is, those which consider also the group of outer automorphisms of a Lie groupoid. Additionally, in Subsection \ref{Moreontriviality} below, we sketch other types of equivalences between deformations of morphisms which arise naturally.

\begin{definition}\label{Def:weaklytrivial}
A deformation $\Phi_\e$ of a Lie groupoid morphism $\Phi_0: \calH \to \calG$ is said to be \textbf{trivial up to automorphisms of $\G$} if there exist an open interval $I$ containing 0 and smooth families $F_\e:\calG\longrightarrow\calG$ of Lie groupoid automorphisms and $\tau_\e:N\to\G$ of gauge maps over $\phi_0$ such that $F_0 = Id_\calG$, $\tau_0=\phi_0$ and $\Phi_\e = F_\e\circ(\tau_\e\cdot\Phi_0)$ for all $\e\in I$. Analogously we say that $\Phi_\e$ is \textbf{strongly trivial up to automorphisms of} $\calG$ if there exists a smooth family $\sigma_\e:M\longrightarrow \calG$ of bisections of $\calG$ such that $\sigma_0=u_\G$ and $\Phi_\e=F_\e\circ I_{\sigma_\e}\circ\Phi_0$ for all $\e\in I$.
\end{definition}

\begin{remark}
Note that considering only outer automorphisms produces \textbf{weakly trivial} deformations $\Phi_\e=F_\e\circ\Phi_0$ whose study is already include in the strongly trivial up to automorphisms deformations.
\end{remark}

Recall that a Lie groupoid morphism $\Phi: \calH \to \calG$ induces a pull-back map $\mathrm{\Phi^*:H^k_{def}(\calG) \to H^k_{def}(\Phi)}$. The key to characterizing the previous types of deformations lies in studying the pre-image of the deformation cocycle of a deformation $\Phi_\e$ through the pull-back map $\Phi_\e^{*}$, when it exists.

\begin{definition}
We will say that a family $[X_\e] \in H_{def}^1(\Phi_\e)$ has a \textbf{smooth pre-image in $H^1_{def}(\calG)$} if there exist smooth families of cocycles $Z_\e \in C^1_{def}(\calG)$ and of cochains $\bar{\alpha}^\e \in C^{0}_{def}(\Phi_\e)$ such that
\[\Phi_\e^*(Z_\e) = X_\e +\delta_{\Phi_\e}(\bar{\alpha}^\e).\] Analogously, considering the variated cohomology, we say that $[X_\e] \in \tilde{H}_{def}^1(\Phi_\e)$ has a \textbf{smooth pre-image in $H^{1}_{def}(\calG)$} if the sections $\bar{\alpha}_\e$ are of the form $\phi^{*}_{\e}\alpha_\e$, for $\alpha_\e\in\Gamma(A_\G)$.
\end{definition}

The statements of the following two theorems concerning the two types of deformations in definition \ref{Def:weaklytrivial} are analogous to the statements of the Theorems \ref{thm:homostronglytrivial} and \ref{gauge-triviality} concerning strongly trivial and trivial deformations.

\begin{theorem}\label{thm:preimage pullback t.u.t.a.r.}
Let $(\Phi_\e,\phi_\e)$ be a deformation of the morphism $(\Phi_0,\phi_0):(\calH\rightrightarrows N)\longrightarrow (\calG\rightrightarrows M)$. Assume that $\G$ is compact. Then, $\Phi_\e$ is trivial up to automorphisms of $\calG$ if and only if $[X_\e]$ has a smooth preimage by $\Phi_\e^{*}$ in $H^{\bullet}_{def}(\calG)$ for all $\e\in I$, where $I$ is some interval containing the zero.
\end{theorem}

\begin{proof}
Let $\Phi_\e$ be a gauge trivial deformation up to automorphisms of $\calG$, that is, $\Phi_\e=F_\e\circ(\tau_\e\cdot\Phi_0)$, for smooth families $F_\e$ of automorphisms of $\G$, with $F_0=Id_\G$, and $\tau_\e$ of gauge maps with base $\phi_0$ such that $\tau_0=\phi_0$. By applying $\frac{d}{d\e}$ to both sides of the equation we obtain

\begin{align*}
X_\e&=\Phi_\e^{*}Z_\e+(F_\e)_{*}\delta_{\tau_\e\cdot\Phi_0}(\bar{\alpha}^{\e})\\
&=\Phi_\e^{*}Z_\e+\delta_{\Phi_\e}(\tilde{\alpha}^{\e})
\end{align*}
where $\tilde{\alpha}^{\e}=(F_\e)_{*}(\bar{\alpha}^{\e})\in C^{0}_\mathrm{def}(\Phi_\e)$ and $Z_\e:=\frac{dF_\e}{d\e}\circ F^{-1}_\e$ is a smooth family of 1-cocycles in $C^{1}_\mathrm{def}(\calG)$. It follows that $[X_\e]=\Phi_\e^{*}[Z_\e]$ for all $\e$.

Conversely, assume that
\begin{equation}\label{preimage}
X_\e=\Phi_\e^{*}Z_\e+\delta_{\Phi_\e}(\bar{\alpha}^\e),
\end{equation}
for $Z_\e$ and $\bar{\alpha}^\e$ smooth families of elements in $Z^{1}_{def}(\calG)$ and $C^{0}_{def}(\Phi_\e)$, respectively. Let $F_\e=\Phi^{\e, 0}$ be the flow from time $0$ to $\e$ of the time dependent vector field $Z_\e$ on $\calG$. Recall that every $F_\e$ will be an automorphism of $\calG$. We claim that $\Phi'_\e=F_\e^{-1}\circ\Phi_\e$ is a trivial deformation of $\Phi_0$. That is, we will obtain a smooth family $\tau_\e$ of gauge maps with base $\phi_0$, starting at $\phi_0$, such that $F_\e^{-1}\circ\Phi_\e = I_{\tau_\e}\circ\Phi_0$ for $\e$ small enough, concluding the proof.

On the one hand, taking $\tilde{\alpha}^\e\in C^{0}_{def}(F_\e^{-1}\circ\Phi_\e)$ to be such that $\bar{\alpha}_\e=(F_\e)_{*}(\tilde{\alpha}^\e)$, equation \eqref{preimage} becomes
\begin{equation}\label{preimage2}
\begin{split}
X_\e&=\Phi_\e^{*}(Z_\e)+\delta_{\Phi_\e}((F_\e)_{*}(\tilde{\alpha}^\e))\\
&=\Phi_\e^{*}(Z_\e)+(F_\e)_{*}(\delta_{F_{\e}^{-1}\Phi_\e}(\tilde{\alpha}^\e)).
\end{split}
\end{equation}

On the other hand, we set $X'_\e \in C^{1}_{def}(\Phi'_\e)$ to be the family of deformation cocycles associated to the deformation $\Phi'_\e = F_\e^{-1}\Phi_\e$ of $\Phi_0$. We will check that $X'_\e=\delta_{\Phi'_\e}(\tilde{\alpha}^\e)$, i.e., $X'_\e$ is smoothly exact. In fact,
\begin{equation*}
\left.\frac{d}{d\e}\right|_{\e=\l}\Phi'_\e(h)=\left.\frac{d}{d\e}\right|_{\e=\l}F_\e^{-1}(\Phi_\l(h))+dF_\l^{-1}(\left.\frac{d}{d\e}\right|_{\e=\l}\Phi_\e(h)),
\end{equation*}
from where it follows that
\begin{align*}
X'_\l(h)&=-dF_\l^{-1}(Z_\l(\Phi_\l(h)))+dF_\l^{-1}(X_\l(h))\\
&=-dF_\l^{-1}(\Phi_\l^{*}Z_\l)(h)+dF_\l^{-1}(X_\l)(h)\\
&=\delta_{\Phi'_\l}(\tilde{\alpha}^\l)(h),
\end{align*}
where the last equality follows from equation \eqref{preimage2}. It then follows from Theorem \ref{gauge-triviality} that $\Phi'_\e$ is trivial concluding the proof of the theorem.
\end{proof}

Analogously, adding the condition of $\Phi_0$ to be an injective inmersion, the previous proof and Theorem \ref{thm:homostronglytrivial} prove the following.

\begin{theorem}\label{weak-triviality}
Let $(\Phi_\e,\phi_\e)$ be a deformation of the morphism $(\Phi_0,\phi_0):(\calH\rightrightarrows N)\longrightarrow (\calG\rightrightarrows M)$. Assume that $\phi_0$ is an injective immersion and that $N$ and $\calG$ are compact. Then, $\Phi_\e$ is strongly trivial up to automorphisms of $\G$ if and only if the family $[X_\e]$ of cohomology classes has a smooth preimage by $\Phi_\e^{*}$ in $H^{\bullet}_{def}(\calG)$.
\end{theorem}

\begin{remark}%\textcolor{red}{argumento para la necesidad de la prueba tal como est\'a de strong triviality... pues es necesario la extension de secciones para este teorema.\\
In particular, if the family $X_\e$ of the previous Theorem is smoothly exact, then Theorem \ref{thm:homostronglytrivial} shows that the family of automorphisms of $\G$ can be taken as a family of inner automorphisms.
\end{remark}

With a very similar proof to that of Theorem \ref{thm:preimage pullback t.u.t.a.r.}, we can consider the variated complex $\tilde{C}_{def}^{*}(\Phi_\e)$ and prove 

\begin{theorem}\label{thm:preimage pullback strong t.u.t.a.r., variated complex}
Let $(\Phi_\e,\phi_\e)$ be a deformation of the morphism $(\Phi_0,\phi_0):(\calH\rightrightarrows N)\longrightarrow (\calG\rightrightarrows M)$. Assume that $\G$ is compact. Then, $\Phi_\e$ is strongly trivial up to automorphisms of $\calG$ if and only if $[X_\e]\in\tilde{H}^{\bullet}_{def}(\Phi_\e)$ has a smooth preimage by $\Phi_\e^{*}$ in $H^{\bullet}_{def}(\calG)$ for all $\e\in I$, where $I$ is some interval containing the zero.
\end{theorem}

As a final result of this subsection we sketch an alternative characterization of these deformations under the weaker condition of smooth exactness on the cokernel complexes $\mathrm{Coker}(\Phi_\e^{*})$ of the pullback maps $\Phi_\e^{*}$. Given a deformation $\Phi_\e$, let $\bar{X}_\e\in\mathrm{Coker}(\Phi_\e^{*})$ denote the image of the cocycles $X_\e$ in the cokernel complex. The smooth exactness of the cocycles $\bar{X}_\e$ is defined following the philosophy that all the elements involved in the transgression of the family $\bar{X}_\e$ form smooth families.

\begin{theorem}\label{thm:trivial utar, surbmersion}
Let $(\Phi_\e,\phi_\e)$ be a deformation of the morphism $\Phi_0:\calH\longrightarrow\calG$. Assume that $\Phi_0$ is a surjective submersion and that $\calH$ and $\calG$ are compact and connected. Then, the deformation $\Phi_\e$ is trivial up to automorphisms of $\G$ if and only if the family of cocycles $\bar{X}_\e$ in $\mathrm{Coker}(\Phi_\e^{*})$ is smoothly exact.
\end{theorem}

\begin{proof}[Sketch of proof]
%*****-Show that $Z_\e$ "will must" be multiplicative under in the image of $\Phi_\e$..which is $\G$. So the rest of the proof follows from previous Theorem to prove tiviality u.t.a.r. 

The exactness condition amounts to have

\begin{equation}\label{eq:exactness in cokernel}
X_\e=\Phi_\e^{*}Z_\e+\delta_{\Phi_\e}(\bar{\alpha}^\e),
\end{equation}
for $Z_\e$ and $\bar{\alpha}^\e$ smooth families of elements in $C^{1}_{def}(\calG)$ and $C^{0}_{def}(\Phi_\e)$, respectively. Let $F_\e=\Phi^{\e, 0}$ denote the flow from time $0$ to $\e$ of the time dependent vector field $Z_\e$ on $\calG$. A priori, every $F_\e$ will be a diffeomorphism of $\calG$. However we \emph{claim} that every $F_\e$ is a morphism for all $\e$ small enough.

Indeed, in order to prove the claim note that by applying $\delta_{\Phi_\e}$ to equation \eqref{eq:exactness in cokernel} we get
$$\delta_{\Phi_\e}(\Phi_\e^{*}Z_\e)=0,$$
which evaluating in arrows $(g,h)\in\calH^{(2)}$ shows us that $Z_\e$ is multiplicative on the image of $\Phi_\e$. Thus, since the morphism $\Phi_0$ is assumed to be a surjective submersion then, by the compactness of $\calH$ and the connectedness of the groupoids, $\Phi_\e$ will be a surjective submersion for all $\e$ small enough. That is, every $Z_\e$ is multiplicative on $\G$. Therefore one can follows now the proof of Theorem \ref{thm:preimage pullback t.u.t.a.r.} to check that the deformation $\Phi_\e$ is trivial up to automorphisms of $\G$. The converse statement follows as in Theorem \ref{thm:preimage pullback t.u.t.a.r.}.

\end{proof}

\subsection{Additional remarks on triviality}\label{Moreontriviality}

The previous results take account of four notions of equivalences defined for deformations of morphisms. However, there are some other very natural types of deformations that can be considered. Here we briefly explain these notions and show how they fit in the framework of deformation complexes.

\begin{definition}\label{Def:strongly trivial up to automorp}
We say that a deformation $\Phi_\e$ of $\Phi_0:\calH\to\calG$ is \textbf{strongly trivial up to automorphisms of $\calH$} or \textbf{strongly trivial up to automorphisms on the left} if there exist an open interval $I$ around 0 and a smooth family $F_\e:\calH\to\calH$ of automorphisms of $\calH$ with $F_0=Id_\calH$ and a smooth family $\sigma_\e$ of bisections of $\calG$ with $\sigma_0=1_\calG$ such that $\Phi_\e\circ F_\e=I_{\sigma_\e}\circ\Phi_0$, for all $\e\in I$. Analogously a deformation is called \textbf{trivial up to automorphisms of $\calH$} if there exists a smooth family $F_\e:\calH\to\calH$ of automorphisms of $\calH$ with $F_0=Id_\calH$ and a smooth family $\tau_\e:N\to\G$ of gauge maps over $\phi_0$ such that $\tau_0=1_\G\circ\phi_0$ and $\Phi_\e\circ F_\e=I_{\tau_\e}\circ\Phi_0$, for all $\e\in I$.
\end{definition}

\begin{remark}
Deformations which are strongly trivial up to automorphisms on left are very much related to what should be called \emph{trivial deformations of Lie subgroupoids} where it is required to deform the groupoid on left as well, and the family $F_\e$ defines the type of deformations allowed. These deformations are studied in detail in \cite{CardS-Subgroupoids}.
\end{remark}

The arguments used in the proof of Theorem \ref{gauge-triviality} can also be applied to get results concerning the two types of deformations defined above. The following two results give conditions to characterize deformations which are strongly trivial up to automorphisms on left by using the usual deformation complex $C^{*}_{def}(\Phi_\e)$ and its variation $\tilde{C}^{*}_{def}(\Phi_\e)$. In order to illustrate that, recall first that a Lie groupoid morphism $\Phi:\calH\to\calG$ induces a \textit{push-forward} map $(\Phi_\e)_{*}:H^{1}_{def}(\calH)\to H^{1}_{def}(\Phi)$.

\begin{definition}
We will say that a family $[X_\e] \in H_{def}^1(\Phi_\e)$ has a \textbf{smooth pre-image in $H^1_{def}(\calH)$} if there exist smooth families of cocycles $Y_\e \in C^1_{def}(\calH)$ and $\bar{\alpha}^\e \in C^{0}_{def}(\Phi_\e)$ such that
\[X_\e = (\Phi_\e)_{*}(Y_\e)+\delta_{\Phi_\e}(\bar{\alpha}^\e).\] Analogously, considering the variation $\tilde{C}_{def}^{*}(\Phi_\e)$, since $\tilde{C}^{1}_{def}(\Phi_\e)=C^{1}_{def}(\Phi_\e)$, we say that a family $[X_\e] \in \tilde{H}_{def}^1(\Phi_\e)$ has a \textbf{smooth pre-image in $H^1_{def}(\calH)$} if there exist smooth families of cocycles $Y_\e \in C^1_{def}(\calH)$ and sections $\alpha'_{\e} \in \Gamma(A_\G)$ such that
\begin{equation}\label{eq:leftstronglytrivial}
X_\e = (\Phi_\e)_{*}(Y_\e)+\delta_{\Phi_\e}(\phi_\e^{*}\alpha_\e).
\end{equation}
\end{definition}

\begin{theorem}\label{almost-triviality}
Let $(\Phi_\e,\phi_\e)$ be a deformation of $(\Phi_0,\phi_0):(\calH\rightrightarrows N)\longrightarrow(\calG\rightrightarrows M)$. Assume that $\calH$ is compact and $\phi_0$ is an injective immersion. Then $\Phi_\e$ is strongly trivial up to automorphisms of $\calH$ if and only if $[X_\e]\in H^{1}_{def}(\Phi_\e)$ has a smooth pre-image in $H^{1}_{def}(\calH)$ by $(\Phi_\e)_{*}$ for $\e\in I$, where $I$ is an open interval containing 0.
\end{theorem}

\begin{proof} Since the proof of this theorem is completely analogous to the proof of the previous theorems, we will only sketch the main ingredients. Assume that
\begin{equation}\label{pre-image3}
X_\e-\delta_{\Phi_\e}(\bar{\alpha}_\e)=-(\Phi_\e)_{*}Y_\e,\ \ \text{for all } \e,
\end{equation}
for smooth families $\bar{\alpha}_\e\in C^{0}_{def}(\Phi_\e)$ and $Y_\e\in Z^{1}_{def}(\calH)$.
Let $F_\e$ denote the flow from time 0 to $\e$ of the time dependent vector field $Y_\e$ on $\calH$. Recall that since every $Y_\e$ is a 1-cocycle it follows that $F_\e$ is a family of automorphisms of $\calH$.
Next we enunciate the key steps which prove the statement.

\begin{itemize}

\item Shrinking $I$ if necessary, use the injective immersive condition of $\phi_0$ to extend the sections $\bar{\alpha}_\e\in\Gamma(\phi_\e^{*}A_\calG)$ to a smooth time dependent section $\alpha_\e$ in $\Gamma(A_\calG)$ which vanishes outside a compact set containing $\bigcup_{\e\in I}\phi_\e(N)\subset M$. That is, $\alpha_\e$ satisfies $\bar{\alpha_\e}=\Phi_\e^{*}\alpha_\e$.

\item Consider the smooth family $\sigma_\e$ of bisections of $\calG$ induced by the flow of the time dependent right-invariant vector field $\overrightarrow{\alpha}_\e$ on $\calG$, and check that $\Phi_\e\circ F_\e(h)=I_{\sigma_\e}\circ \Phi_0(h)$, for all $h\in\calH$. This last part can be checked by showing that both sides are integral curves of the time dependent vector field $\delta(\alpha_\e)$ on $\calG$ starting at the same point at the same time. Indeed, on the one hand, the derivative $\frac{d}{d\e}I_{\sigma_\e}\circ \Phi_0(h)$ follows from Proposition \ref{cociclomorph.}; and on the other hand, the derivative $\frac{d}{d\e}\Phi_\e\circ F_\e(h)$ is
\begin{align*}
X_\e(F_\e(h))+d\Phi_\e(Y_\e(F_\e(h)))&=\delta_{\Phi_\e}(\bar{\alpha}_\e)(F_\e(h))\\
&=\Phi_\e^{*}(\delta(\alpha_\e))(F_\e(h))\\
&=\delta(\alpha_\e)(\Phi_\e\circ F_\e(h)).
\end{align*}
\end{itemize}

The proof of the converse statement is a direct computation after applying $\frac{d}{d\e}$ to the left strongly trivial expression $\Phi_\e\circ F_\e=I_{\sigma_\e}\circ \Phi_0$.
\end{proof}

Analogously, we can avoid the injective immersive condition in the previous Theorem by considering the variation complex $\tilde{C}^{*}_{def}(\Phi_\e)$.

\begin{theorem}\label{thm:left-strongtriviality}
Let $(\Phi_\e,\phi_\e)$ be a deformation of $(\Phi_0,\phi_0):(\calH\rightrightarrows N)\longrightarrow(\calG\rightrightarrows M)$. Assume that $\calH$ is compact. Then $\Phi_\e$ is strongly trivial up to automorphisms of $\calH$ if and only if $[X_\e]\in\tilde{H}^{1}_{def}(\Phi_\e)$ has a smooth pre-image in $H^{1}_{def}(\calH)$ by $(\Phi_\e)_{*}$ for all $\e\in I$, where $I$ is an open interval containing 0.%***That is, they hold condition \eqref{eq:leftstronglytrivial}.
\end{theorem}

\begin{proof}
The proof of this theorem follows the idea similar to that of the previous Theorem, however we now use the argument of extension of sections of the Theorem \ref{thm:strongtriviality} to obtain the smooth family of global sections which give rise to the family of bisections of $\G$.
\end{proof}

We remark that considering the cokernel complex $\mathrm{Coker}((\Phi_\e)_{*})$ we get an equivalent result to that of Theorem \ref{almost-triviality}. Indeed, let $\bar{X}_\e\in\mathrm{Coker}((\Phi_\e)_{*})$ denote the image in the cokernel complex of the cocycles $X_\e$ associated to the deformation $\Phi_\e$.

\begin{theorem}\label{thm:strongtriv utal, exactness coker}
Let $(\Phi_\e,\phi_\e)$ be a deformation of $(\Phi_0,\phi_0):(\calH\rightrightarrows N)\longrightarrow(\calG\rightrightarrows M)$. Assume that $\calH$ is compact and $\phi_0$ is an injective immersion. Then $\Phi_\e$ is strongly trivial up to automorphisms of $\calH$ if and only if $\bar{X}_\e$ is smoothly exact for $\e\in I$, where $I$ is an open interval containing 0.
\end{theorem}

%**** Being exact in the cokernel amounts to check that the family of cocycles have a smooth pre-image.*** Under the immersive-injectiveness condition.

%cokernel + injective immers = mapping cone complex = pre-image? YES 90\... now 100%%

\begin{proof}
The proof of this theorem follows directly from the fact that, under the injectivity of the cochain maps $(\Phi_\e)_{*}$, the exactnees condition of the family of cocycles $\bar{X}_\e$ is equivalent to the existence of a smooth pre-image in $H^{1}_{def}(\calH)$ by $(\Phi_\e)_{*}$ of the classes of $X_\e$. Thus, we will be able to use the same proof of Theorem \ref{almost-triviality}.
\end{proof}

\begin{remark}\label{rmk:strongly triv utal, variated complex}
In contrast, we can modify the cokernel complex $Coker((\Phi_{\e})_{*})$ in zero degree changing it by the 0-cochains $\widetilde{C}_{def}^{0}(\Phi_\e)=\Phi_\e^{*}(\Gamma(A_\calG))$ of the variation complex, and the exactness of the deformation cocycles in this modified cokernel complex will allow us to define a trivial deformation different from that of Theorem \ref{thm:left-strongtriviality}. Indeed, in this case, the exactness of the deformation cocycles in the complex is equivalent to get a \textbf{strongly trivial up to diffeomorphisms on the left} deformation. These deformations are totally analogous to the ones which are strongly trivial up to automorphisms on the left, the only difference lies in taking a family of diffeomorphisms $F_\e$ instead of a family of automorphisms of the Lie groupoid as in Definition \ref{Def:strongly trivial up to automorp}.
\end{remark}

As an application of the Theorem \ref{thm:left-strongtriviality} we deduce now the particular case of 1-deformations of smooth maps of the Thom-Levine's Theorem (see \cite{guillemin-golubitsky}, p. 124) regarding the characterization of trivial $k$-deformations of smooth functions between manifolds. In Section \ref{Section:Thom-Levine} we will get the full Thom-Levine's Theorem as an instance of the deformation theory developed in this Section.

\begin{example}\label{thm:T-L thm 1-deform}[Thom-Levine's Theorem for 1-deformations]
Note that smooth functions between manifolds are in 1-1 correspondence with morphisms between the associated pair groupoids. Indeed, these morphisms are of the form $f\times f:\textrm{Pair}(M)\to\textrm{Pair}(N)$, where $f:M\to N$ is a smooth function between manifolds. From this one also gets that an automorphism of a pair groupoid is totally determined by the corresponding diffeomorphism on the base manifold.

%Let $f:M\to N$ be a smooth function between manifolds. Consider the induced morphism $f\times f:\textrm{Subm}(M)\to\textrm{Subm}(N)$ between the respective submersion groupoids.
With this setting, one checks that the characterizing equation \eqref{eq:leftstronglytrivial} of strongly trivial up to automorphisms deformations of morphisms between pair groupoids translates exactly to the condition which characterizes trivial deformations of smooth functions from the Thom-Levine Theorem. Indeed, if $F:M\times I\to N\times I$ is a deformation of the smooth function $f$, the characterizing condition given by Thom-Levine's Theorem is described by the equation

\begin{equation}\label{Eq:condition 1-deform T-L.}
F_{*}(\partial/\partial t)=F_{*}(\zeta)+ F^{*}(\eta)+F^{*}(\partial/\partial t),
\end{equation}
where $\zeta$ and $\eta$ are time-dependent vector fields on the source and target manifolds, respectively.
The equivalence between equations \eqref{eq:leftstronglytrivial} and \eqref{Eq:condition 1-deform T-L.} follows directly from the fact that 1-cocycles $Y\in C^{*}_{def}(\calG)$ (i.e. multiplicative vector fields) on a pair groupoid $\calG$ are of the form
$$Y(p,q)=(\zeta(p),\zeta(q)),$$
where $\zeta$ is an usual vector field on the base of the pair groupoid.
\end{example}

\begin{theorem}\label{thm:trivial utal, pre-image}
Let $(\Phi_\e,\phi_\e)$ be a deformation of $(\Phi_0,\phi_0):(\calH\rightrightarrows N)\longrightarrow(\calG\rightrightarrows M)$. Assume that $\calH$ is compact. Then $\Phi_\e$ is trivial up to automorphisms of $\calH$ if and only if $[X_\e]\in H^{1}_{def}(\Phi_\e)$ has a smooth pre-image in $H^{1}_{def}(\calH)$ by $(\Phi_\e)_{*}$ for all $\e\in I$, where $I$ is an open interval containing 0.
\end{theorem}

\begin{proof}
Assume that $[X_\e]$ has smooth preimage by $(\Phi_\e)_{*}$ for all $\e\in I$. That is, equation \eqref{pre-image3} is satisfied. Then, following the notations of the proof of the Theorem above, define the deformation $\Phi'_\e=\Phi_\e\circ F_\e$ of $\Phi_0$. The proof is completed after checking that $\Phi'_\e$ is a trivial deformation. In fact, observe first that the cocycles $\frac{d}{d\e}\Phi'_\e \in C^{1}_{def}(\Phi'_\e)$ are smoothly exact, and this smooth exactness follows from equation \eqref{pre-image3}. Therefore, the triviality of $\Phi'_\e$ is a consequence of Theorem \ref{gauge-triviality}.
\end{proof}

\begin{remark}\label{rmk:trivial utdl}
Notice that, by considering the cokernel complex $C^{*}_{def}(\Phi_\e)$, one similarly checks that the smooth exactness of the cocycles, for all $\e$ small enough, is equivalent to get a deformation which is \textbf{trivial up to diffeomorphisms of $\calH$}.
\end{remark}

As a final remark of this section, observe that all the notions of triviality defined here arise from the complexes and maps involved in the following diagram of exact sequences.\\

%*******CROSS DIAGRAM********

\[\xymatrix{ & C_{def}^{*}(\calH) \ar[d]^{\Phi_{*}} & \\
C_{def}^{*}(\calG) \ar[r]^{\Phi^{*}} & C_{def}^{*}(\Phi) \ar[r]^{\pi_1}   \ar[d]^{\pi_2} & \mathrm{Coker(\Phi^{*})}\\
 & \mathrm{Coker(\Phi_{*})} & \\
 }
\] 

Indeed we can summarizes the results as follows. %Pending when finish: ***MAKE A WIDER VISION of the triviality notions defined and of the diagram (e.g. some corner missing?, or similar? To simple  explanation for head and few words.)
On the vertical direction: under the injective immersive condition of the morphisms, the requirements on the deformation cocycles, regarding either existence of smooth pre-images or exactness on the cokernel complexes, turn out to be equivalent to characterize strongly trivial u.t.a.l. (up to automorphisms on the left) deformations (Theorems \ref{almost-triviality}, \ref{thm:strongtriv utal, exactness coker}). Under the non-injective immersive condition, the existence of smooth pre-image requirement gives us triviality up to automorphisms on the left (triviality u.t.a.l.), see Theorem \ref{thm:trivial utal, pre-image}. However the exactness of the cocycles on the cokernel complex, gives us triviality u.t.d.l. (up to diffeomorphisms on the left), see Remark \ref{rmk:trivial utdl}.

Analogously, considering the variated complex $\tilde{C}^{*}_{def}(\Phi)$ we can avoid the injective imersive requirement on the morphism and show that the existence of smooth pre-images condition gives us deformations which are strongly trivial u.t.a.l. (Theorem \ref{thm:left-strongtriviality}) and the exactness on a variation of the cokernel complex (using the 0-degree cochains $\widetilde{C}^{0}_{def}(\Phi_\e)$) gives us strongly trivial u.t.d.l. deformations (Remark \ref{rmk:strongly triv utal, variated complex}).

%Finally, we could try to add embedding in order to see if something gets stronger in the result. RTA: No, because injectivity we used it to extend sections of the Lie algebroid... but with the variation complex we do not need that condition for the purposes, we use restriction+extension of sections.

On the horizontal direction, pre-image by the pullback gives deformations which are trivial up to automorphisms on the right (u.t.a.r.), see Theorem \ref{thm:preimage pullback t.u.t.a.r.}. By adding the injective immersive condition on the morphism, we obtain strongly trivial u.t.a.r. (see Theorem \ref{weak-triviality}). %What about exactness on cokernel complex?***
Similarly, the weaker condition of exactness on the cokernel complex gives us deformations which are trivial u.t.a.r. after assuming surjectivity and submersion conditions on the morphism (Theorem \ref{thm:trivial utar, surbmersion}).
%And regarding the variated complex...

Finally, regarding the variated complexes $\tilde{C}_{def}^{*}(\Phi_\e)$, assuming the existence of smooth pre-images for the cocycles, we do not need to require the injective immersive condition on the morphism to prove the strong triviality u.t.a.r. of the deformation (see Theorem \ref{thm:preimage pullback strong t.u.t.a.r., variated complex}).

%%%%%%%%%%%%%%%%%%%%%%%%%%%% Section %%%%%%%%%%%%%%%%%%%%%%%%%%%%%

\section{$k$-deformations and Thom-Levine's Theorem}\label{Section:Thom-Levine}
In the previous sections we have studied 1-parameter deformations of morphisms or, in other words, paths of morphisms. The main point of this section is sketching the behaviour of deformations with $k$ parameters, that is, deformations depending on $I^{k}=I\times\cdots\times I$ instead of just $I$, and to get as an application the Thom-Levine's Theorem regarding triviality of $k$-deformations of differentiable maps.

%*******$k$-deformations of morphisms

\begin{definition}
Given a morphism $\Phi_0:\calH\to\calG$ and an interval $I$ containing zero, a Lie groupoid morphism $\Phi:\calH\times I^{k}\to\calG$ such that $\Phi(\cdot,0,...,0)=\Phi_0$ will be called a \textbf{$k$-deformation} of $\Phi_0$. For every $(\e_1,...,\e_k)\in I^{k}$ we denote by $\Phi_{\e_1,...,\e_k}$ the morphism $\Phi(\cdot,\e_1,...,\e_k):\calH\to\calG$.
\end{definition}

The notions of triviality, strong triviality and all the other equivalences between deformations of the previous section are defined by taking $k$-parameter families instead of 1-parameter families of the elements involved in the definitions. For instance, a $k$-deformation $\Phi$ of $\Phi_0$ is said to be \textbf{trivial} if there exists a smooth $k$-parameters family $\tau_{\e_1,...,\e_k}:N\to\G$ of gauge maps covering $\phi_0:N\to M$ such that $\tau_{0,...,0}=\phi_0$ and
$$\Phi_{\e_1,...,\e_k}=\tau_{\e_1,...,\e_k}\cdot\Phi_0.$$
Similarly a $k$-deformation is called \textbf{strongly trivial up to automorphisms on the left} if there exist smooth $k$-parameters families $F_{\e_1,...,\e_k}:\calH\to\calH$ and $\sigma_{\e_1,...,\e_k}\in\mathrm{Bis}(\G)$ of automorphisms of $\calH$ and bisections of $\G$ such that $F_{0,...,0}=Id_\calH$, $\sigma_{0,...,0}=u_\G$ and
$$\Phi_{\e_1,...,\e_k}\circ F_{\e_1,...,\e_k}=I_{\sigma_{\e_1,...,\e_k}}\circ\Phi_0,$$
for all $\e_1,...,\e_k$ small enough.

By fixing the first $(k-1)$ components of $I^{k}$, say in $a=(a_1,...,a_{k-1})\in I^{k-1}$, in a $k$-deformation $\Phi:\calH\times I^{k}\to\calG$, one gets a 1-parameter family $\Phi^{a,k}_\e$ which we call a \textbf{canonical 1-deformation in $\Phi$ along the (canonical) direction $e_k$} of $I^{k}$. Similarly one defines the \textbf{canonical 1-deformations in $\Phi$ along the other canonical directions} $e_i$ of $I^{k}$, for $i=1,...,k-1$ and all $\e_1,...,\e_k$ small enough.

\begin{remark}
Notice that all the canonical 1-deformations in a trivial $k$-deformation $\Phi$ are trivial but, on the other hand, all the canonical 1-deformations in $\Phi$ along a specific direction, say $e_1$, might be trivial without implying that the whole $k$-deformation $\Phi$ is trivial. In fact, as Theorem below shows, a necessary and sufficient condition for the triviality of $\Phi$ is that all the canonical 1-deformations in $\Phi$ are trivial.
Note also that, by Theorem \ref{gauge-triviality}, the triviality of all the canonical 1-deformations (in $\Phi$) along $e_1$ amounts to the existence of a section $\bar{\alpha}^{1}\in\Gamma(\phi^{*}A_\G)$ (i.e., a $k$-parameter family of sections $\bar{\alpha}^{1}_{\e_1,...,\e_k}\in\Gamma(\phi_{\e_1,...,\e_k}^{*}(A_\G))$) such that
$$\frac{\partial\Phi}{\partial e^{1}}=\delta_{\Phi}(\bar{\alpha}^{1}),$$
or equivalently
$$\frac{\partial\Phi_{\e_1,...,\e_k}}{\partial\e_{1}}=\delta_{\Phi_{\e_1,...,\e_k}}(\bar{\alpha}^{1}_{\e_1,...,\e_k}),$$
for all $(\e_1,...,\e_k)\in I^{k}$.
\end{remark}

\begin{theorem}
Let $\Phi:\calH\times I^{k}\to\calG$ be a $k$-deformation of $\Phi_0:\calH\to\calG$. Then the deformation $\Phi$ is trivial if and only if the canonical 1-deformations in $\Phi$ are trivial. That is, if and only if there exist $k$ $k$-families $\bar{\alpha}^{1},...,\bar{\alpha}^{k}\in\Gamma(\phi^{*}A_\G)$ of sections, such that
\begin{equation}
\frac{\partial\Phi}{\partial e^{i}}=\delta_{\Phi}(\bar{\alpha}^{i}),
\end{equation}
for $i=1,...,k$ and all $\e_1,...,\e_k$ small enough.
\end{theorem}

\begin{proof}

Assume that the canonical 1-deformations are trivial, and that $\tau^{1}_{\e_1,...,\e_k}$,..., $\tau^{k}_{\e_1,...,\e_k}$ are the families of gauge maps, induced by the families of sections $\bar{\alpha}^{1}_{\e_1,...,\e_k}$,..., $\bar{\alpha}^{k}_{\e_1,...,\e_k}$, which make the triviality along the canonical directions, that is, they hold
$$\Phi_{\e_1,...,\e_k}=\tau^{i}_{\e_1,...,\e_k}\cdot\Phi_{\e_1,...,\e_{i-1},0,\e_{i+1},...,\e_k},$$
for every $i=1,...,k$. Then, it follows that
\begin{equation}
\begin{split}
\Phi_{\e_1,...,\e_k}&=\tau^{k}_{\e_1,...,\e_k}\cdot\Phi_{\e_1,...,\e_{k-1},0}\\
&=\tau^{k}_{\e_1,...,\e_k}\cdot\tau^{k-1}_{\e_1,...,\e_{k-1},0}\cdots\tau^{2}_{\e_1,\e_2,0,...,0}\cdot\tau^{1}_{\e_1,0,...,0}\cdot\Phi_0,
\end{split}
\end{equation}
proving that the $k$-deformation $\Phi$ is trivial. The converse statement follows easily.
\end{proof}

Analogous to Theorem \ref{thm:left-strongtriviality}, to characterize $k$-deformations which are strongly trivial up to automorphisms on the left we consider the pre-images by $(\Phi_{\e_1,...,\e_k})_{*}$ of the deformation classes in $\tilde{H}_{def}^{*}(\Phi_{\e_1,...,\e_k})$ of the canonical 1-deformations in $\Phi$.

\begin{theorem}\label{thm:k-left-strongtriviality}
Let $\Phi:\calH\times I^{k}\to\calG$ be a $k$-deformation of $\Phi_0:\calH\to\calG$, and assume that $\calH$ is compact. Then the deformation $\Phi$ is strongly trivial up to automorphisms on the left if and only if the canonical 1-deformations in $\Phi$ are strongly trivial up to automorphisms on the left. That is, if and only if there exist $k$ $k$-families $\alpha^{1},...,\alpha^{k}\in\Gamma_{M\times I^{k}}(A_\G)$ and $Y^{1}$,..., $Y^{k}\in\Gamma_{\calH\times I^{k}}(T\calH)$ of sections of $A_\G$ and multiplicative vector fields on $\calH$, such that
\begin{equation}\label{eqs:towards Thom-Levine}
\frac{\partial\Phi_{\e_1,...,\e_k}}{\partial\e^{i}}=(\Phi_{\e_1,...,\e_k})_{*}Y^{i}_{\e_1,...,\e_k}+\delta_{\Phi_{\e_1,...,\e_k}}((\Phi_{\e_1,...,\e_k})^{*}\alpha^{i}_{\e_1,...,\e_k}),
\end{equation}
for $i=1,...,k$, and $\e_1,...,\e_k$ small enough.
\end{theorem}

\begin{proof}
Assume that equations \eqref{eqs:towards Thom-Levine} are satisfied by the $k$-deformation $\Phi$ then, by Theorem \ref{thm:left-strongtriviality}, every canonical 1-deformation in $\Phi$ is strongly trivial up to automorphisms of $\calH$. Let
$$\sigma^{1}_{\e_1,...,\e_k},..., \sigma^{k}_{\e_1,...,\e_k}\ \ \text{ and}\ \ \ F^{1}_{\e_1,...,\e_k},..., F^{k}_{\e_1,...,\e_k}$$
denote the families of bisections of $\G$ and automorphisms of $\calH$ induced, respectively, by the families of sections $\alpha^{1}_{\e_1,...,\e_k}$,..., $\alpha^{k}_{\e_1,...,\e_k}$ and vector fields $Y^{1}_{\e_1,...,\e_k}$,..., $Y^{k}_{\e_1,...,\e_k}$. Then, they hold
$$\Phi_{\e_1,...,\e_k}\circ F^{i}_{\e_1,...,\e_k}=I_{\sigma^{i}_{\e_1,...,\e_k}}\circ\Phi_{\e_1,...,\e_{i-1},0,\e_{i+1},...,\e_k},$$
for every $i=1,...,k$.
Hence, it follows that
\begin{equation}
\begin{split}
\Phi_{\e_1,...,\e_k}&=I_{\sigma^{k}_{\e_1,...,\e_k}}\circ\Phi_{\e_1,...,\e_{k-1},0}\circ\left(F^{k}_{\e_1,...,\e_k}\right)^{-1}\\
&=I_{\sigma^{k}_{\e_1,...,\e_k}}\circ\cdots\circ I_{\sigma^{1}_{\e_1,0,...,0}}\circ\Phi_{(0,...,0)}\circ\left(F^{k}_{\e_1,...,\e_k}\circ\cdots\circ F^{1}_{\e_1,0,...,0}\right)^{-1},
\end{split}
\end{equation}
which proves the strong triviality u.t.a.l. of $\Phi$. The converse statement follows directly in a similar way to that of the triviality case in the previous theorem.
\end{proof}

\begin{remark}
Note that the proofs of the previous two theorems can be seen as an application of a \textbf{zig-zag principle} where we go through each of the canonical directions checking the triviality one by one.
\end{remark}

\begin{remark}
Notice that, using this zig-zag principle, we can state analogous results for $k$-deformations considering all the other type of "trivial" deformations defined in the previous section. In fact, as in the previous two theorems, which imitate the corresponding comological equations of Theorems \ref{gauge-triviality} and \ref{thm:left-strongtriviality} in the Section above, we get then a set of $k$ cohomological equations imitating the respective 1-parameter ones which characterize each type of deformation.
\end{remark}

We next consider the notion of $k$-deformations of smooth maps necessary for the statement of the Thom-Levine Theorem which we establish below as well. After that we will see that the Thom-Levine Theorem is just a particular case of the Theorem \ref{thm:k-left-strongtriviality} above.

%*****$k$-deformations of functions..

Let $f_0:M\to N$ be a smooth map between the manifolds $M$ and $N$. The smooth $k$-family $\mathrm{f:M\times I^{k}\to N}$ of maps between the manifolds is called a \textbf{$k$-deformation of} $f_0$ if $f(\cdot,(0,...,0))=f_0$. We will also refer to the $k$-family $f$ by its restrictions $f_{\e_1,...,\e_k}:=f(\cdot,\e_1,...,\e_k)$. The deformation $f$ is said to be \textbf{trivial} if there exist $k$-deformations $F_{\e_1,...,\e_k}:M\to M$ and $G_{\e_1,...,\e_k}:N\to N$ of the identity maps $Id_M$ and $Id_N$ such that
\begin{equation}\label{def:triv defor of map}
f_{\e_1,...,\e_k}\circ F_{\e_1,...,\e_k}=G_{\e_1,...,\e_k}\circ f_0,
\end{equation}
for all $\e_1,...,\e_k$ small enough. The following version of Thom-Levine's Theorem we take from (\cite{guillemin-golubitsky}, p.124).

\begin{theorem}[Thom-Levine's Theorem]
Let $f:M\times I^{k}\to N$ be a $k$-deformation of $f_0:M\to N$, and assume that $M$ is compact. Then $f$ is trivial if and only if there exist smooth families $\zeta^{i}_{\e_1,...,\e_k}\in\mathfrak{X}(M)$ and $\eta^{i}_{\e_1,...,\e_k}\in\mathfrak{X}(N)$ of vector fields on $M$ and $N$ (for $i=1,...,k$) such that

\begin{equation}\label{Eq:condition k-deform T-L.}
\frac{\partial f_{\e_1,...,\e_k}}{\partial\e_i}=(df_{\e_1,...,\e_k})(\zeta^{i}_{\e_1,...,\e_k})+f^{*}_{\e_1,...,\e_k}(\eta^{i}_{\e_1,...,\e_k}),
\end{equation}
for $i=1,...,k$.
\end{theorem}

\begin{remark}
Notice that when considering 1-deformations, as in Example \ref{thm:T-L thm 1-deform}, equations \eqref{Eq:condition k-deform T-L.} are just another way to express equation \eqref{Eq:condition 1-deform T-L.}.
\end{remark}

\begin{remark}
We remark that Theorem \ref{thm:k-left-strongtriviality} turns out to be a generalization to Lie groupoids of the Thom-Levine Theorem, more precisely, one checks that the Thom-Levine Theorem identifies exactly with the Theorem \ref{thm:k-left-strongtriviality} applied to morphisms between pair groupoids. Indeed, similar to Example \ref{thm:T-L thm 1-deform}, which considers 1-deformations, a $k$-deformation of a morphism between two pair groupoids amounts to a $k$-deformation of smooth maps and, also, a strongly trivial up to automorphisms on the left deformation of a morphism between pair groupoids is equivalent to a trivial $k$-deformation of the corresponding smooth map between the base manifolds of the pair groupoids. And additionally, the equations \eqref{eqs:towards Thom-Levine} translate exactly to the equations \eqref{Eq:condition k-deform T-L.}, where this matching follows directly from the fact that 1-cocycles $Z\in C^{*}_{def}(\calG)$ (i.e. multiplicative vector fields) on a pair groupoid $\calG$ are of the form
$$Z(p,q)=(\zeta(p),\zeta(q)),$$
where $\zeta$ is an usual vector field on the base of the pair groupoid.
\end{remark}

\begin{remark}
The Thom-Levine Theorem is a supporting step in order to prove the equivalence between the \emph{stability} and the \emph{infinitesimal stability} of smooth maps (\cite{guillemin-golubitsky}, p. 127). Heuristically, the original proof of Thom-Levine Theorem and ours are very similar, consisting in integrate certain vector fields to get the triviality of the deformation out of these integrations, however our proof using Lie groupoids turns out to be more straightforward and geometrical, no needing so many technical details and also obtaining the necessity condition for triviality directly (and with a fast computation) as a velocity interpretation of the deformation.
%The original proof of Thom-Levine's theorem is based on...
%however our proof just uses...
In summary, to get the Thom-Levine Theorem, we have directly used Theorems \ref{thm:k-left-strongtriviality}, \ref{thm:left-strongtriviality}, \ref{thm:strongtriviality} and Proposition \ref{prop:exactcocyles variatedcomplex}.
\end{remark}

%***** something about stability using Thom-Levine's Theorem...

%%%%%%%%%%%%%%%%%%%%%%%%%  Section  %%%%%%%%%%%%%%%%%%%%%%%%%%%%%%%%%%%

\section{Stability of morphisms}\label{Section:Rigidity}

%-Rigidity, examples of rigid morphisms. Rigid subgroupoids(?).\\

In this section we apply the results obtained in the Section \ref{Section:Triviality} in order to obtain stability properties of Lie groupoid morphisms under deformations. The key fact in the proof of the results below will be to combine the vanishing results (to get smooth transgressions of the deformation cocycles) with the Moser type arguments explored in Section \ref{Section:Triviality}.

\begin{theorem}\label{thm:gauge-rigidity}
Let $\Phi_0:\calH\longrightarrow\calG$ be a Lie groupoid morphism. Assume that $\calG$ is transitive. If either the groupoid $\calH$ is proper or $\G$ has trivial isotropy (i.e. $\G$ is a pair groupoid!), then any deformation of $\Phi_0$ is trivial.
\end{theorem}

\begin{proof}
Following the notations of Definition \ref{def:smoothfamilycochains}, let $(\tilde{\Phi},\tilde{\phi}):\calH\times I\to\calG$ be a deformation of $\Phi_0$. On the one hand, since the groupoid $\calG$ is transitive, it follows that the normal bundle $\nu_{\calG}$ becomes the null bundle over $M$. That fact implies that the 1-cocycle $\tilde{X}:=\tilde{\Phi}_{*}(\partial/\partial\e)\in C^{1}_{def}(\tilde{\Phi})$ (see the proof of Proposition \ref{prop:exactcocyles}, where this cocycle is defined), corresponding to the deformation $\tilde{\Phi}$, lies in the kernel $H^{1}(\calH\times I, \tilde{\phi}^{*}\mathfrak{i}_\calG)$ of the map $\tilde{s}:H^{1}_{def}(\tilde{\Phi})\to\Gamma(\tilde{\phi}^{*}\nu_{\calG})^{\mathrm{inv}}$ of the sequence \eqref{lowdegreemorphism} in Section \ref{section:lowdegrees}.
On the other hand, the condition of either trivial isotropy of $\G$ or properness of $\calH$ implies that $H^{1}(\calH\times I, \tilde{\phi}^{*}\mathfrak{i}_\calG)$ vanishes. Hence, the 1-cocycle $\tilde{X}\in C^{1}_{def}(\tilde{\Phi})$ is exact, which amounts to the smooth exactness of the family of 1-cocycles $X_\e=\frac{d}{d\e}\Phi_\e$ associated to the deformation. Thus, the Theorem \ref{gauge-triviality} concludes the proof.
\end{proof}

Additionally, by using either the Proposition \ref{Prop:familybisectionsfromgauge} or Theorem \ref{thm:homostronglytrivial}, we then obtain:

\begin{corollary}\label{Corol 1, stability}
Let $\Phi_0:\calH\longrightarrow\calG$ be a Lie groupoid morphism whose base map is an injective immersion. Assume that $\G$ is transitive. If the groupoid $\calH$ is compact, then any deformation of $\Phi_0$ is strongly trivial.
\end{corollary}

%Or by using Theorem \ref{thm:homostronglytrivial}, one gets:
\begin{corollary}\label{Corol 2, stability}
Let $\Phi_0:\calH\longrightarrow\calG$ be a Lie groupoid morphism whose base map is an injective immersion. Assume that $\G$ is transitive and that the base $N$ of $\H$ is compact. If the groupoid $\G$ has trivial isotropy, then any deformation of $\Phi_0$ is strongly trivial.
\end{corollary}

Thus, Theorem \ref{thm:gauge-rigidity} says that under properness and transitivity of $\H$ and $\G$, respectively, any curve $\e\mapsto\Phi_\e$, with $\Phi_0=\Phi$, is constant when viewed in the category of Lie groupoids and isomorphism classes of morphisms. Moreover, such conditions of compactness and transitivity assumed in the results are fundamental: for instance, any curve passing through more than one orbit of $\G$ determines %if we %\emph{assume that $\H$ is proper and $\calG$ is non-transitive}, \emph{assume that $\calG$ is non-transitive}, then there exist
a non-trivial deformation for a constant morphism $\Phi_0$. %: Let $\gamma:I\to \calG$ be a smooth curve on the unit space $M$ which goes across the orbits of $\calG$. Then, the family $\Phi_\e$ of constant morphisms $\Phi_\e(h)=\gamma(\e)$ is a non trivial deformation of $\Phi_0$. 
It also follows that if we additionally take $\calH$ as being a compact Lie group, then the base map of $\Phi_0$ is an injective immersion but the deformation will not be strongly trivial; thus the transitivity of $\G$ is also necessary in the corollaries. Also for $\calH$ non-compact in Corollary \ref{Corol 1, stability}, one can check that if $\G=\mathbb{R}^{2}$ and $\H=\mathbb{R}$ viewed as Lie groups, then the inclusions $i_{\e}(x)=(x,(1-\e)x)$ of the linear spaces in $\mathbb{R}^{2}$ yield a non-strongly trivial deformation even if $\G$ is transitive. Considering now the last corollary, this latter counterexample also verifies that the trivial isotropy condition can not be removed in the statement. Additionally, assuming that $N$ is non-compact, take $\G=Pair(\mathbb{R})=\calH$ the pair groupoids over $\mathbb{R}$, since the identity map over $\mathbb{R}$ can be deformed, by using a bump function, to not be a diffeomorphism for any $\e$ small, it follows that it induces a non-strongly trivial deformation of the identity morphism.

%%%%%%%%  A HACER PARA VERSION 2 DE ARXIV:  %%%%%%%%%%%%%%%%%%%%%% ****And $N$ compact and $\G$ with isotropy and transitive.

%Analogously, the Theorem \ref{thm:gauge-rigidity} is not true if we \emph{assume that $\calH$ is non-proper}. Let $\calH=\mathbb{R}\times\mathbb{R}$ be the bundle of Lie groups over $\mathbb{R}$, and $\calG=\mathbb{R}$ viewed as a Lie groupoid over a point. Consider the morphism $\Phi(r,t):=r$, which is the identity over every fiber. Then, $\Phi_\e:=(1-\e)\cdot\Phi$ is a non trivial deformation of $\Phi$.\\

%%%%%%%%%%%%%%%%%%%%%%%%%%%%  SECTION  %%%%%%%%%%%%%%%%%%%%%%%%%%%%

\section{Simultaneous deformations}\label{Section:Simultaneous}

In this section we put together the deformation theory of  both Lie groupoids and Lie groupoid morphisms in order to study the most general deformation problem: the simultaneous deformation of the triple given by $\H\stackrel{\Phi}{\longrightarrow}\G$, where $\Phi$ is a morphism of Lie groupoids.

\begin{definition}
Let $\Phi_0:\H\to\G$ be a Lie groupoid morphism. A \textbf{deformation of the triple} $(\H,\Phi_0,\G)$ is a deformation $\tilde{\H}$ of $\H$ and a deformation $\tilde{\G}$ of $\G$ over a common open interval $I$ containing 0, and a morphism $\tilde{\Phi}:\tilde{\H}\to\tilde{\G}$ which deforms $\Phi_0$ in a compatible manner, in the sense that $\tilde{\Phi}|_{\tilde{\H}_0}=\Phi_0$ and, for each $\e\in I$, $\Phi_\e:=\tilde{\Phi}|_{\tilde{\H}_\e}:\tilde{\H}_\e\to\tilde{\G}_\e$ is a morphism of Lie groupoids. We will denote a deformation of the triple $(\H,\Phi_0,\G)$ by $(\H_\e,\Phi_\e,\G_\e)$.
\end{definition}

\begin{definition}
Let $(\H_\e,\Phi_\e,\G_\e)$ and $(\H'_\e,\Phi'_\e,\G'_\e)$ be two deformations of the triple $(\H,\Phi,\G)$. We will say that both deformations are \textbf{equivalent} if there exist an open interval $I$ containing 0 and isomorphisms of Lie groupoids $\tilde{F}:\tilde{\H}\to\tilde{\H}'$, $\tilde{G}:\tilde{\G}\to\tilde{\G}'$ and a gauge map $\tau:\tilde{N}\to\tilde{\G}$ covering $\phi:\tilde{N}\to\tilde{M}$ defined over $I$ such that
\begin{equation}\label{eq:equivalencetriple}
G^{-1}\circ\Phi'\circ F=\tau\cdot\Phi.
\end{equation}
\end{definition}
\noindent We will also say that a deformation $(\tilde{\H},\Phi,\tilde{\G})$ of $(\H,\Phi_0,\G)$ is \textbf{trivial} if it is equivalent to the \textbf{constant deformation} $(\H\times I,\Phi_0\times Id,\G\times I)$.

\begin{remark}
On rigidity of triples it is straighforward checking, from the rigidity of compact Lie groupoids and Lie groupoid morphisms from Section \ref{Section:Rigidity}, that if $\H$ and $\G$ are compact Lie groupoids, and moreover $\G$ is transitive then any deformation $\tilde{\Phi}:\tilde{\H}\to\tilde{\G}$ of the triple $\Phi:\H\to\G$ is trivial. 
\end{remark}

\begin{remark}
For deformations of the triple $(\H,\Phi_0,\G)$ we can also define \textbf{strongly equivalent} deformations. This equivalence relation corresponds to the special case in which the gauge map $\tau$ of equation \eqref{eq:equivalencetriple} is induced by a bisection $\sigma$ of $\tilde{\G}$. That is, when $\tau=\sigma\circ\phi$, for $\sigma\in\mathrm{Bis}(\tilde{\G})$. A deformation will be called \textbf{strongly trivial} if it is strongly equivalent to the constant deformation.
\end{remark}

\subsection{Deformation complex and triviality of simultaneous deformations}

Consider the diagram of cochain maps explained in Remark \ref{chain maps}
\begin{equation*}
C_{def}^{k}(\calH)\stackrel{\Phi_{*}}{\longrightarrow}C_{def}^{k}(\Phi)\stackrel{\Phi^{*}}{\longleftarrow}C_{def}^{k}(\calG).
\end{equation*}
%Out of this data,
We construct the complex which controls the deformations of the triple $\calH\stackrel{\Phi}{\longrightarrow}\calG$ by taking the cone of this diagram as follows. Take the mapping-cone complex associated to the cochain map $\Phi_{*}$,

\begin{equation}\label{MC(phi)}
C^{*}(\Phi_{*})=C^{*+1}_{def}(\calH)\oplus C^{*}_{def}(\Phi) 
\end{equation}
with differential $\delta_{C(\Phi_{*})}(c,Y)=\left(\delta_{\calH}c,\ \Phi_{*}c-\delta_{\Phi}Y\right)$. Notice that the map $\Phi^{*}$ above induces a cochain map $\tilde{\Phi}^{*}:C^{*}_{def}(\calG)\longrightarrow C^{*}(\Phi_{*})$ putting zero on the first component, $\tilde{\Phi}^{*}:c\longmapsto (0,(-1)^{deg(c)}c)$. Take now the mapping-cone associated to $\tilde{\Phi}^{*}$, getting the \textbf{deformation complex of the triple},

\begin{equation*}
\begin{split}
C^{*+1}_{def}(\calH,\Phi,\calG)&:=\mathrm{Mapping}(\tilde{\Phi}^{*})=C^{*}(\Phi_{*})\oplus C^{*+1}_{def}(\calG)\\
&=C^{*+1}_{def}(\calH)\oplus C^{*}_{def}(\Phi)\oplus C^{*+1}_{def}(\calG),
\end{split}
\end{equation*}
with differential $\delta(c,X,\bar{c})=\left(-\delta_{\calH}c, \delta_{\Phi}(X)-\Phi_{*}c+(-1)^{deg(\bar{c})}\Phi^{*}\bar{c},\delta_{\calG}\bar{c}\right)$.

In this way, given a $s$-constant deformation $(\calH_\e\stackrel{\Phi_\e}{\longrightarrow}\calG_\e)$ of $(\calH,\Phi,\calG)$ (i.e., those where $\H_\e$ and $\G_\e$ are $s$-constant deformations), by computations similar to those in Proposition \ref{cociclomorph.}, %and Proposition \ref{cociclosubgpd},
we get that $(\xi_\calH,-X,\xi_\calG)$ is a 2-cocycle in $C^{*}_{def}(\calH,\Phi,\calG)$, where $\xi_\calH$ and $\xi_\calG$ are the respective deformation cocycles for $\calH$ and $\calG$, and $X:=\left.\frac{d}{d\e}\right|_{\e=0}\Phi_\e$ is the usual cochain (Section \ref{Section:Triviality}) associated to a deformation of morphisms. Indeed, the fact that $(\xi_\calH,-X,\xi_\calG)$ is a cocycle follows by applying $\left.\frac{d}{d\e}\right|_{\e=0}$ to the compatibility of the deformation $(\calH_\e,\Phi_\e,\calG_\e)$:

\begin{align*}
& &\left.\frac{d}{d\e}\right|_{\e=0}\Phi_\e(\bar{m}_{\calH_\e}(gh,h))&=\left.\frac{d}{d\e}\right|_{\e=0}\bar{m}_{\calG_\e}(\Phi_\e(gh),\Phi_\e(h))\\
&\Longleftrightarrow&  X(g)+\Phi_{*}\xi_{\calH}(g,h)&=\xi_{\calG}(\Phi(g),\Phi(h))+d\bar{m}_{\calG}(X(gh),X(h))\\
&\Longleftrightarrow&  \delta_{\Phi}(X)+\Phi_{*}\xi_{\calH}-\Phi^{*}\xi_{\calG}&=0.
\end{align*}
The fact that the corresponding cohomology class of $(\xi_\calH,-X,\xi_\calG)$ only depends on the equivalence class of the deformation is also an analogous computation.

\begin{remark}
One can also consider a non $s$-constant deformation $(\tilde{\H},\tilde{\Phi},\tilde{\G})$ of the triple $(\H,\Phi,\G)$ and obtain its associated deformation cohomology class. Indeed, if $\tilde{X}_\H\in\mathfrak{X}(\widetilde{\H})$ and $\tilde{X}_\G\in\mathfrak{X}(\widetilde{\G})$ are transversal vector fields (i.e., vector fields which project to $\partial/\partial\e\in\mathfrak{X}(I)$) with $\xi_\H\in C^{2}_{def}(\H)$ and $\xi_\G\in C^{2}_{def}(\G)$ being the corresponding deformation cocycles, then $(\xi_\H,X,\xi_\G)\in C^{2}_{def}(\H,\Phi,\G)$ is the associated deformation cocycle, where $X:=\tilde{\Phi}_{*}\tilde{X}_\H|_{\H}-\tilde{\Phi}^{*}\tilde{X}_\G|_{\H}\in C^{1}_{def}(\Phi)$.
\end{remark}
We pass now to establish the main results concerning simultaneous $s$-constant deformations.

\begin{theorem}
Let $(\calH_\e,(\Phi_{\e},\phi_\e),\calG_\e)$ be a deformation of $(\calH\stackrel{\Phi}{\longrightarrow}\calG)$, with $\calH$ and $\calG$ compact. If the family of associated cocycles $(\xi_{\calH_\e},-X_\e,\xi_{\calG_\e})\in C_{def}^{2}(\calH_\e, \Phi_\e, \calG_\e)$ is transgressed by a smooth family of cochains $(Y_\e, \tilde{\alpha}_\e, Z_\e)\in C_{def}^{1}(\calH_\e, \Phi_\e, \calG_\e)$, then the deformation $(\calH_\e,\Phi_{\e},\calG_\e)$ is trivial.
\end{theorem}

\begin{proof}
Exactness of the family of cocycles amounts to

\begin{equation}\label{eq:transgresionequations}
\begin{split}
\xi_{\calH_\e}&=-\delta_{\calH_\e}Y_\e;\\
-X_\e&=\delta_{\Phi_\e}(\tilde{\alpha}_\e)-\Phi_{*}Y_\e-\Phi^{*}Z_\e;\\
\xi_{\calG_\e}&=\delta_{\calG_\e}Z_\e.\\
\end{split}
\end{equation}
By the first and third equation, if $\varphi_\e$ and $\psi_\e$ are the time-dependent flows starting at zero of the vector fields $\{-Y_\e\}$ and $\{Z_\e\}$ respectively, then they define the equivalences with the corresponding constant deformations of Lie groupoids. We claim that these equivalences can be used to prove the triviality of $(\calH_\e,\Phi_{\e},\calG_\e)$. In fact, by using Theorem \ref{gauge-triviality}, we will show that the family of morphisms $\{f_\e:=\psi_{\e}^{-1}\circ\Phi_\e\circ\varphi_\e\}$ is a trivial deformation in the sense of the definition at the beginning of Section \ref{Section:Deform.morph.}. Such a family satisfies,

\begin{equation*}
\begin{split}
\frac{d}{d\e}f_\e(h)&=\left(\frac{d}{d\e}\psi_{\e}^{-1}\right)(\Phi_\e\circ\varphi_\e(h))+d\psi_{\e}^{-1}(X_\e(\varphi_\e(h)))-d(\psi_{\e}^{-1}\circ\Phi_\e)\left(Y_\e(\varphi_\e(h))\right)\\
&=-d\psi_{\e}^{-1}\left(Z_\e(\Phi_\e\circ\varphi_\e(h))\right)+d\psi_{\e}^{-1}\left(X_\e(\varphi_\e(h))\right)-d\psi_{\e}^{-1}\left[d\Phi_\e\left(Y_\e(\varphi_\e(h))\right)\right]\\
&=d\psi_{\e}^{-1}\left[X_\e(\varphi_\e(h))-((\Phi_\e)_{*}Y_\e)_{(\varphi_\e(h))}-(\Phi_{\e}^{*}Z_\e)_{(\varphi_\e(h))}\right]\\
&=-d\psi_{\e}^{-1}\left(\delta_{\Phi_\e}(\tilde{\alpha}_\e)(\varphi_\e(h))\right),\ \ \ \ \text{by equations } \eqref{eq:transgresionequations}\\
&=-d\psi_{\e}^{-1}\left(\delta_{\Phi_\e\circ\varphi_\e}(\bar{\alpha}_\e)(h)\right)\ \ \ (\text{by }\bar{\alpha}_\e=\varphi_\e^{*}\tilde{\alpha}_\e)\\%(\varphi_e^{*}\text{ is cochain-map})
&=-\delta_{\psi_{\e}^{-1}\circ\Phi_\e\circ\varphi_\e}(\alpha_\e)(h)\ \ \ (\text{by }\alpha_\e=-(\psi^{-1}_\e)_{*}\bar{\alpha}_\e)\\%((\psi_e^{-1})^{*}\text{ is cochain-map})
&=\delta_{f_\e}(\alpha_\e)(h),
\end{split}
\end{equation*}
where the second equality follows from the fact that $$(\frac{d}{d\e}\psi_\e^{-1})(\psi_\e(h))+d\psi_\e^{-1}(Z_\e(\psi_\e(h)))=0,$$ which is obtained by applying $\frac{d}{d\e}$ to $\psi_\e^{-1}\circ\psi_\e=Id$. Therefore, by Theorem \ref{gauge-triviality}, $\psi_{\e}^{-1}\circ\Phi_\e\circ\varphi_\e=\tau_\e\cdot\Phi$, as claimed.
\end{proof}

Analogously, one can check the following

\begin{theorem}\label{Deformtriples}%(Triviality of deformations)

Let $(\calH_\e,(\Phi_{\e},\phi_\e),\calG_\e)$ be a deformation of $(\calH\stackrel{\Phi}{\longrightarrow}\calG)$, with $\calH$ and $\calG$ compact, and $\phi_0$ an injective immersion. If the family of associated cocycles $(\xi_{\calH_\e},-X_\e,\xi_{\calG_\e})\in C_{def}^{2}(\calH_\e, \Phi_\e, \calG_\e)$ is transgressed by a smooth family of cochains $(Y_\e, \tilde{\alpha}_\e, Z_\e)\in C_{def}^{1}(\calH_\e, \Phi_\e, \calG_\e)$, then the deformation $(\calH_\e,\Phi_{\e},\calG_\e)$ is strongly trivial.
\end{theorem}

\subsection{Particular cases and relations between (sub)complexes}

In view that the complex $C^{*}_{def}(\calH,\Phi,\calG)$ controls the most general type of deformations of the three structures $(\calH,\Phi,\calG)$, in this section we consider particular cases of deformations of the triple and their relation with some subcomplexes of $C^{*}_{def}(\calH,\Phi,\calG)$. We begin with the simplest case.

\subsubsection*{\textbf{$\calH$ and $\calG$ are fixed}}

In this case, we get a deformation $(\H,\Phi_\e,\G)$ of a Lie groupoid morphism. This fact is expressed, in cohomological terms, by the injection $C^{*}_{def}(\Phi)\longrightarrow C^{*}_{def}(\calH,\Phi,\calG)$
$$X\mapsto(0,-X,0).$$

Moreover, this map takes the infinitesimal cocycle of $\Phi_\e$ to the infinitesimal cocycle of $(\calH,\Phi_\e,\calG)$. Therefore, in this case, the relevant subcomplex controlling deformations of this type is given by $\{0\}\oplus C^{*}_{def}(\Phi)\oplus\{0\}$ as expected.

\subsubsection*{\textbf{$\calG$ is fixed}}

In this case, the relevant subcomplex is $C^{*+1}_{def}(\calH)\oplus C^{*}_{def}(\Phi)\oplus \{0\}$. In fact, it is not hard to see that a deformation of the form $(\calH_\e,\Phi_\e,\calG)$ is governed by the mapping-cone complex $C^{*}((\Phi)_{*})$ (see \eqref{MC(phi)}), where one associates the cocycle $(\xi_{\calH},-X)\in C^{2}_{def}(\calH)\oplus C^{1}_{def}(\Phi)$ to the deformation. Thus the (injective) chain map $C^{*}(\Phi_{*})\longrightarrow C^{*}_{def}(\calH,\Phi,\calG)$
\begin{equation}\label{injection}
(c,X)\mapsto (-1)^{deg(c)}(c,X,0)
\end{equation}
shows that the subcomplex $C^{*+1}_{def}(\calH)\oplus C^{*}_{def}(\Phi)\oplus \{0\}$ controls the deformations of the triple when we fix the groupoid $\calG$.

This kind of $\G$-fixed deformations is quite related to what is called \textit{deformations of Lie subgroupoids}. Indeed, it can be checked that the subcomplex $C^{*+1}_{def}(\calH)\oplus C^{*}_{def}(\Phi)\oplus \{0\}$ can be viewed as the complex which controls such a deformations. The details of that will be developed in the future work \cite{CardS-Subgroupoids}.

%%%%%%%%%%%%%%%%%%%%%%%%%%%%%% SECTION %%%%%%%%%%%%%%%%%%%%%%%%%%%%%%%%%%%

\section{Morita invariance and Deformation cohomology of generalized morphisms}\label{Sec:Moritainv}
We now investigate the behaviour of the deformation cohomology under Morita maps of Lie groupoids, show its invariance by Morita morphisms and use it to define a deformation cohomology for generalized morphisms between Lie groupoids.
%Such a fact evidences the natural role that the cohomologies defined in this paper have in the theory of Lie groupoids.
%In particular, this indicates that the deformation cohomology can be though as a cohomology associated to the particular class of morphisms between stacks which can be represented by actual morphisms of Lie groupoids.
The proof of the invariance results here are just applications of the recent developed concept of VB-Morita maps between VB-groupoids \cite{dHO}.

\begin{proposition}\label{Mor1}
Let $\Phi:\calH\longrightarrow\calG$ be a Lie groupoid morphism. Assume $F:\calH'\longrightarrow\calH$ is a Morita map. Then $H^{\bullet}_{def}(\Phi)\cong H^{\bullet}_{def}(\Phi\circ F)$.
\end{proposition}

\begin{proof}
Recall that the deformation complexes of $\Phi$ and $\Phi\circ F$ are respectively isomorphic to the VB-complexes of the pullback VB-groupoids $\Phi^{*}T^{*}\calG$ and $F^{*}\Phi^{*}T^{*}\calG$. Then, since $F$ is a Morita map, then the canonical vector bundle map $F^{*}\Phi^{*}T^{*}\calG\longrightarrow\Phi^{*}T^{*}\calG$ insures that such VB-groupoids are VB-Morita equivalent (see Corollary 3.7 of \cite{dHO}) and thus, by the VB-Morita invariance of the VB-cohomology of \cite{dHO}, we have that $H^{\bullet}_{def}(\Phi)\cong H^{\bullet}_{VB}(\Phi^{*}T^{*}\calG)\cong H^{\bullet}_{VB}(F^{*}\Phi^{*}T^{*}\calG)\cong H^{\bullet}_{def}(\Phi\circ F)$, as claimed.

\end{proof}

\begin{proposition}\label{Mor2}
Let $\Phi:\calH\longrightarrow\calG$ be a Lie groupoid morphism. If $F:\calG\longrightarrow\calG'$ is a Morita map, then $H^{\bullet}_{def}(\Phi)\cong H^{\bullet}_{def}(F\circ\Phi)$.
\end{proposition}

\begin{proof}

Recall that the complexes computing the deformation cohomology of $\Phi$ and $F\circ\Phi$ are isomorphic, respectively, to the VB-complexes of the two VB-groupoids $\Phi^{*}T^{*}\calG$ and $\Phi^{*}F^{*}T^{*}\calG'$, thus we shall prove that the cohomologies of these VB-complexes are isomorphic. 

Since $F$ is a Morita map, it follows that the differential $TF:T\calG\to T\calG'$ and the canonical bundle map $F^{*}T\calG'\longrightarrow T\calG'$ are VB-Morita maps. Thus, the induced map $(TF)^{!}:T\calG\longrightarrow F^{*}T\calG'$ of the VB-groupoids (over $\G$) turns out to be also a VB-Morita map.

Therefore, Corollary 3.9 in \cite{dHO} ensures that its dual map
$$\Psi:=((TF)^{!})^{*}:F^{*}T^{*}\calG'\longrightarrow T^{*}\calG$$ %(WRITE ON THIS COMMUTATIVITY OF DUAL AND PULLBACK OF VB!)
is a VB-Morita map. Hence, finally by taking the pullback by $\Phi$ of these VB-groupoids, one gets the induced VB-Morita map %(WRITE!)pullback takes VB-morita maps to VB-morita maps
$\Phi^{*}\Psi:\Phi^{*}F^{*}T^{*}\calG'\longrightarrow \Phi^{*}T^{*}\calG$ which, by the VB-Morita invariance of the VB-cohomology \cite{dHO}, then induces the isomorphism of the indicated VB-cohomologies. 
\end{proof}

We can use now the results in this Section to define a deformation cohomology for \emph{generalized maps} which are regarded as the morphisms in the category of \emph{differentiable stacks}.

\subsection{A deformation complex for fractions}

In the setting of the theory of \emph{localization} of categories and calculus of fractions \cite{Kashiwara2006categoriesandsheaves}, given two Lie groupoids $\calH$ and $\calG$, a \textbf{fraction} $\Psi/\Phi:\calH\longrightarrow\calG$ is defined by two maps $\Phi:\calK\to\calH$ and $\Psi:\calK\to\calG$ where $\calK$ is a third Lie groupoid and $\Phi$ is a Morita morphism. That fraction is often also denoted by $\Psi\Phi^{-1}$.
\begin{equation}\label{diagr:fraction}
\calH\stackrel{\Phi}{\longleftarrow}\calK\stackrel{\Psi}{\longrightarrow}\calG.
\end{equation}
Two fraction are said to be \textbf{equivalent} $\Psi_1/\Phi_1\cong\Psi_2/\Phi_2$ if there exist a third fraction $\Psi_3/\Phi_3$ and Morita maps $F_1:\calK_3\to\calK_1$ and $F_2:\calK_3\to\calK_2$ making the below diagram commutative up to isomorphisms of morphisms.

\[\xymatrix{ & \calK_1 \ar[ld]_{\Phi_1} \ar[rd]^{\Psi_1} & \\
\calH & \calK_3 \ar[r]_{\Psi_3} \ar[l]^{\Phi_3} \ar[u]^{F_1} \ar[d]_{F_2} & \calG\\
 & \calK_2 \ar[lu]^{\Phi_2} \ar[ru]_{\Psi_2} & \\
% \ar@<-0.25pc>[d]& \calG\times I \ar@<0.25pc>[d]
%\ar@<-0.25pc>[d]\\
%N\times I \ar[r]_{\phi} & M\times I
 }
\]

This is indeed an equivalence relation on fractions as can be proved by using weak fibred products (\cite{MM}, p. 124). The equivalence class of the fraction $\Psi/\Phi$ determines a \textbf{generalized map} $\mathrm{\Psi/\Phi:\calH\dashrightarrow\calG}$, also known as \textbf{generalized morphism} or \textbf{stacky map}, which can be viewed as a smooth map between the differentiable stacks presented by $\calH$ and $\calG$ (see \cite{2019riemannianstacks}, Section 6.2).

Given the fraction $\Psi/\Phi$, there is a map of complexes $\Phi_{*}\oplus\Psi_{*}:C_{def}^{*}(\calK)\to C_{def}^{*}(\Phi)\oplus C_{def}^{*}(\Psi)$ induced by the morphisms $\Phi$ and $\Psi$. We define the \textbf{deformation complex of the fraction $\Psi/\Phi$} by the mapping-cone complex of the map $\Phi_{*}\oplus\Psi_{*}$. That is,
$$C_{\mathrm{def}}^{*}(\Psi/\Phi):=C_{\mathrm{def}}^{*}(\Phi)\oplus C_{\mathrm{def}}^{*+1}(\calK)\oplus C_{\mathrm{def}}^{*}(\Psi),$$
with differential
$\delta(a,b,c)=(\Phi_{*}b-\delta_{\Phi}a,\; \delta_{\calK}b,\; \Psi_{*}b-\delta_{\Psi}a)$. 

\begin{remark}
Notice that this complex can also be defined for any pair of morphisms set as in diagram \eqref{diagr:fraction}. However the fact that in a fraction the left leg is a Morita map can be used to get an alternative expression of the deformation complex useful for computations.% This will be explored in future work.... 
\end{remark}

Equivalent fractions have isomorphic deformation cohomology as can be proven by using the quasi-isomorphisms $F^{*}$ and $F_{*}$ of deformation complexes (Propositions \ref{Mor1} and \ref{Mor2}), induced by a Morita map $F$ which relates two fractions, and by the isomorphism of deformation complexes of isomorphic morphisms (Theorem \ref{Theor:IsomorpMorphisms-Cohomology}). Hence, the deformation complex of a fraction induces a well-defined \textbf{deformation cohomology for generalized morphisms}. We register that fact in the following theorem.

\begin{theorem}\label{thm:cohomologyGeneraliz.maps}
If $\Psi/\Phi$ and $\Psi'/\Phi'$ are equivalent fractions from $\H$ to $\G$, then their deformation cohomologies $H^{*}_{def}(\Psi/\Phi)$ and $H^{*}_{def}(\Psi'/\Phi')$ are isomorphic.
\end{theorem}

Thus, the deformation cohomology of a fraction is an algebraic object associated to the stacky map it represents. %It is invariant by equivalence of fractions, so it is associated to equivalence class of fractions, that is, it is associated to generalized morphisms of Lie groupoids. Or in other words, it is associated to maps of differentiable stacks.
The deformation cohomology $H_{def}^{*}(\Psi/\Phi)$ also turns out to be very involved in the infinitesimal study of the space of generalized maps. For instance, every deformation of a fraction $\Psi/\Phi$ has a corresponding 1-cocycle whose cohomology class should be regarded as the velocity vector at (the class of) $\Psi/\Phi$ of the associated path of generalized morphisms. More extended and detailed results will lie on future work.\\

%%%%%%%%%%%%%%%%%%%%%%%%%%%%%% SECTION %%%%%%%%%%%%%%%%%%%%%%%%%%%%%%%%%%%%%%%%%%%

\section{Application: some remarks on deformations of multiplicative forms}\label{Section:multiplicativeforms}

In this section we use the deformation complex of morphisms to study deformations of multiplicative forms on Lie groupoids. Also, we characterize \textit{trivial} defomations of multiplicative forms in cohomological terms. The content of this section is also relevant to develope the theory of deformations of symplectic groupoids as in \cite{CardMS}.

Recall that a $k$-form $\omega\in\Omega^k(\G)$ is called a \textbf{multiplicative $k$-form} if it satisfies the
%Let $\calG$ be a Lie groupoid and $\omega\in\Omega^k(\G)$ be a multiplicative $k$-form on $\G$. That is, $\omega$ satisfies the
\textit{multiplicativity condition}
\begin{equation}\label{multiplicat.condition}
m^{*}\omega=pr_1^{*}\omega+pr_2^{*}\omega,
\end{equation}
where $pr_i,\ m:\calG^{(2)}\rightarrow\calG$ are the canonical projections and multiplication of $\calG$. A map \newline $\overline{\omega}:\bigoplus^{k} T\G\times I\to\mathbb{R}$ is called a \textbf{smooth family of multiplicative $k$-forms} if, for every $\e$, $\omega_\e:=\overline{\omega}(\cdot,\e)$ is a multiplicative $k$-form. We say that a smooth family $\omega_\e\in\Omega^k(\G)$ of multiplicative $k$-forms is a \textbf{deformation of} $\omega$ if $\omega_0=\omega$.

We consider first the particular case of multiplicative symplectic 2-forms on $\G$ (Proposition \ref{Prop: 2-forms} below), then we will generalize the situation to multiplicative $k$-forms. A more advanced study of the case of multiplicative symplectic 2-forms is made in \cite{CardMS} where we consider also a simultaneous deformation of the underlying Lie groupoid and a relation of its deformation cohomology to the Bott-Shulmann-Stasheff complex. Recall that the classical Moser's theorem of symplectic geometry deals with symplectic 2-forms on a differentiable manifold. This theorem says that a smooth family $\omega_\e$ of symplectic forms on a manifold $M$ is recovered as the pullback $F_\e^{*}\omega_0$ by a family of diffeomorphisms $F_\e$ of the symplectic form at time zero if, and only if, there exists a smooth family $\alpha_\e\in\Omega^{1}(M)$ of 1-forms such that
\begin{equation}\label{eq:Moser}
\frac{d}{d\e}\omega_\e=d_{dR}\alpha_\e,\ \text{for every } \e.
\end{equation}

The following proposition formulates an analogous result in the context of Lie groupoids, where instead we consider the \emph{multiplicative de Rham complex} $(\Omega_{mult}^{*}(\G), d_{dR})$ of $\G$, whose elements are 
%which consists of the complex of
multiplicative forms of $\G$.

%In this way, the direct fashion to express an analogous result in the context of Lie groupoids is considering the multiplicative de Rham complex $(\Omega_{mult}^{*}(\G), d_{dR})$ of $\G$, which consists of the multiplicative forms of $\G$. Explicitely, one has

\begin{proposition}\label{Prop: 2-forms}
Let $(\G,\omega)$ be a compact symplectic groupoid, and assume that $\omega_\e$ is a deformation of $\omega$. Then, $\omega_\e=F_\e^{*}\omega$ for a smooth family of groupoid automorphisms of $\G$, with $F_0=Id_\G$, if and only if the family of cocycles $\frac{d}{d\e}\omega_\e\in\Omega^{2}_{mult}(\G)$ is smoothly exact in $(\Omega_{mult}^{*}(\G), d_{dR})$.
\end{proposition}
\begin{proof}
This proof is just a multiplicative version of the proof of the classical Moser theorem. In order to prove it, we notice that the multiplicative symplectic 2-form $\omega$ yields an isomorphism between the space of multiplicative 1-forms $\Omega^{1}_{mult}(\G)$ and the space of multiplicative vector fields $\mathfrak{X}_{mult}(\G)\cong Z^{1}_{def}(\G)$. Therefore the time dependent flow $F_\e$ of the transgressing family $X_\e$ of vector fields will be given by a family of automorphisms of $\G$ starting at the identity $Id_\G$. %Noting that $Z^{1}_{def}(\G)\cong\Omega^{1}_{mult}(\G)$.....
\end{proof}

\begin{remark}
We can express equivalently the smooth exactness of $\frac{d}{d\e}\omega_\e\in\Omega^{2}_{mult}(\G)$ by saying that $\frac{d}{d\e}\omega_\e$ has a \textbf{smooth preimage by} $d_{dR}$ in $\Omega^{1}_{mult}(\G)$. That is, saying that there exists a smooth family $\alpha_\e\in\Omega^{1}_{mult}(M)$ such that
$\frac{d}{d\e}\omega_\e=d_{dR}\alpha_\e,$ for every $\e$, which is the same condition of equation \eqref{eq:Moser}.
\end{remark}

The following example shows that the classical Moser theorem is obtained as an application of the previous proposition to the pair groupoid.

\begin{example}
Let $\omega_\e$ be a smooth family of symplectic structures on a compact manifold $M$, and consider the induced family of symplectic groupoids $(\mathrm{Pair}(M),\tilde{\omega}_\e)$, where $\tilde{\omega}_\e=pr_1^{*}\omega_\e-pr_2^{*}\omega_\e$. Thus, since $\tilde{\omega}_\e$ is smoothly exact in $\Omega_{mult}^{*}(\mathrm{Pair}(M))$ if and only if $\omega_\e$ is smoothly exact in $\Omega^{*}(M)$, and the automorphisms of $\mathrm{Pair}(M)$ are of the form $F=f\times f$, where $f\in\mathrm{Diff}(M)$, it follows that the previous proposition reduces to the classical Moser's theorem in this case.
\end{example}

Next we pass to the general case of multiplicative $k$-forms on the groupoid $\G$. A multiplicative $k$-form $\omega$ can also be viewed as the morphism of Lie groupoids

\[\xymatrix{\bigoplus_{p_{\calG}}^{k}T\calG \ar[r]^{\hat{\omega}} \ar@<0.25pc>[d] \ar@<-0.25pc>[d] & \mathbb{R} \ar@<0.25pc>[d] \ar@<-0.25pc>[d]\\
\bigoplus_{p_{M}}^{k}TM \ar[r] & {*}}\]
which is $k$-linear and skewsymmetric in the sense that $\hat{\omega}:\bigoplus_{p_{\calG}}^{k}T\calG\to\mathbb{R}$ is $k$-linear with respect to the linear structure of $\bigoplus_{p_{\calG}}^{k}T\G$ over $\calG$. 
%A $k$-form $\omega\in\Omega^k(\G)$ is said to be \textbf{multiplicative} if the skew-symmetric and $k$-linear map $\omega^{b}:\bigoplus^{k}T\G\to \mathbb{R}$, determined by it, is a morphism of Lie groupoids.
With this viewpoint, a smooth family of skew-symmetric and $k$-linear Lie groupoid morphisms $\hat{\omega}_\e:\bigoplus^{k}T\G\to \mathbb{R}$ is called a \textbf{deformation} of the multiplicative $k$-form $\omega_0\in\Omega^{k}(\G)$.

%Thus, a deformation $\omega_\e\in\Omega_{mult}^{k}(\calG)$ of $\omega$ amounts to a deformation $\Phi_\e$ of the morphism $\hat{\omega}$ such that, for every $\e$, $\Phi_\e$ is $k$-linear and skewsymmetric. Explicitely, we will have $\Phi_\e=\hat{\omega}_\e$ for every $\e$.

From the cohomological perspective, the skew-symmetry and $k$-linearity of the morphism $\hat{\omega}$ translate to \emph{skew-symmetric and $k$-linear deformation cochains} in $C^{*}_{def}(\hat{\omega})$ as described below. 

\begin{definition}
The \textbf{deformation complex of a multiplicative $k$-form} $\omega\in\Omega^{k}_{mult}(\G)$ consists of the skew-symmetric and $k$-linear cochains of $C^{\bullet}_{def}(\hat{\omega})$. We denote by $C^{*}_{def}(\omega)$ such a subcomplex of deformation cochains.
\end{definition}

Explicitly, the deformation complex $C^{\bullet}_{def}(\omega)$ can be described as follows. Consider the natural identification $\bigoplus_{p_{\calG^{(l)}}}^{k}T\calG^{(l)}\cong(\bigoplus_{p_\calG}^{k}T\calG)^{(l)}$, where $p_{\G^{(l)}}:T\calG^{(l)}\to\G^{(l)}$ and $p_{\G}:T\calG\to\G$ are the projections of the respective tangent bundles. Thus, the elements $c$ of $C^{*}_{def}(\omega)$ are given by those elements $c$ in $C^{*}_{def}(\hat{\omega})$ which turn the composition
\begin{equation}\label{eq:naturalidentifications}
\bigoplus_{p_{\calG^{(l)}}}^{k}T\calG^{(l)}\cong(\bigoplus_{p_\calG}^{k}T\calG)^{(l)}\stackrel{c}{\longrightarrow} T\mathbb{R}\cong\mathbb{R}\oplus\mathbb{R}\stackrel{pr_2}{\longrightarrow}\mathbb{R}
\end{equation}
a $k$-linear and skew-symmetric map. Along with the restriction $\delta_\omega$ of the deformation differential $\delta_{\hat{\omega}}$ of $C^{\bullet}_{def}(\hat{\omega})$, $(C^{\bullet}_{def}(\omega),\delta_\omega)$ becomes indeed a subcomplex: the \emph{deformation complex of} $\omega$.

\begin{remark}\label{rmk:differsubcomplex}
Considering the identification of $(\bigoplus_{p_\calG}^{k}T\calG)^{(l)}$ with $\bigoplus_{p_{\calG^{(l)}}}^{k}T\calG^{(l)}$ and the composition
$$(\bigoplus_{p_\calG}^{k}T\calG)^{(l)}\stackrel{c}{\longrightarrow} T\mathbb{R}\cong\mathbb{R}\oplus\mathbb{R}\stackrel{pr_2}{\longrightarrow}\mathbb{R},$$
the elements of the complex $C^{\bullet}_{def}(\omega)$ can be regarded as the elements of the subcomplex $C^{\bullet}_{k\mathrm{-lin},\mathrm{sk}}(\bigoplus^{k} T\calG)$ of $C^{\bullet}_{\mathrm{diff}}(\bigoplus^{k}T\calG)$ consisting of \emph{fiberwise $k$-linear and skew-symmetric differentiable cochains} of $C^{\bullet}_{\mathrm{diff}}(\bigoplus^{k}T\calG)$. Moreover, from a straightforward computation, one also gets the correspondence between the differentials of these complexes.
\end{remark}

The following proposition shows that this deformation complex is isomorphic to one that only depends on the simplicial structure of $\G$ and not of $T\G$ or $\omega$ explicitly.

\begin{proposition}
Given any multiplicative $k$-form $\omega\in\Omega^{k}_{mult}(\G)$, the deformation complex $(C^{\bullet}_{def}(\omega),\delta_\omega)$ is isomorphic to $(\Omega^{k}(\G^{(\bullet)}),\delta)$; where $\delta$ is the differential induced from the simplicial structure of $\G$.
\end{proposition}
\begin{proof}
The correspondence between the cochains is straightforward from the description of the elements of $C^{\bullet}_{def}(\omega)$ in expression \eqref{eq:naturalidentifications}. And the correspondence between the differentials follows directly from the identification of the deformation differential $\delta_\omega$ with the simplicial differential of the differentiable subcomplex $C^{\bullet}_{k-lin,\ sk}(\bigoplus^{k}T\G)$ of Remark \ref{rmk:differsubcomplex} above, which in turn identifies with the simplicial differential of the simplicial complex of $k$-forms over the nerve of $\G$.

\end{proof}

With this setting, a deformation $\hat{\omega}_\e$ of $\omega$ by multiplicative forms on $\calG$ determines:
\begin{enumerate}
  \item a smooth family $X_\e=\frac{d}{d\e}\hat{\omega}_\e$ of 1-cocycles in $C^{\bullet}_{k\mathrm{-lin},\ \mathrm{sk}}(\oplus^{k}T\calG)\cong C^{\bullet}_{def}(\omega)$ (Proposition \ref{cociclomorph.});\\
	
  \item a family of cochain maps
$$(\hat{\omega}_\e)_{*}:C_{def}^{\bullet}(\bigoplus^{k} T\calG)\to C^{\bullet}_{def}(\omega)\cong\Omega^{k}(\G^{(\bullet)}),$$
\noindent which follows from Remark \ref{chain maps}.\\

\end{enumerate}

Notice, however, that the cocycle condition $(1)$ also follows by taking directly the derivative $\frac{d}{d\e}$ of the multiplicativity condition of $\omega_\e$ (equation \eqref{multiplicat.condition}), obtaining that $\frac{d}{d\e}\omega_\e\in\Omega^{k}_{mult}(\calG)\cong Z^{1}_{def}(\omega)$.\\

We define now a map of complexes $\mathcal{T}:C_{def}^{*}(\calG)\to C_{def}^{*}(\bigoplus^{k}T\calG)$ which, composing with the map $(\hat{\omega}_{\e})_{*}$ above, is a relevant element in the statement of the theorem below that generalizes Proposition \ref{Prop: 2-forms} to $k$-forms. Consider the \emph{tangent lift} $T:C^{*}_{def}(\calG)\longrightarrow C^{*}_{def}(T\calG)$ of deformation cochains which, as checked in \cite{deformationsofVBgroupoids}, turns out to be a cochain map. For $c\in C^{k}_{def}(\calG)$ it is defined by
$$Tc:=J_\calG\circ dc,$$
where $J_\calG:T(T\calG)\rightarrow T(T\calG)$ is the involution map of the double tangent bundle of $\calG$. %(Appendix *** shows that $T$ is a chain map).
We thus define $\mathcal{T}$ as the map $\oplus^{k}T$. Explicitly,
$$\oplus^{k}T:C^{*}_{def}(\calG)\rightarrow C^{*}_{def}(\bigoplus^{k}T\calG);\ \ \oplus^{k}T(c):=\oplus^{k}(Tc).$$

\begin{theorem}\label{Theorem: k-forms}
Let $\omega_\e\in\Omega^{k}(\calG)$ be a deformation of the multiplicative $k$-form $\omega\in\Omega^{k}(\calG)$. Assume that $\G$ is compact. Then, $\omega_\e=\Phi_\e^{*}\omega$ for a smooth family $\Phi_\e$ of groupoid automorphisms of $\G$, with $\Phi_0=Id_\G$, if and only if the family of multiplicative $k$-forms $X_\e:=\frac{d}{d\e}\omega_\e\in\Omega^{k}_{mult}(\G)$ has a smooth pre-image in $Z_{def}^{1}(\G)$ by the map $(\hat{\omega}_\e)_{*}\circ\mathcal{T}:Z_{def}^{1}(\G)\to Z^{1}(\Omega^{k}(\G^{(\bullet)}))=\Omega^{k}_{mult}(\G)$.
\end{theorem}

\begin{proof}%[\textbf{Proof} of the Theorem \ref{Theorem: k-forms}]

The smoothness of the pre-images of $-X_\l$ implies
$$-X_\l=(\omega_\l)_{*}\circ\oplus^{k} T(Z_\l),$$
for some smooth family $Z_\e$ of deformation 1-cocycles (i.e. multiplicative vector fields) of $\calG$. Thus, let denote by $F_\e$ the flow of the time-dependent vector field $\{Z_\e\}_\e$ (starting at time zero) covering $f_\e$, then we have

\begin{align*}
-\left.\frac{d}{d\e}\right|_{\e=\l}\omega_\e&=(d\omega_\l)\left(\oplus^{k} T(\left.\frac{d}{d\e}\right|_{\e=\l}(F_\e\circ F_\l^{-1}))\right)\\
&=(d\omega_\l)\left(\left.\frac{d}{d\e}\right|_{\e=\l} \left(\oplus^{k} TF_\e \circ \oplus^{k} T(F_\l^{-1})\right)\right).
\end{align*}
Equivalently,
$$-\left.\frac{d}{d\e}\right|_{\e=\l}\omega_\e\circ\oplus^{k} TF_\l=(d\omega_\l)\left(\left.\frac{d}{d\e}\right|_{\e=\l}\oplus^{k} TF_\e\right).$$
In other words,

\begin{align*}
\left.\frac{d}{d\e}\right|_{\e=\l}(F_\e^{*}\omega_\e)&=F_\l^{*}(\left.\frac{d}{d\e}\right|_{\e=\l}\omega_\e)+\left.\frac{d}{d\e}\right|_{\e=\l}F_\e^{*}\omega_\l\\
&=0.
\end{align*}
%where $\frac{d}{d\e}\bar{\alpha}_\e:=f_\e^{*}\tilde{\alpha_\e}$.
%Therefore, $F_\e^{*}\omega_\e=\omega_0$, where $\widehat{\alpha}_\e=\bar{\alpha}_0-\bar{\alpha}_\e$.

This says that
$$\omega_\e=\Phi_\e^{*}\omega,$$
where $\Phi_\e=F_{\e}^{-1}$ is a smooth family of automorphisms of $\calG$ due to the multiplicativity of the time-dependent vector field $Z_\e$.

\end{proof}

\begin{remark}
The previous Theorem reduces to the Proposition \ref{Prop: 2-forms} when $\omega_\e$ is taken as a family of multiplicative and symplectic 2-forms on $\calG$.
\end{remark}

The following theorem now tells us about the meaning of the smooth cohomological triviality of the family of 1-cocycles $X_\e$ in terms of the deformation $\omega_\e$. Its proof is similar to the previous one. This fact will be approached a bit more for the particular case of symplectic groupoids in \cite{CardMS}.

\begin{theorem}
Let $\omega_\e\in\Omega^{k}(\calG)$ be a deformation of the multiplicative $k$-form $\omega\in\Omega^{k}(\calG)$. Assume that the groupoid $\G\tto M$ is compact.
Then, smooth exactness of the familiy of cocycles $X_\e$ amounts to the fact that the deformation $\omega_\e$ is of the form $\omega_\e=\Phi_\e^{*}\omega+s^{*}\beta_\e-t^{*}\beta_\e$; where $\Phi_\e$ is a smooth family of groupoid automorphisms of $\G$, with $\Phi_0=Id_\G$ and $\beta_\e$ is a smooth family of 2-forms on $M$ with $\beta_0=0$.
\end{theorem}

%\backmatter \singlespacing   % espaçamento simples
%\bibliographystyle{plainnat} % citação bibliográfica textual
\bibliographystyle{acm} % citação bibliográfica numerica
\bibliography{bibliografia}  % associado ao arquivo: 'bibliografia.bib'

\Addresses

\end{document}